\documentclass[a4paper,notitlepage,thmsa,french,12pt]{article}%
\usepackage{amssymb}
\usepackage{theorem}
\usepackage{doublespace}
\usepackage{euscript}
\usepackage{amsmath}
\usepackage{graphicx}
\usepackage{babel}
\usepackage{amsfonts}
\usepackage{hyperref}%
\providecommand{\U}[1]{\protect\rule{.1in}{.1in}}
\renewcommand\baselinestretch{1.3}
\newtheorem{theorem}{Théorème}

\newtheorem{corollary}[theorem]{Corollaire}

\newtheorem{definition}[theorem]{Définition}

\newtheorem{lemma}[theorem]{Lemme}

\newtheorem{proposition}[theorem]{Proposition}

\newenvironment{proof}[1][Preuve]{\textbf{#1.} }{\ \rule{0.5em}{0.5em}}
\def\QTR#1{\csname #1\endcsname}
\begin{document}

\title{Probl\`{e}me inverse pour la conductivit\'{e} en dimension deux}
\author{Vincent Michel\\Universit\'{e} Pierre et Marie Curie, 4 place Jussieu, 75252 Paris cedex}
\date{17 juin 2016;
\today
}
\maketitle

\begin{abstract}%
\begin{spacing}{1}%
\medskip Cet article propose un proc\'{e}d\'{e} de reconstruction d'une
surface de Riemann \`{a} bord coupl\'{e} \`{a} un tenseur de conductivit\'{e}
\`{a} partir de son bord et de l'op\'{e}rateur de DirichletNeumann associ\'{e}
\`{a} cette conductivit\'{e}. Lorsque la donn\'{e}e de d\'{e}part provient
d'une surface riemannienne r\'{e}elle de dimension deux \'{e}quip\'{e}e d'un
tenseur de conductivit\'{e}, ce proc\'{e}d\'{e} restitue l'int\'{e}gralit\'{e}
de ce qui peut \^{e}tre d\'{e}termin\'{e} \`{a} partir de ces
donn\'{e}es.\medskip\medskip

\hspace{15.2263pc}\textbf{Abstact}

This article proposes a process to reconstruct a Riemann surface with boundary
equipped with a linked conductivity tensor from its boundary and the
Dirichlet-Neumann operator associated to this conductivity. When initial data
comes from a two dimensional real Riemannian surface equipped with a
conductivity tensor, this process recover the whole of what can be determined
from this data.\medskip\medskip

\hspace{15.2834pc}\textbf{D\'{e}dicace}

En janvier 2016 d\'{e}c\'{e}dait mon ami Gennadi Henkin avec qui j'ai
travaill\'{e} pendant plus de quinze ans. Ce travail qui revient sur un sujet
qu'il avait apport\'{e} lui est d\'{e}di\'{e}. Par les nombreuses citations
d'articles o\`{u} il figure en tant qu'auteur, le lecteur pourra constater la
profondeur de la pens\'{e}e math\'{e}matique de Gennadi.\medskip

\hspace{15.3407pc}---------------

\noindent\textit{Mots cl\'{e}s: }surface de Riemann, probl\`{e}me de
Dirichlet-Neumann, fonction de Green, conductivit\'{e}, onde de choc, plongement.


\noindent\textit{Classification AMS: }Primaire, 32c25, 32d15, 32v15, 35r30,
58j32 secondaire 35r30

%

\end{spacing}%

\newpage

\end{abstract}
\tableofcontents%

\begin{spacing}{1}%
\medskip

\noindent Cet article est organis\'{e} de la mani\`{e}re suivante. La
section~\ref{S/ SC} propose une formalisation du probl\`{e}me d'imp\'{e}dance
tomographique \'{e}lectrique inverse qui, si elle n'est pas nouvelle, n'est
pas vraiment explicite dans la litt\'{e}rature. La
section~\ref{S/ SurfacesNodales} \'{e}nonce ce que nous avons besoin de savoir
sur les surfaces nodales qui apparaissent naturellement dans ce probl\`{e}me.

La troisi\`{e}me section regroupe des th\'{e}or\`{e}mes in\'{e}dits et
d'autres d\'{e}j\`{a} connus afin de dresser le panorama de l'unicit\'{e} et
de la construction d'une solution au probl\`{e}me formalis\'{e} dans la
section~\ref{S/ PbDNI}.

La section~\ref{S/ ReconsCondiso} aborde le probl\`{e}me de la construction
effective d'une solution dans le cas cl\'{e} de la reconstruction d'une
surface de Riemann \`{a} bord \`{a} partir de son op\'{e}rateur de
Dirichlet-Neumann. La m\'{e}thode propos\'{e}e est \`{a} notre connaissance
enti\`{e}rement nouvelle et repose sur l'analyse a priori de la
d\'{e}composition d'une fonction holomorphe de deux variables en somme d'onde
de chocs, c'est \`{a} dire de solutions de l'\'{e}quation diff\'{e}rentielle
$\frac{\partial h}{\partial y}=h\frac{\partial h}{\partial x}$.

\section{Structures de conductivit\'{e}\label{S/ SC}}

Nous d\'{e}finissons une structure de conductivit\'{e} de dimension deux comme
un couple $\left(  M,\sigma\right)  $ o\`{u} $M$ est une surface r\'{e}elle
\`{a} bord$^{(}$\footnote{On consid\`{e}re qu'une surface \`{a} bord $M$ est
un ouvert dense d'une vari\'{e}t\'{e} r\'{e}elle orient\'{e}e \`{a} bord de
dimension deux $\overline{M}$ dont toutes les composantes connexes sont
bord\'{e}es par des vari\'{e}t\'{e}s r\'{e}elles de dimension pure $1$~;
ainsi, le bord $bM$ de $M$ est $\overline{M}\backslash M$ . Une surface de
Riemann \`{a} bord est une vari\'{e}t\'{e} complexe connexe de dimension $1$
qui est aussi une surface \`{a} bord r\'{e}elle.}$^{)}$ connexe orient\'{e}e
munie d'une conductivit\'{e} $\sigma:T^{\ast}\overline{M}\rightarrow T^{\ast
}\overline{M}$, c'est-\`{a}-dire d'un tenseur tel que pour tout $p\in
\overline{M}$,
\[
\Sigma_{p}:T_{p}^{\ast}\overline{M}\times T_{p}^{\ast}M\ni\sigma_{p}\left(
a,b\right)  \mapsto\frac{a\wedge\sigma_{p}\left(  b\right)  }{\mu_{p}}%
\]
est une forme bilin\'{e}aire d\'{e}finie positive, $\mu$ \'{e}tant une forme
volume $\mu$ pour $\overline{M}$ fix\'{e}e une fois pour toute. Cette
d\'{e}finition, peut \^{e}tre inhabituelle, n'est qu'une reformulation$^{(}%
$\footnote{Si on fixe un point $p$ dans $\overline{M}$, des coordonn\'{e}es
$\left(  x,y\right)  $ au voisinage de $p$ et qu'on pose comme
dans~\cite{LeJ-UhG1989} $\left(  \xi,\eta\right)  =\left(  dy,-dx\right)  $
puis $\sigma\left(  dx\right)  =r\xi+t\eta$ et $\sigma\left(  dy\right)
=u\xi+s\eta$, pour $a=a_{x}dx+a_{y}dy$ et $b=b_{x}dx+b_{y}dy$ dans
$T_{p}\overline{M}$, $\sigma_{p}\left(  b\right)  =\left(  b_{x}%
r+b_{y}u\right)  \xi+\left(  b_{x}t+b_{y}s\right)  \eta$ et
\begin{align*}
a\wedge\sigma_{p}\left(  b\right)   &  =\left(  a_{x}dx+a_{y}dy\right)
\wedge\left[  \left(  b_{x}r+b_{y}u\right)  dy-\left(  b_{x}t+b_{y}s\right)
dx\right] \\
&  =\left(  ra_{x}b_{x}+ua_{x}b_{y}+ta_{y}b_{x}+sb_{x}b_{y}\right)  dx\wedge
dy
\end{align*}
\par
{}}$^{)}$ intrins\`{e}que de celle donn\'{e}e par~\cite{LeJ-UhG1989}
et~\cite{SyJ1990}. La section~\ref{S/ PbDNI} montre comment affranchir cette
d\'{e}finition du choix de cette forme volume auxiliaire.

Pour toute fonction continue $u:bM\rightarrow\mathbb{R}$, il existe alors dans
$C^{0}\left(  \overline{M},\mathbb{R}\right)  $ une unique fonction
$E_{\sigma}u$ qui r\'{e}sout le probl\`{e}me de Dirichlet suivant%
\begin{equation}
d\sigma dU=0~\&~\left.  U\right\vert _{bM}=u. \label{F/ dsdu=0}%
\end{equation}
Dans le probl\`{e}me d'imp\'{e}dance tomographique \'{e}lectrique inverse qui
est le sujet de cet article, il est courant de consid\'{e}rer que la
donn\'{e}e essentielle est l'op\'{e}rateur de Dirichlet-Neumann associ\'{e}
\`{a} (\ref{F/ dsdu=0}). Dans son acceptation usuelle, celui-ci fait
intervenir une d\'{e}riv\'{e}e normale et donc une m\'{e}trique. Lorsque $M$
est une surface r\'{e}elle contenue dans $\mathbb{R}^{3}$, on peut utiliser
celle induite sur $\overline{M}$ par celle standard de $\mathbb{R}^{3}$ bien
que cette derni\`{e}re ne soit pas forc\'{e}ment corr\'{e}l\'{e}e \`{a}
$\sigma$. Avant de donner d\'{e}finition intrins\`{e}que \`{a} $\sigma$ de cet
op\'{e}rateur, la section suivante donne une analyse a priori, souvent en
filigrane dans la litt\'{e}rature, des structures de conductivit\'{e} de
dimension 2.

Certains auteurs pr\'{e}f\`{e}rent consid\'{e}rer une m\'{e}trique
Riemannienne $g$ sur $M$ et construire l'op\'{e}rateur de Dirichlet-Neumann
\`{a} partir de la solution du probl\`{e}me de Dirichlet $\left(  \Delta
_{g}U=0~\&~\left.  U\right\vert _{bM}=u\right)  $ o\`{u} $\Delta_{g}$ est
l'op\'{e}rateur de Laplace-Beltrami associ\'{e} \`{a} $g$. En \'{e}crivant les
\'{e}quations $d\sigma dU=0$ et $\Delta_{g}U=0$ en coordonn\'{e}es, on
constante que ces deux formulations ne sont \'{e}quivalentes que dans le cas
particulier o\`{u} $\det\sigma_{p}=1$ pour tout $p\in\overline{M}$. La
formule~(\ref{F/ dec sigma}) de la section~(\ref{S/ fact}) montre que $\sigma$
se r\'{e}duit alors \`{a} une structure complexe et que dans le probl\`{e}me
de tomographie inverse, c'est alors la seule chose \`{a} reconstruire.
reconstruire autre chose que cette structure conforme.

\subsection{Factorisation d'une conductivit\'{e}\label{S/ fact}}

Les conditions impos\'{e}es \`{a} $\sigma$ pour \^{e}tre une conductivit\'{e}
indiquent l'implication d'une m\'{e}trique naturelle. Il est remarqu\'{e}
dans~\cite{HeG-MiV2007} qu'une forme volume $\mu$ pour $\overline{M}$ ayant
\'{e}t\'{e}, on d\'{e}finit une m\'{e}trique naturelle $g_{\mu,\sigma}$ sur
$\overline{M}$ en posant pour tout $t\in T\overline{M}$%
\begin{equation}
g_{\mu,\sigma}\left(  t\right)  =\frac{\sigma^{-1}\left(  t\,\lrcorner
\,\mu\right)  \wedge\left(  t\,\lrcorner\,\mu\right)  }{\mu}.
\label{F/ gmusigma}%
\end{equation}
Sa classe conforme $\mathcal{C}_{\sigma}$ ne d\'{e}pend pas de $\mu$ et
$\sigma$ se factorise (voir \cite{HeG-MiV2007}) \`{a} travers $\mathcal{C}%
_{\sigma}$ au sens o\`{u} il existe une fonction $s_{\sigma}:\overline
{M}\rightarrow\mathbb{R}_{+}^{\ast}$ de m\^{e}me r\'{e}gularit\'{e} que
$\sigma$, appel\'{e}e coefficient de conductivit\'{e} dans cet article, telle
que lorsque $\left(  x_{1},x_{2}\right)  $ est un couple de coordonn\'{e}es
isothermiques locales pour $\mathcal{C}_{\sigma}$,
\begin{equation}
Mat_{dx}^{\left(  dx_{2},-dx_{1}\right)  }(\sigma_{p})=s_{\sigma}\left(
p\right)  I_{2} \label{F/ diag}%
\end{equation}
pour tout $p$ dans l'ouvert de $\overline{M}$ o\`{u} $\left(  x_{1}%
,x_{2}\right)  $ est d\'{e}fini, $I_{2}$ \'{e}tant la matrice unit\'{e}
d'ordre $2$ et $dx=\left(  dx_{1},dx_{2}\right)  $. Notons $\det\sigma$
l'application qui \`{a} tout point $p$ de $\overline{M}$ associe le
d\'{e}terminant de l'application lin\'{e}aire $\sigma_{p}$. Alors
(\ref{F/ diag}) entra\^{\i}ne que $\det\sigma=s_{\sigma}^{2}$,
c'est-\`{a}-dire $s_{\sigma}=\sqrt{\det\sigma}$. Si on note $c_{\sigma}$ la
conductivit\'{e} d\'{e}finie par%
\begin{equation}
\sigma=s_{\sigma}.c_{\sigma}=\sqrt{\det\sigma}.c_{\sigma},
\label{F/ dec sigma}%
\end{equation}
$\mathcal{C}_{\sigma}$ est aussi la classe conforme associ\'{e}e \`{a}
$c_{\sigma}$ et quand $\left(  x_{1},x_{2}\right)  $ est un couple de
coordonn\'{e}es isothermiques locales pour $\mathcal{C}_{\sigma}$,
\[
Mat_{dx}^{dx}(c_{\sigma})=\left(
\begin{array}
[c]{cc}%
0 & -1\\
1 & 0
\end{array}
\right)  \overset{d\acute{e}f}{=}J.
\]
Autrement dit, $c_{\sigma}$ est aussi l'op\'{e}rateur de conjugaison agissant
sur les 1-formes de $M$. De plus, si $d^{\sigma}=c_{\sigma}d$, $\overline
{\partial}{}^{\sigma}=\frac{1}{2}\left(  d-id^{\sigma}\right)  $ est
l'op\'{e}rateur de Cauchy-Riemann associ\'{e} \`{a} $\mathcal{C}_{\sigma}$ et
\[
d\sigma dU=ds_{\sigma}d^{\sigma}U
\]
pour toute fonction $U\in C^{2}\left(  \overline{M}\right)  $.

Supposons que $\mathcal{C}$ soit une structure complexe sur $\overline{M}$,
c'est-\`{a}-dire la donn\'{e}e d'un atlas de $\overline{M}$ qui fasse de $M$
une surface de Riemann \`{a} bord de bord $bM$. Si $x_{1}$ et $x_{2}$ sont les
parties r\'{e}elles et imaginaires d'une m\^{e}me coordonn\'{e}e holomorphe de
$M$, les matrices jacobiennes relatives \`{a} $\left(  x_{1},x_{2}\right)  $
d'applications holomorphes commutent avec $J$. Cela signifie qu'on peut
d\'{e}finir un tenseur $c:T\overline{M}\rightarrow T\overline{M}$ par le fait
que dans de telles coordonn\'{e}es, $Mat_{dx}^{dx}(c)=J$. Par construction,
$c$ est une conductivit\'{e} de coefficient constant $1$, $c\circ d=i\left(
\overline{\partial}-\partial\right)  \overset{d\acute{e}f}{=}d^{c}$ et $c$ est
son op\'{e}rateur de Hodge agissant sur les 1-formes quand $M$ est muni de la
m\'{e}trique duale de celle d\'{e}finie sur chaque $T_{p}^{\ast}\overline{M}$
par $\left\langle a,b\right\rangle \mu=a\wedge\ast b=\frac{1}{\sqrt{\det
\sigma}}a\wedge\sigma\left(  b\right)  $.

La d\'{e}composition (\ref{F/ dec sigma}) fait donc appara\^{\i}tre une
structure complexe naturellement associ\'{e}e \`{a} $\sigma$. Celle-ci est
unique au sens o\`{u} si $c^{\prime}$ est l'op\'{e}rateur de conjugaison
associ\'{e} \`{a} une structure complexe $\mathcal{C}^{\prime}$ et si
$s^{\prime}\in\left(  \mathbb{R}_{+}^{\ast}\right)  ^{M^{\prime}}$,
l'\'{e}galit\'{e} $\sigma=s^{\prime}.c^{\prime}$ impose $s=s^{\prime}$ puis
$c_{\sigma}=c^{\prime}$ car $\det c_{\sigma}=1=\det c^{\prime}$.

La formule (\ref{F/ diag}) montre que pour tout $p\in M$, $\sigma_{p}$ commute
avec les automorphismes orthogonaux de $\left(  T_{p}M,\left(  g_{\mu,\sigma
}\right)  _{p}\right)  $. Lorsque $M$ est une sous-vari\'{e}t\'{e} de
$\mathbb{R}^{3}$, en particulier si $M$ est un domaine de $\mathbb{R}^{2}$, et
lorsque la classe conforme de $g_{\mu,\sigma}$ est induite par la m\'{e}trique
standard de $\mathbb{R}^{3}$, ceci signifie que $\sigma$ est isotrope au sens
usuel (voir \cite{LeJ-UhG1989} et \cite{SyJ1990} par exemple). La proposition
ci-dessous r\'{e}sume ce qui pr\'{e}c\`{e}de.

\begin{proposition}
Soit $M$ une surface \`{a} bord orient\'{e}e de dimension r\'{e}elle $2$. Une
structure complexe $\mathcal{C}$ sur $\overline{M}$ d\'{e}finit un tenseur de
conductivit\'{e} de coefficient \'{e}gal \`{a} $1$. Inversement, pour toute
conductivit\'{e} $\sigma$ sur $\overline{M}$, il existe une unique structure
complexe $\mathcal{C}_{\sigma}$ telle que $\sigma=\sqrt{\det\sigma}c_{\sigma}$
o\`{u} $c_{\sigma}$ est l'op\'{e}rateur de conjugaison associ\'{e} \`{a}
$\mathcal{C}_{\sigma}$.
\end{proposition}

Ainsi, il est naturel de dire qu'une fonction $f$ d'un ouvert $U$ de $M$ \`{a}
valeurs dans $\mathbb{C}$ est $\sigma$-holomorphe si $\overline{\partial}%
{}^{\sigma}f=0$, ou de fa\c{c}on \'{e}quivalente, lorsque pour toute carte
$z:V\rightarrow\mathbb{C}$ de l'atlas holomorphe de $\left(  M,\mathcal{C}%
_{\sigma}\right)  $, $f\circ z^{-1}$ est holomorphe sur $z^{-1}\left(
U\right)  $ au sens usuel.

Si $\left(  M^{\prime},\sigma^{\prime}\right)  $ est une autre structure de
conductivit\'{e} de dimension $2$, une fonction $f$ d'un ouvert $U$ de $M$
\`{a} valeurs dans $M^{\prime}$ est dite $\left(  \sigma,\sigma^{\prime
}\right)  $-analytique si pour toute carte holomorphe $z^{\prime}:V^{\prime
}\rightarrow\mathbb{C}$ de $\left(  M^{\prime},\mathcal{C}_{\sigma^{\prime}%
}\right)  $, $z^{\prime}\circ f$ est $\sigma$-holomorphe sur $f^{-1}\left(
V^{\prime}\right)  \cap U$, c'est-\`{a}-dire si $z^{\prime}\circ f\circ
z^{-1}$ est holomorphe sur $z^{-1}\left(  f^{-1}\left(  V^{\prime}\right)
\cap U\right)  $ au sens usuel pour toute carte holomorphe $z:V\rightarrow
\mathbb{C}$ de $\left(  M,\mathcal{C}_{\sigma}\right)  $. Cette
propri\'{e}t\'{e} se caract\'{e}rise aussi par le lemme suivant.

\begin{lemma}
\label{L/ sigma-analyticite}Soient $\left(  M,\sigma\right)  $ et $\left(
M^{\prime},\sigma^{\prime}\right)  $ des structures de conductivit\'{e} de
dimension $2$, $U$ un ouvert de $M$ et $f:U\rightarrow M^{\prime}$ une
application diff\'{e}rentiable. Alors $f$ est $\left(  \sigma,\sigma^{\prime
}\right)  $-analytique si et seulement si $\left(  {}^{t}Df\right)  \circ
c_{\sigma^{\prime}}=c_{\sigma}\circ\left(  {}^{t}Df\right)  $. Lorsque $f$
r\'{e}alise un diff\'{e}omorphisme $\varphi$ de $U$ sur $f\left(  U\right)  $,
$\varphi$ est $\left(  \sigma,\sigma^{\prime}\right)  $-analytique si et
seulement si $\varphi_{\ast}c_{\sigma}=c_{\sigma^{\prime}}$ et en particulier
si $\varphi_{\ast}\sigma=\sigma^{\prime}$.
\end{lemma}

\begin{proof}
Consid\'{e}rons des cartes holomorphes $z:V\rightarrow\mathbb{C}$ de $\left(
M,\mathcal{C}_{\sigma}\right)  $ et $z^{\prime}:V^{\prime}\rightarrow
\mathbb{C}$ de $\left(  M^{\prime},\mathcal{C}_{\sigma^{\prime}}\right)  $ et
posons $F=Mat_{\left(  dx,dy\right)  }^{\left(  dx^{\prime},dy^{\prime
}\right)  }\left(  Df\right)  $ o\`{u} $\left(  x,y\right)  =\left(
\operatorname{Re}z,\operatorname{Im}z\right)  $ et $\left(  x^{\prime
},y^{\prime}\right)  =\left(  \operatorname{Re}z^{\prime},\operatorname{Im}%
z^{\prime}\right)  $. Alors%
\begin{align*}
Mat_{\left(  dx^{\prime},dy^{\prime}\right)  }^{\left(  dx,dy\right)  }\left(
\left(  {}^{t}Df\right)  \circ c_{\sigma^{\prime}}\right)   &  =Mat_{\left(
dx^{\prime},dy^{\prime}\right)  }^{\left(  dx,dy\right)  }\left(  {}%
^{t}Df\right)  Mat_{\left(  dx^{\prime},dy^{\prime}\right)  }^{\left(
dx^{\prime},dy^{\prime}\right)  }\left(  c_{\sigma^{\prime}}\right)  ={}%
^{t}FJ\\
Mat_{\left(  dx^{\prime},dy^{\prime}\right)  }^{\left(  dx,dy\right)  }\left(
c_{\sigma}\circ\left(  {}^{t}Df\right)  \right)   &  =Mat_{\left(
dx,dy\right)  }^{\left(  dx,dy\right)  }\left(  c_{\sigma}\right)
Mat_{\left(  dx^{\prime},dy^{\prime}\right)  }^{\left(  dx,dy\right)  }\left(
{}^{t}Df\right)  =J{}^{t}F
\end{align*}
La relation $\left(  {}^{t}Df\right)  \circ c_{\sigma^{\prime}}=c_{\sigma
}\circ\left(  {}^{t}Df\right)  $ a donc lieu si et seulement si $JF=FJ$.
Traduite sur les entr\'{e}es de matrice, ceci est \'{e}quivalent au fait que
$\operatorname{Re}f$ et $\operatorname{Im}f$ v\'{e}rifient le syst\`{e}me des
\'{e}quations de Cauchy-Riemann, c'est-\`{a}-dire $\frac{\partial f}%
{\partial\overline{z}}=0$.

Supposons maintenant que $\varphi=f\left\vert _{U}^{f\left(  U\right)
}\right.  $ est diff\'{e}omorphisme. Puisque par d\'{e}finition,
$\varphi_{\ast}c_{\sigma}=\left(  {}^{t}Df\right)  _{\psi}^{-1}\circ\left(
c_{\sigma}\right)  {}_{\psi}\circ{}^{t}\left(  D\varphi\right)  _{\psi}$
o\`{u} $\varphi=\psi^{-1}$, le point pr\'{e}c\'{e}dent donne que $\varphi$ est
$\left(  \sigma,\sigma^{\prime}\right)  $-analytique si et seulement si
$\varphi_{\ast}c_{\sigma}=c_{\sigma^{\prime}}$. Par ailleurs, $\varphi_{\ast
}c_{\sigma}=\left(  \det\sigma\right)  _{\psi}.\varphi_{\ast}c_{\sigma}%
=\det\left(  \sigma_{\psi}\right)  .\varphi_{\ast}c_{\sigma}$. Donc la
relation $\varphi_{\ast}\sigma=\sigma^{\prime}$ impose $\det\left(
\sigma_{\psi}\right)  =\det c_{\sigma^{\prime}}$ et $\varphi_{\ast}c_{\sigma
}=c_{\sigma^{\prime}}$.
\end{proof}

\subsection{L'op\'{e}rateur de Dirichlet-Neumann\label{S/ PbDNI}}

Consid\'{e}rons une structure de conductivit\'{e} $\left(  M,\sigma\right)  $
de dimension $2$ et $\mu$ une forme volume sur $\overline{M}$. On \'{e}quipe
$M$ d'une m\'{e}trique $g$ arbitraire et on note $\nu$ le champ de vecteur
d\'{e}fini sur $bM$ tel que pour tout $p\in bM$, $\left(  \nu_{p},\tau
_{p}\right)  $ est une base $g$-orthonorm\'{e}e directe de $T_{p}\overline{M}%
$. On dispose alors de l'op\'{e}rateur de Dirichlet-Neumann
\guillemotleft normal\guillemotright \ $N_{\nu}^{\sigma}$ d\'{e}fini pour
toute fonction suffisamment r\'{e}guli\`{e}re $u:bM\rightarrow\mathbb{R}$ par
\begin{equation}
N_{\nu}^{\sigma}u=\left.  \frac{\partial E_{\sigma}u}{\partial\nu}\right\vert
_{bM} \label{F/ Nsigmag}%
\end{equation}
Ainsi, lorsque $u:bM\rightarrow\mathbb{R}$ est suffisamment
r\'{e}guli\`{e}re,
\[
dE_{\sigma}u=\left(  E_{\sigma}u.\nu\right)  \nu^{\ast}+\left(  E_{\sigma
}u.\tau\right)  \tau^{\ast}=\left(  N_{\nu}^{\sigma}u\right)  \nu^{\ast
}+\left(  du.\tau\right)  \tau^{\ast}.
\]
Cette formule montre que la donn\'{e}e de $N_{\nu}^{\sigma}$ qui d\'{e}pend du
choix d'une m\'{e}trique peut \^{e}tre remplac\'{e}e par celle de
l'op\'{e}rateur de Dirichlet-Neumann
\guillemotleft diff\'{e}rentiel\guillemotright \ $N_{d}^{\sigma}$ dont
l'action sur une fonction $u:bM\rightarrow\mathbb{R}$ suffisamment
r\'{e}guli\`{e}re est d\'{e}finie par%
\begin{equation}
N_{d}^{\sigma}u=d\left.  E_{\sigma}u\right\vert _{bM} \label{F/ Nsigma}%
\end{equation}
Dans le cas particulier o\`{u} $\det\sigma=1$, il est not\'{e}
dans~\cite{HeG-MiV2007} que $\left.  \partial E_{\sigma}u\right\vert
_{bM}=\left(  L_{\nu}^{\sigma}u\right)  \left(  \nu^{\ast}+i\tau^{\ast
}\right)  $ o\`{u} $L_{\nu}^{\sigma}=\frac{1}{2}\left(  N_{\nu}^{\sigma
}-i\frac{\partial}{\partial\tau}\right)  $ de sorte qu'on peut consid\'{e}rer
que l'op\'{e}rateur de Dirichlet-Neumann
\guillemotleft complexe\guillemotright \ $\theta_{c}^{\sigma}$ d\'{e}fini sur
les fonctions $u:bM\rightarrow\mathbb{R}$ suffisamment r\'{e}guli\`{e}res par
\begin{equation}
\theta_{c}^{\sigma}u=\left.  \partial E_{\sigma}u\right\vert _{bM}=\left(
L_{\nu}^{c_{\sigma}}u\right)  \left(  \nu^{\ast}+i\tau^{\ast}\right)
\label{F/ theta}%
\end{equation}
\medskip

Ceci \'{e}tant, la t\^{a}che \`{a} accomplir est g\'{e}n\'{e}ralement
pens\'{e}e comme la reconstruction de $\left(  M,\sigma\right)  $ \`{a} partir
de $\left(  \partial M,N_{\nu}^{\sigma}\right)  $, $\partial M$ \'{e}tant $bM$
muni de l'orientation induite par $M$ et $\nu$. Cette formulation est
ambigu\"{e} car elle ne dit pas s'il s'agit de d\'{e}terminer $M$ comme une
vari\'{e}t\'{e} abstraite, comme une vari\'{e}t\'{e} plong\'{e}e ou encore
plus pr\'{e}cis\'{e}ment comme une certaine sous-vari\'{e}t\'{e} d'un espace
standard et le succ\`{e}s de cette entreprise d\'{e}pend du point de vue adopt\'{e}.

Lorsque $M$ est un domaine de $\mathbb{R}^{2}$, la conductivit\'{e} est
souvent assimil\'{e}e \`{a} la matrice $\left(  \sigma_{jk}\right)
=Mat_{\left(  dx_{1},dx_{2}\right)  }^{\left(  dx_{2},-dx_{1}\right)  }\left(
\sigma\right)  $ o\`{u} $\left(  x_{1},x_{2}\right)  $ est le couple des
coordonn\'{e}es standards de $\mathbb{R}^{2}$; (\ref{F/ dsdu=0}) s'\'{e}crit
alors sous la forme%
\begin{equation}
\sum\limits_{j,k=1,2}\frac{\partial}{\partial x_{j}}\left(  \sigma_{jk}%
\frac{\partial U}{\partial_{x_{k}}}\right)  =0~~\&~~\left.  U\right\vert
_{bM}=u. \label{F/ dsdu=0 ds R^2}%
\end{equation}
et les conditions impos\'{e}es \`{a} $\sigma$ pour \^{e}tre une
conductivit\'{e} se traduisent par le fait que $\left(  \sigma_{jk}\right)  $
est sym\'{e}trique et d\'{e}finie positive.

Si le probl\`{e}me pos\'{e} est compris comme celui de reconstruire $\left(
\sigma_{jk}\right)  $ \`{a} partir $\left(  \partial M,N_{d}^{\sigma}\right)
$, il est sans solution naturelle car d'apr\`{e}s une remarque de Tartar
cit\'{e}e par \cite{KoR-VoM1984}, lorsque $\varphi\in C^{1}\left(
\overline{M},\overline{M}\right)  $ est un diff\'{e}omorphisme co\"{\i}ncidant
avec l'identit\'{e} sur $bM$ et $\Phi$ est la matrice jacobienne de $\varphi$,
$\left(  \beta_{jk}\right)  =\frac{1}{\det\Phi}{}^{t}\Phi\left(  \sigma
_{jk}\right)  \Phi$ d\'{e}finit une conductivit\'{e} $\beta$ telle que
$N_{d}^{\beta}=N_{d}^{\sigma}$. Cependant, la section pr\'{e}c\'{e}dente
indique que $\varphi$ est alors un biholomorphisme entre les surfaces de
Riemann $\left(  M,\mathcal{C}_{\beta}\right)  $ et $\left(  M,\mathcal{C}%
_{\alpha}\right)  $. Bien qu'elles le m\^{e}me ensemble sous-jacent, il est
plus indiqu\'{e} de les voir comme deux plongement diff\'{e}rents de la
m\^{e}me surface de Riemann abstraite.

Cet exemple invite \`{a} consid\'{e}rer le probl\`{e}me de conductivit\'{e}
inverse comme \'{e}tant celui de la reconstruction d'une surface de Riemann
\`{a} bord abstraite $M$ et d'une fonction $s:\overline{M}\rightarrow
\mathbb{R}_{+}^{\ast}$ \`{a} partir de la donn\'{e}e de $bM$, de la
restriction de $s.c$ \`{a} $T_{bM}\overline{M}$ o\`{u} $c$ est l'op\'{e}rateur
de conjugaison de $T^{\ast}\overline{M}$ et de l'op\'{e}rateur%
\[
N^{s,c}:\mathcal{F}\left(  bM\right)  \ni u\mapsto\left.  d^{c}E_{s.c}%
u\right\vert _{bM}%
\]
o\`{u} $\mathcal{F}\left(  M\right)  $ est n'importe quel espace raisonnable
de fonctions comme $C^{0}\left(  bM\right)  $, $C^{\infty}\left(  bM\right)  $
ou $H^{1/2}\left(  bM\right)  $, $d^{c}=i\left(  \overline{\partial}%
-\partial\right)  $, $\partial=d-\overline{\partial}$ et $\overline{\partial}$
est l'op\'{e}rateur de Cauchy-Riemann de $M$.\medskip

Comme il a \'{e}t\'{e} pr\'{e}cis\'{e} au d\'{e}but de la section
pr\'{e}c\'{e}dente, il est possible de s'affranchir de la forme volume
auxiliaire $\mu$. Puisque dans le probl\`{e}me inverse envisag\'{e} ici,
l'op\'{e}rateur de Dirichlet-Neumann $N_{d}^{\sigma}$ et $\sigma\left\vert
_{bM}\right.  $ sont consid\'{e}r\'{e}s comme connus, l'op\'{e}rateur
$\ast_{\sigma}$ de conjugaison associ\'{e} \`{a} la structure complexe
$\mathcal{C}_{\sigma}$ de $\left(  M,\sigma\right)  $ est connu quand il agit
sur le fibr\'{e} $T_{bM}^{\ast}\overline{M}=\underset{s\in bM}{\cup}%
T_{s}\overline{M}$. Ayant fix\'{e} une section lisse et g\'{e}n\'{e}ratrice
$\tau^{\ast}$ de $T^{\ast}bM$, on pose $\nu_{s}^{\ast}=-\left(  \ast_{\sigma
}\right)  _{s}\tau_{s}^{\ast}$ pour tout $s\in bM$. Par d\'{e}finition d'une
conductivit\'{e}, $bM\ni s\mapsto\nu_{s}^{\ast}\wedge\tau_{s}^{\ast}$ est
alors une section lisse du fibr\'{e} des formes volumes de $\overline{M}$ et
peut \^{e}tre prolong\'{e}e en une forme volume lisse pour $\overline{M}$.
Bien que ce prolongement ne soit pas unique, tout tenseur qui serait une
conductivit\'{e} pour l'un de ces prolongements le serait pour tous.

\section{Surfaces de Riemann nodales \`{a} bord\label{S/ SurfacesNodales}}

Une surface de Riemann nodale \`{a} bord $Q$ est un ensemble de la forme
$\left(  \overline{S}/\mathcal{R}\right)  \backslash\pi\left(  bS\right)  $
o\`{u} $S$ est une surface de Riemann \`{a} bord, $\mathcal{R}$ une relation
d'\'{e}quivalence identifiant un nombre fini de points de $\overline{S}$ mais
telle que deux points distincts de $bS$ sont dans deux classes diff\'{e}rentes
et $\pi$ est la projection naturelle de $\overline{S}$ sur $\overline
{S}/\mathcal{R}$. En particulier, $\pi_{bS}=\pi\left\vert _{bS}^{bS}\right.  $
est une bijection.

On munit $\overline{S}/\mathcal{R}$ de la topologie quotient de sorte que $Q$
est un ouvert, $\overline{Q}=\overline{S}/\mathcal{R}$ et $bQ=\pi\left(
bS\right)  $. On d\'{e}signe par $\operatorname*{Reg}Q$ l'ensemble des points
de $Q$ qui n'ont par $\pi$ qu'un ant\'{e}c\'{e}dent et on pose
$\operatorname*{Sing}Q=Q\backslash\operatorname*{Reg}Q$~; on d\'{e}finit de
fa\c{c}on analogue $\operatorname*{Reg}\overline{Q}$ et $\operatorname*{Sing}%
\overline{Q}$. Puisque $\pi$ est bijective de $\overline{S}\cap\pi^{-1}\left(
\operatorname*{Reg}\overline{Q}\right)  $ sur $\operatorname*{Reg}\overline
{Q}$, $\pi$ permet de donner \`{a} $\operatorname*{Reg}\overline{Q}$ une
structure de surface de Riemann ouverte \`{a} bord de bord $\left(  bQ\right)
\cap\operatorname*{Reg}\overline{Q}$ et \`{a} $\pi_{\operatorname*{reg}}%
=\pi\left\vert _{S\cap\pi^{-1}\left(  \operatorname*{Reg}\overline{Q}\right)
}^{\operatorname*{Reg}\overline{Q}}\right.  $ le statut d'isomorphisme.

Cette structure complexe se prolonge \`{a} travers les \'{e}ventuelles
singularit\'{e}s de $\overline{Q}$ de la fa\c{c}on suivante. Soient pour
$q\in\overline{Q}$ donn\'{e}, les ant\'{e}c\'{e}dents $p_{1},...,p_{\nu\left(
q\right)  }$ de $q$ par $\pi$ et $V_{q,1},...,V_{q,\nu\left(  q\right)  }$ des
voisinages connexes ouverts dans $\overline{S}$ de ces points tels que pour
tout $j\in\left\{  1,...,\nu\left(  q\right)  \right\}  $, $\overline{V_{q,j}%
}\cap\pi^{-1}\left(  \operatorname*{Sing}\overline{Q}\right)  =\left\{
p_{j}\right\}  $, $\overline{V_{q,j}}\cap bS$ est soit vide, soit un ouvert
connexe de $bS$ et $\overline{V_{q,j}}\cap\overline{V_{q,k}}=\emptyset$
lorsque $k\neq j$. Pour chaque $j$, $\pi$ r\'{e}alise une bijection $\pi
_{q,j}$ de $V_{q,j}\backslash bS$ sur $Q_{q,j}$, ce qui permet de munir
$Q_{q,j}$ ou $Q_{q,j}\cup\left(  \overline{Q_{q,j}}\cap bS\right)  $ si $q\in
bQ$ et $j$ est l'indice tel que $p_{j}\in\left(  bS\right)  \cap\pi
^{-1}\left(  Q\right)  $, d'une structure de surface de Riemann, \`{a} bord si
$q\in bQ$. Par construction, $\pi_{q,j}$ prolonge holomorphiquement la
restriction de $\pi_{\operatorname*{reg}}$ \`{a} $V_{q,j}$ pour tout $j$.

On appelle branche de $\overline{Q}$ en $q$ toute surface de Riemann connexe
$B$ de la forme $Q_{q,j}$ - si $\overline{B}\cap bS=\varnothing$, $B$ est dite
int\'{e}rieure et si $\overline{B}\cap bS$ contient un voisinage de $q$ dans
$bS$, $B$ est dite de bord. \medskip

Par fonction nodale classe $C^{k}$ sur $U$ ou $\overline{U}$, $0\leqslant
k\leqslant\infty$, (resp. holomorphe sur $U$), $U$ \'{e}tant un ouvert de $Q$,
on entend une fonction usuelle sur $\operatorname*{Reg}U$ qui se prolonge en
fonction de classe $C^{k}$ (resp. holomorphe) le long de toute branche de $U$
ou $\overline{U}$ selon le cas \'{e}ch\'{e}ant. On d\'{e}finit de fa\c{c}on
similaire la notion d'application nodale de classe $C^{k}$ ou holomorphe \`{a}
valeurs dans une vari\'{e}t\'{e} de classe $C^{k}$ ou analytique. Une
application $f$ \`{a} valeurs dans $\overline{Q}$ est dite de classe $C^{k}$
(resp. holomorphe) si pour toute branche $B$ de $\overline{Q}$, $f^{-1}\left(
B\right)  $ est une vari\'{e}t\'{e} nodale de classe $C^{k}$ (ou analytique)
et $f$ est de classe $C^{k}$ (resp. holomorphe) sur $f^{-1}\left(  B\right)  $.

Si $X$ est une vari\'{e}t\'{e} r\'{e}elle \`{a} bord de classe $C^{k}$ et $Q$
une surface de Riemann \`{a} bord, une $C^{k}$-normalisation (resp.
normalisation) de $X$ sur $Q$ est une application $f:\overline{X}%
\rightarrow\overline{Q}$ qui est un $C^{k}$-diff\'{e}omorphisme (resp.
biholomorphisme) usuel de $f^{-1}\left(  \operatorname*{Reg}\overline
{Q}\right)  $ sur $\operatorname*{Reg}\overline{Q}$ et de $f^{-1}\left(
B\right)  $ sur $B$ pour toute branche $B$ de $Q$. Avec cette d\'{e}finition,
$\pi:\overline{S}\rightarrow\overline{Q}$ est une normalisation de
$Q$.\medskip

Consid\'{e}rons une autre surface de Riemann nodale \`{a} bord $Q^{\prime}$
qui est le quotient d'une surface de Riemann \`{a} bord $S^{\prime}$ et notons
$\pi^{\prime}$ la projection naturelle de $\overline{S^{\prime}}$ sur
$\overline{Q^{\prime}}$. On se donne une application nodale $\varphi
:\overline{Q}\longrightarrow\overline{Q^{\prime}}$~; $\varphi$ est donc
univalu\'{e}e sur $\operatorname*{Reg}\overline{Q}$ et multivalu\'{e}e sur
$\operatorname*{Sing}\overline{Q}$. On dit que $\varphi$ est un isomorphisme
approximatif de surfaces de Riemann nodale \`{a} bord si les deux conditions
suivantes sont r\'{e}alis\'{e}es~:

i) $\varphi$ est bijective de $\varphi^{-1}\left(  \operatorname*{Reg}%
\overline{Q^{\prime}}\right)  \cap\operatorname*{Reg}\overline{Q}$ sur
$\varphi\left(  \operatorname*{Reg}\overline{Q}\right)  \cap
\operatorname*{Reg}\overline{Q^{\prime}}$.

ii) Pour toute branche $B^{\prime}$ int\'{e}rieure (resp. de bord) issue d'un
point $q^{\prime}$ de $\overline{Q}$, $\varphi^{-1}\left(  B^{\prime}\right)
$ est une branche int\'{e}rieure (resp. de bord) de $\overline{Q}$ en $q$ et
$\varphi\left\vert _{B}^{B^{\prime}}\right.  $ est un isomorphisme de surfaces
de Riemann (resp. \`{a} bord). \newline Ces conditions entra\^{\i}nent les
propri\'{e}t\'{e}s suivantes~:

iii) $\varphi$ r\'{e}alise isomorphisme de surfaces de Riemann \`{a} bord de
$\varphi^{-1}\left(  \operatorname*{Reg}\overline{Q^{\prime}}\right)
\cap\operatorname*{Reg}\overline{Q}$ sur $\varphi\left(  \operatorname*{Reg}%
\overline{Q}\right)  \cap\operatorname*{Reg}\overline{Q^{\prime}}$

iv) Pour chaque $q^{\prime}\in\operatorname*{Sing}\overline{Q^{\prime}}$, on
se donne un jeu $\left(  Q_{q^{\prime},j}^{\prime}\right)  _{1\leqslant
j\leqslant\nu\left(  q^{\prime}\right)  }$ complet de branches deux \`{a} deux
disjointes de $\overline{Q^{\prime}}$ en $q^{\prime}$. Alors $\left(
\varphi^{-1}\left(  Q_{q^{\prime},j}^{\prime}\right)  \right)  _{q\in
\operatorname*{Sing}\overline{Q^{\prime}},~1\leqslant j\leqslant\nu\left(
q^{\prime}\right)  }$ est un jeu de branches de $\overline{Q}$ en ses points
singuliers, complet au sens o\`{u} pour tout $q\in\operatorname*{Sing}%
\overline{Q} $, il existe une partie $E_{q}$ de $\left\{  \left(  q^{\prime
},j\right)  ;~q^{\prime}\in\operatorname*{Sing}\overline{Q^{\prime}%
},~1\leqslant j\leqslant\nu\left(  q^{\prime}\right)  \right\}  $ telle que
les branches de $\overline{Q}$ en $q$ sont les $\varphi^{-1}\left(
Q_{q^{\prime},j}^{\prime}\right)  $, $\left(  q^{\prime},j\right)  \in E_{q}$.

v) $\varphi$ se rel\`{e}ve en un isomorphisme $\widetilde{\varphi}%
:\overline{S}\longrightarrow\overline{S^{\prime}}$ tel que pour toute branche
$B$ de $\overline{Q}$, $\varphi=\pi^{\prime}\circ\widetilde{\varphi}%
\circ\left(  \pi\left\vert _{B}^{\pi^{-1\left(  B\right)  }}\right.  \right)
^{-1}$ .\newline Si $\varphi$ est un isomorphisme approximatif et respecte les
noeuds de $\overline{Q}$ et $\overline{Q^{\prime}}$, c'est-\`{a}-dire si les
branches de $\overline{Q^{\prime}}$ en $\varphi\left(  q\right)  $ sont les
images par $\varphi$ des branches de $\overline{Q}$ en $q$, ou de fa\c{c}on
\'{e}quivalente, $\left(  \pi^{\prime}\right)  {}^{-1}\left(  \left\{
\varphi\left(  q\right)  \right\}  \right)  =\widetilde{\varphi}\left(
\pi^{-1}\left(  \left\{  q\right\}  \right)  \right)  $ pour tout
$q\in\overline{Q^{\prime}}$, on dit que $\varphi$ est un isomorphisme de
surfaces de Riemann nodales \`{a} bord. Dans ce cas, $\varphi$ est un
hom\'{e}omorphisme de $\overline{Q}\ $sur $\overline{Q^{\prime}}$.\medskip

Une structure de conductivit\'{e} sur $\overline{Q}$ est une structure de
conductivit\'{e} sur $\operatorname*{Reg}\overline{Q}$ qui se prolonge en
structure de conductivit\'{e} le long de toute branche de $\overline{Q}$. Les
consid\'{e}rations de la section pr\'{e}c\'{e}dente s'appliquent alors de
fa\c{c}on naturelle et nous n'en faisons pas le d\'{e}tail. Dans cet article,
nous n'utilisons pour les surfaces nodales que les conductivit\'{e}s de
coefficient $1$, c'est \`{a} dire les structures complexes
sous-jacentes.\medskip

Pour conclure cette section, pr\'{e}cisons que d'apr\`{e}s$^{(}$%
\footnote{Cette caract\'{e}risation est prouv\'{e}e pour le cas o\`{u}
$bQ\subset\operatorname*{Reg}\overline{Q}$ mais la preuve s'applique sans
changement au cas pr\'{e}sent.}$^{)}$~\cite[prop. 2]{HeG-MiV2012} une
distribution $u$ sur un ouvert $W$ de $\overline{Q}$ est harmonique si elle
est harmonique au sens usuel sur $W\cap\operatorname*{Reg}Q$, continue sur
$W\cap\operatorname*{Reg}\overline{Q} $ et si pour tout point singulier $q$ de
$\overline{Q}$ les deux conditions suivantes sont v\'{e}rifi\'{e}es~:

\hspace{2.015pc}1) pour toute branche int\'{e}rieure $B$ de $\overline{Q}$ en
$q$ suffisamment petite pour qu'on puisse s'en donner une une coordonn\'{e}e
holomorphe $z$, il existe $c_{B}\in\mathbb{C}$ tel que $u\left\vert
_{Q_{q,j}\backslash\left\{  q\right\}  }\right.  -2c_{B}\ln\left\vert
z\right\vert $ se prolonge \`{a} $B$ en fonction harmonique usuelle.

\hspace{2.0075pc}2) $\sum c_{B}=0$, la somme \'{e}tant \'{e}tendue aux
branches int\'{e}rieures de $\overline{Q}$ en $q$.

Inversement, toute distribution ayant ces deux propri\'{e}t\'{e}s est une
distribution harmonique. Ceci entra\^{\i}ne qu'une m\^{e}me fonction $u$
continue sur $bQ$ admet plusieurs prolongements en tant que distribution
harmonique $U$~; le probl\`{e}me de Dirichlet pour $u$ n'est bien pos\'{e} que
si on sp\'{e}cifie pour tout $q\in\operatorname*{Sing}\overline{Q}$ et toute
branche int\'{e}rieure $\mathcal{B}$ de $\overline{Q}$ en $q$, le r\'{e}sidu
$c_{B}$ de $\partial U\left\vert _{\mathcal{B}}\right.  $ en $q$. En
particulier, $\widehat{u}$ d\'{e}signant le prolongement harmonique de
$u\circ\pi_{bS}^{-1}$ \`{a} $S$, $\pi_{\ast}\widehat{u}$ est la seule
distribution harmonique qui est continue le long de toute branche de
$\overline{Q}\backslash\left\{  q\right\}  $ et co\"{\i}ncide avec $u$ sur
$bQ$~; on l'appelle le prolongement harmonique simple de $u$.

En particulier, on peut d\'{e}finir pour une surface de Riemann nodale un
op\'{e}rateur de Dirichlet-Neumann complexe $\theta_{c}^{\sigma}$ o\`{u} dans
(\ref{F/ theta}) on utilise les extensions harmoniques simples.

\section{Th\'{e}or\`{e}mes d'unicit\'{e}\label{S/ Unicite}}

\subsection{Conductivit\'{e} anisotrope\label{S/ condaniso}}

Lorsque la conductivit\'{e} n'est pas isotrope, les auteurs se sont
concentr\'{e}s sur l'injectivit\'{e} \`{a} diff\'{e}omorphisme pr\`{e}s de
$\sigma\mapsto N_{\nu}^{\sigma}$, c'est-\`{a}-dire \`{a} la r\'{e}ciproque de
la remarque de Tartar. Cette injectivit\'{e} est prouv\'{e}e par
Nachman~\cite{NaA1996} pour une conductivit\'{e} de classe $C^{2}$ et un
domaine de $\mathbb{R}^{2}$. Dans~\cite{AsK-PaL-LaM2005}, elle est \'{e}tablie
pour une conductivit\'{e} de classe $L^{\infty}$ mais pour un domaine
simplement connexe de $\mathbb{R}^{2}$.

Dans le cas sp\'{e}cial o\`{u} le coefficient de conductivit\'{e} est
constant, la question est de savoir si deux structures conformes sur $M$ sont
identiques lorsqu'elles partagent le m\^{e}me op\'{e}rateur de Dirichlet \`{a}
Neumann. Une r\'{e}ponse positive est affirm\'{e}e par Lassas et
Uhlmann~\cite{LaM-UhG2001} quand $M$ est connexe et Belishev la corrobore
dans~\cite{BeM2003} en montrant que $M$ peut \^{e}tre retrouv\'{e}e comme le
spectre de l'alg\`{e}bre des restrictions \`{a} $bM$ des fonctions holomorphes
sur $M$ se prolongeant contin\^{u}ment \`{a} $\overline{M}$.

Dans~\cite{LaM-UhG2001} et~\cite{BeM2003}, la connaissance compl\`{e}te de
l'op\'{e}rateur de Dirichlet-Neumann est n\'{e}cessaire pour obtenir
l'unicit\'{e} de la structure conforme. Dans~\cite{HeG-MiV2007}, il est dit
que celle-ci est d\'{e}termin\'{e}e par l'action de l'op\'{e}rateur de
Dirichlet-Neumann sur trois fonctions g\'{e}n\'{e}riques mais la preuve
donn\'{e}e de ce r\'{e}sultat n'est correcte qu'en renfor\c{c}ant un peu les
conditions g\'{e}n\'{e}riques impos\'{e}es \`{a} ces trois fonctions, ce qui
est fait dans \cite{HeG-MiV2012}. L'unicit\'{e} en question ici peut \^{e}tre
aussi obtenue en augmentant le nombre de fonctions g\'{e}n\'{e}riques comme
dans~\cite{HeG-MiV2015}. Le th\'{e}or\`{e}me~\ref{T/ unicite} de cet article
en donne une preuve avec les hypoth\`{e}ses de~\cite{HeG-MiV2007} et la
section~\ref{S/ algorithme} propose une reconstruction plus explicite de la
structure complexe $M$ \`{a} partir de son op\'{e}rateur de Dirichlet-Neumann.

Dans~\cite{HeG-SaM2010} pour un domaine de $\mathbb{R}^{2}$, puis
dans~\cite{HeG-SaM2012} pour le cas g\'{e}n\'{e}ral d'une vari\'{e}t\'{e}
r\'{e}elle connexe $M$ de dimension $2$, Henkin et Santacesaria ont fait une
avanc\'{e}e majeure dans la th\'{e}orie en prouvant que l'op\'{e}rateur de
Dirichlet-Neumann d\'{e}termine la structure complexe de $\left(
M,\sigma\right)  $ sous la forme d'une surface de Riemann nodale \`{a} bord
plong\'{e}e dans $\mathbb{C}^{2}$.

\begin{theorem}
[Henkin-Santacesaria, 2012]\label{T/ HS}Soit $\left(  M,\sigma\right)  $ une
structure de conductivit\'{e} de dimension $2$, $\sigma$ \'{e}tant de classe
$C^{3}$. Il existe alors une surface de Riemann nodale \`{a} bord
$\mathcal{M}$ plong\'{e}e dans $\mathbb{C}^{2}$ et une $C^{3}$-normalisation
$F:\overline{M}\rightarrow\overline{\mathcal{M}}$ telle que $F_{\ast}%
\sigma=tc_{\mathcal{M}}$ o\`{u} $t\in C^{3}\left(  \mathcal{M},\mathbb{R}%
_{+}^{\ast}\right)  $ et $c_{\mathcal{M}}$ est la structure complexe induite
par $\mathbb{C}^{2}$ sur $\mathcal{M}$. En outre si $F:\overline{M}%
\rightarrow\overline{\mathcal{M}^{\prime}}$ est une autre $C^{3}%
$-normalisation du m\^{e}me type, $\mathcal{M}$ et $\mathcal{M}^{\prime}$ sont
approximativement isomorphes. Enfin, les valeurs au bord de $F$ et en
particulier $b\mathcal{M}$ sont d\'{e}termin\'{e}es par $bM$, $\sigma
\left\vert _{T_{bM}M}\right.  $ et l'op\'{e}rateur de Dirichlet-Neumann
$N_{d}^{\sigma}$ de $\left(  M,\sigma\right)  $.
\end{theorem}

Notons que gr\^{a}ce au lemme~\ref{L/ sigma-analyticite}, $F$ est holomorphe
au sens o\`{u} pour toute branche $B$ de $\mathcal{M}$, $F$ est analytique de
$\left(  F^{-1}\left(  B\right)  ,\mathcal{C}_{\sigma}\right)  $ dans
$\mathbb{C}^{2}$. Par ailleurs, il r\'{e}sulte de la preuve de ce
th\'{e}or\`{e}me que les singularit\'{e}s de $\mathcal{M}$ sont les points de
$F\left(  \overline{M}\right)  $ ayant plusieurs ant\'{e}c\'{e}dents par $F$.
Ainsi, lorsque $\mathcal{M}$ n'a pas de singularit\'{e}, $F$ est un
diff\'{e}omorphisme de $\overline{M}$ sur $\mathcal{M}$ v\'{e}rifiant les
hypoth\`{e}ses du lemme~\ref{L/ sigma-analyticite}, ce qui en fait un
isomorphisme de surfaces de Riemann \`{a} bord.

Dans~\cite{HeG-SaM2012}, il est dit que $\mathcal{M}$ et $\mathcal{M}^{\prime
}$ sont isomorphes mais sans que le sens en soit pr\'{e}cis\'{e}.
D\'{e}montrons succinctement qu'il s'agit au moins d'isomorphie approximative.
Supposons que $F:\overline{M}\rightarrow\overline{\mathcal{M}}$ et
$G:F:\overline{M}\rightarrow\overline{\mathcal{M}^{\prime}}$ sont deux $C^{3}%
$-normalisations du type ci-dessus. Posons $G_{\operatorname{reg}}=G\left\vert
_{_{\operatorname*{Reg}\mathcal{M}}}^{_{G^{-1}\left(  \operatorname*{Reg}%
\mathcal{M}\right)  }}\right.  $, $F_{\operatorname{reg}}=F\left\vert
_{_{\operatorname*{Reg}\mathcal{M}^{\prime}}}^{_{F^{-1}\left(
\operatorname*{Reg}\mathcal{M}^{\prime}\right)  }}\right.  $ et notons
$H_{\operatorname{reg}}$ l'application de $\operatorname*{Reg}\mathcal{M}\cap
G\left(  F^{-1}\left(  \operatorname*{Reg}\mathcal{M}^{\prime}\right)
\right)  $ dans $\operatorname*{Reg}\mathcal{M}\cap F\left(
\operatorname*{Reg}\mathcal{M}^{\prime}\right)  $ d\'{e}finie par
$H_{\operatorname{reg}}\left(  z\right)  =F_{\operatorname{reg}}\left(
G_{\operatorname{reg}}^{-1}\left(  z\right)  \right)  $. Parce que $F$ et $G$
sont des normalisations, $H_{\operatorname{reg}}$ se prolonge
holomorphiquement le long de toute branche de $\mathcal{M}$ en une application
(multivalu\'{e}e) $H$ de $\mathcal{M}$ \`{a} valeurs dans $\mathcal{M}%
^{\prime}$. Par construction, $H\left(  \mathcal{M}^{\prime}\right)  $ et
$\mathcal{M}$ sont des courbes complexes qui ne diff\'{e}rent qu'en un nombre
fini de points. Elles sont donc \'{e}gales et en particulier,
$\operatorname*{Sing}\mathcal{M}$ et $\operatorname*{Sing}\mathcal{M}^{\prime
}$ ont m\^{e}me cardinal. Il s'ensuit que $\mathcal{M}$ et $\mathcal{M}%
^{\prime}$ sont approximativement isomorphes.

En r\'{e}alit\'{e}, $\mathcal{M}$ et $\mathcal{M}^{\prime}$ sont bien
isomorphes au sens fort donn\'{e} dans cet article. En effet, $\mathcal{M}$ et
$\mathcal{M}^{\prime}$ induisent par rel\`{e}vement sur $M$ des structures
complexes a priori diff\'{e}rentes mais qui co\"{\i}ncident sur $bM$ et
partagent le m\^{e}me op\'{e}rateur de Dirichlet-Neumann. Gr\^{a}ce au
th\'{e}or\`{e}me~\ref{T/ unicite} \'{e}nonc\'{e} et prouv\'{e} plus loin, les
surfaces de Riemann \`{a} bord correspondantes sont donc isomorphes et du
coup, $\mathcal{M}$ et $\mathcal{M}^{\prime}$ le sont aussi en tant que
surfaces de Riemann nodale \`{a} bord.

Ainsi, $\mathcal{M}$ est un mod\`{e}le, \'{e}ventuellement singulier, de la
structure complexe de $\left(  M,\sigma\right)  $. Ce mod\`{e}le est
calculable au sens effectif du terme. En effet, $\mathcal{M}$ est un
sous-ensemble analytique de dimension $1$ de $\mathbb{C}^{2}\backslash
b\mathcal{M}$ qui au sens des courants de $\mathbb{C}^{2}$ satisfait $d\left[
\mathcal{M}\right]  =F_{\ast}\left[  \partial M\right]  $ o\`{u} $\left[
\mathcal{M}\right]  $ d\'{e}signe le courant d'int\'{e}gration sur
$\mathcal{M}$ et $\left[  \partial M\right]  $ celui du bord de $M$
orient\'{e} par $M$. Dans cette situation, on sait, essentiellement depuis les
travaux de Harvey et Lawson~\cite{HaR-LaB1975, HaR1977}, que $\mathcal{M}$
peut \^{e}tre calcul\'{e}e de fa\c{c}on effective gr\^{a}ce \`{a} des formules
de type Cauchy (voir par exemple~\cite[th. 2]{HeG-MiV2007} ou \cite[prop.
1]{HeG-SaM2012}).

On peut consid\'{e}rer que $\mathcal{M}$ suffit comme pr\'{e}sentation de $M$
en tant que surface de Riemann mais on peut aussi souhaiter disposer d'un
atlas de la structure complexe de $M$. Le paragraphe suivant explique comment
le travail de~\cite{HeG-MiV2007, HeG-MiV2012} permet d'obtenir $M$ \`{a}
partir de $\mathcal{M}$ et donc in fine \`{a} partir de $N_{d}^{\sigma}%
$.\medskip

Dans le th\'{e}or\`{e}me ci-dessous, on note $\left[  w_{0}:\cdots
:w_{d}\right]  $ les coordonn\'{e}es homog\`{e}nes standards de $\mathbb{CP}%
_{d}$. Si $\omega_{0},...,\omega_{d}$ sont des $\left(  1,0\right)  $-formes
deux \`{a} deux proportionnelles ne s'annulant pas simultan\'{e}ment, on note
$\left[  \omega_{0}:\cdots:\omega_{d}\right]  $ ou $\left[  \omega\right]  $
l'application d\'{e}finie sur chaque $\left\{  \omega_{j}\neq0\right\}  $ par
$\left[  \omega\right]  =\left[  \frac{\omega_{0}}{\omega_{j}}:\cdots
:\frac{\omega_{d}}{\omega_{j}}\right]  $.

\begin{theorem}
\label{T/ plgt1}Soient $\left(  M,\sigma\right)  $ une structure de
conductivit\'{e} de dimension $2$, $\sigma$ \'{e}tant de classe $C^{3}$, et
$F:\overline{M}\rightarrow\overline{\mathcal{M}}$ une normalisation comme dans
le th\'{e}or\`{e}me~\ref{T/ HS}. On consid\`{e}re l'op\'{e}rateur $\theta
_{c}^{\sigma}$ d\'{e}fini en (\ref{F/ theta}) et, $\mathcal{M}$ \'{e}tant muni
de la conductivit\'{e} de coefficient associ\'{e}e \`{a} sa structure
complexe, on d\'{e}signe par $\theta_{c}^{\mathcal{M}}$ l'op\'{e}rateur de
Dirichlet-Neumann complexe d\'{e}fini de fa\c{c}on similaire \`{a} la fin de
la section~\ref{S/ SurfacesNodales}. Enfin, on pose $f=F\left\vert
_{bM}^{F\left(  bM\right)  }\right.  $.

Alors pour $u\in C^{\infty}\left(  bM,\mathbb{R}\right)  $ et $w=u\circ
f^{-1}=f_{\ast}u$, $F_{\ast}\theta_{c}^{\sigma}u=\theta_{c}^{\mathcal{M}}w$
est la restriction \`{a} $b\mathcal{M}$ de $\partial\widehat{w}$ o\`{u}
$\widehat{w}$ est l'extension harmonique simple de $w$. De plus, pour $\left(
u_{0},u_{1},u_{2},u_{3}\right)  $ g\'{e}n\'{e}rique dans $C^{\infty}\left(
bM,\mathbb{R}\right)  ^{4}$, l'application%
\[
\left[  F^{\ast}\theta_{c}^{\mathcal{M}}f_{\ast}u\right]  =\left[  F^{\ast
}\theta_{c}^{\mathcal{M}}f_{\ast}u_{0}:\cdots:F^{\ast}\theta_{c}^{\mathcal{M}%
}f_{\ast}u_{3}\right]
\]
plonge dans $\mathbb{CP}_{3}$ la courbe r\'{e}elle $bM$ sur une courbe
r\'{e}elle bord\'{e}e dans $\mathbb{CP}_{3}$ par une surface de Riemann \`{a}
bord $S$ isomorphe \`{a} $M$.
\end{theorem}

\begin{proof}
Consid\'{e}rons une fonction $u$ de $C^{0}\left(  bM,\mathbb{R}\right)  $.
Notons $\widetilde{u}$ le prolongement harmonique de $u$ \`{a} $M$~; puisque
$2\overline{\partial}{}^{\sigma}\partial^{\sigma}=idd^{c_{\sigma}}$,
$\widetilde{u}$ est l'unique solution dans $C^{0}\left(  \overline{M}\right)
$ du syst\`{e}me%
\[
\overline{\partial}{}^{\sigma}\partial^{\sigma}U=0~~\&~~U\left\vert
_{bM}\right.  =u.
\]
D'apr\`{e}s~\cite[formule~1.1]{HeG-MiV2007}, $\theta_{c}^{\sigma}u$ est la
restriction \`{a} $bM$ de la $\left(  1,0\right)  $-forme $c_{\sigma}%
$-holomorphe $\partial^{\sigma}\widetilde{u}$.

Puisque $F_{B}=F\left\vert _{F^{-1}\left(  B\right)  }^{B}\right.  $ est un
isomorphisme $\left(  \sigma,c_{\mathcal{M}}\right)  $-analytique pour toute
branche $B$ de $\mathcal{M}$ et puisque $\widetilde{u}$ est lisse, $F_{\ast
}\widetilde{u}$ est lisse le long de toute branche $B$ de $\mathcal{M}$ et
v\'{e}rifie $\left(  \overline{\partial}\partial\left(  F_{B}\right)  {}%
_{\ast}\widetilde{u}\right)  \left\vert _{B}\right.  =\left(  F_{B}\right)
{}_{\ast}\overline{\partial}{}^{\sigma}\partial^{\sigma}\widetilde{u}=0$.
Ainsi, $\widetilde{u}\circ\left(  F_{\operatorname*{Reg}\mathcal{M}}\right)
^{-1}$ se prolonge harmoniquement le long de toute branche de $\mathcal{M}$ et
d\'{e}finit sur $\mathcal{M}$ une distribution $\widehat{w}$ qui est l'unique
solution continue le long de toute branche de $\mathcal{M}$ du probl\`{e}me%
\begin{equation}
\overline{\partial}\partial w=0~~\&~~w\left\vert _{\mathcal{M}}\right.  =w
\label{F/ pbY}%
\end{equation}
quand $w=u\circ f^{-1}$. Les m\^{e}mes calculs que dans~\cite[formule~1.1]%
{HeG-MiV2007} donnent que $\theta_{c}^{\mathcal{M}}w$ est la restriction \`{a}
$b\mathcal{M}$ de la $\left(  1,0\right)  $-forme holomorphe $\partial
\widehat{w}$. Etant donn\'{e} que $F$ est $\left(  \sigma,c_{\mathcal{M}%
}\right)  $-analytique, $\partial\widehat{w}=F_{\ast}\partial^{\sigma
}\widetilde{u}$.

Consid\'{e}rons maintenant dans $C^{\infty}\left(  bM,\mathbb{R}\right)  $
quatre fonctions $u_{0},...,u_{3}$ dont on note par un tilde leur prolongement
$c_{\sigma}$-harmonique \`{a} $\overline{M}$. Gr\^{a}ce \`{a} la remarque
faite page~327 pour le th\'{e}or\`{e}me~4 de~\cite{HeG-MiV2015}, on sait que
pour $\left(  u_{0},...,u_{3}\right)  $ g\'{e}n\'{e}rique, les formes
$\partial^{\sigma}\widetilde{u_{j}}$ ne s'annulent pas simultan\'{e}ment et
que $\left[  \partial^{\sigma}u\right]  =\left[  \partial^{\sigma}u_{0}%
:\cdots:\partial^{\sigma}u_{3}\right]  $ plonge $\overline{M}$ dans
$\mathbb{CP}_{3}$ sur une surface de Riemann $S$ \`{a} bord. Dans une telle
situation, les distributions harmoniques $\widehat{w_{j}}$ qui sont les
solutions du probl\`{e}me~(\ref{F/ pbY}) quand $w=f_{\ast}u_{j}$
v\'{e}rifient
\[
\partial^{\sigma}\widetilde{u_{j}}=F^{\ast}\partial\widehat{w}_{j}=F^{\ast
}\theta_{c}^{\mathcal{M}}w_{j}=F^{\ast}\theta_{c}^{\mathcal{M}}f_{\ast}u_{j}%
\]
pour tout $j\in\mathbb{N}$.
\end{proof}

\noindent\textbf{Remarque. }$\left[  \partial w\right]  $ est d\'{e}finie sans
ambigu\"{\i}t\'{e} sur $\operatorname*{Reg}\mathcal{M}$ et se prolonge
holomorphiquement le long de toute branche $B$ de $\mathcal{M}$. L'application
multivalu\'{e}e $\left[  \partial\widehat{w}\right]  $ qui en r\'{e}sulte
d\'{e}singularise $\mathcal{M}$ au sens o\`{u} il existe une normalisation
$G:S\rightarrow\mathcal{M}$ telle que $G\circ\left[  \partial\widehat{w}%
\right]  =Id_{\mathcal{M}}$.\medskip

$\theta_{c}^{\sigma}$ n'\'{e}tant pas directement accessible \`{a} partir de
$N_{d}^{\sigma}$, la m\'{e}thode propos\'{e}e par le
th\'{e}or\`{e}me~\ref{T/ plgt1} pour construire un plongement de $\overline
{M}$ dans $\mathbb{CP}_{3}$ repose sur l'utilisation de la surface nodale
$\mathcal{M}$ fabriqu\'{e}e par le th\'{e}or\`{e}me de Henkin et Santacesaria
pour obtenir $\theta_{c}^{\mathcal{M}}$ \`{a} partir d'extensions harmonique
simples \`{a} $\mathcal{M}$ de fonctions d\'{e}finies sur $b\mathcal{M}$. Bien
que $\mathcal{M}$ soit connue, le fait que $\mathcal{M}$ soit nodale complique
les choses car ce probl\`{e}me de Dirichlet n'a de solution unique que si on
sp\'{e}cifie des donn\'{e}es aux points nodaux.\medskip

Plus pr\'{e}cis\'{e}ment, notons $q_{1},...,q_{s}$ les points de
$\overline{\mathcal{M}}$ o\`{u} $\overline{\mathcal{M}}$ a au moins deux
branches int\'{e}rieures. Pour chacun de ces points $q_{j}$, notons $B_{j,1}%
$,...,$B_{j,\mu_{j}}$ les branches int\'{e}rieures de $\mathcal{M}$ en $q_{j}%
$. Alors, pour $v\in C^{\infty}\left(  b\mathcal{M}\right)  $ donn\'{e}e, on
sait d'apr\`{e}s~\cite[prop. 2]{HeG-MiV2012} que pour toute famille
$\alpha=\left(  \left(  \alpha_{j,k}^{r}\right)  _{1\leqslant j\leqslant
s}\right)  _{1\leqslant k\leqslant\mu_{j}}$ de nombre complexes v\'{e}rifiant
$\sum\limits_{1\leqslant k\leqslant\mu_{j}}\alpha_{j,k}=0$ pour tout $j$, il
existe un et un seul prolongement de $v$ en distribution harmonique
$\widehat{v}^{\alpha}$ tel que pour tous $j,k$ le r\'{e}sidu en $q_{j}$ de
$\partial\widehat{v}^{\alpha}\left\vert _{\mathcal{B}_{j,k}}\right.  $ vaut
$\alpha_{j,k}$. En particulier, $\widehat{v}^{0}=\widehat{v}$.

Si dans le th\'{e}or\`{e}me~\ref{T/ plgt1}, on remplace $\left[  F^{\ast
}\theta_{c}^{\mathcal{M}}f_{\ast}u\right]  $ par
\[
\left[  F^{\ast}\theta_{c}^{\mathcal{M},\alpha}f_{\ast}u\right]  =\left[
F^{\ast}\partial\widehat{f_{\ast}u_{0}}^{\alpha^{0}}\left\vert _{b\mathcal{M}%
}\right.  :\cdots:F^{\ast}\partial\widehat{f_{\ast}u_{4}}^{\alpha^{4}%
}\left\vert _{b\mathcal{M}}\right.  \right]
\]
o\`{u} les familles de r\'{e}els $\alpha^{r}=\left(  \left(  \alpha_{j,k}%
^{r}\right)  _{1\leqslant j\leqslant s}\right)  _{1\leqslant k\leqslant\mu
_{j}}$ v\'{e}rifient $\sum\limits_{1\leqslant k\leqslant\mu_{j}}\alpha
_{j,k}^{r}=0$ , on obtient encore les m\^{e}mes conclusions quand $\left(
u_{0},u_{1},u_{2},u_{3}\right)  $ est g\'{e}n\'{e}rique et l'application
$\left(  j,k\right)  \mapsto\left[  \alpha_{j,k}^{0}:\cdots:\alpha_{j,k}%
^{3}\right]  $ est bien d\'{e}finie et injective. Notons que la seconde partie
du th\'{e}or\`{e}me admet aussi une modification similaire.\medskip

Cette modification du th\'{e}or\`{e}me~\ref{T/ plgt1} dit en substance qu'il
n'est pas utile de se soucier de quels prolongements en distribution
harmonique on use. Cependant, la construction de tels prolongements n'est pas
forc\'{e}ment \'{e}vidente. La section suivante en propose une.

\subsection{Fonctions de Green nodale\label{S/ FGN}}

Quand $M$ est une surface de Riemann \`{a} bord lisse, une fonction de Green
pour $M$ est une fonction sym\'{e}trique r\'{e}elle $g$ d\'{e}finie sur
$\overline{M}\times\overline{M}$ priv\'{e}e de sa diagonale $\Delta
_{\overline{M}}$ telle que pour tout $q\in M$, $g_{q}=g\left(  q,.\right)  $
est harmonique sur $M\backslash\left\{  q\right\}  $, continue sur
$\overline{M}\backslash\left\{  q\right\}  $ et pr\'{e}sente une
singularit\'{e} logarithmique isol\'{e}e en $q$, ce qui signifie qu'ayant
choisi au voisinage de $q$ dans $M$ une coordonn\'{e}e holomorphe $z$
centr\'{e}e en $q $, $g_{q}-\frac{1}{2\pi}\ln\left\vert z\right\vert $ se
prolonge harmoniquement au voisinage de $q$. L'unique fonction de Green telle
que $g_{q}\left\vert _{bM}\right.  =0$ pour tout $q\in M$ est dite principale.
Avec une fonction de Green $g$, on fabrique l'op\'{e}rateur $T$ qui \`{a}
$v\in C^{\infty}\left(  bM\right)  $ associe la fonction $Tv$ d\'{e}finie par%
\begin{equation}
Tv:M\ni q\mapsto\frac{i}{2}\int_{\partial M}v\overline{\partial}g_{q}
\label{F/ opT}%
\end{equation}
Lorsque $v\in C^{\infty}\left(  bM\right)  $, $Tv$ est harmonique sur $M$, se
prolonge \`{a} $\overline{M}$ en fonction de classe $C^{\infty}$ et si $g$ est
principale, co\"{\i}ncide avec $v$ sur $bM$. Si on se ram\`{e}ne au cas o\`{u}
$\overline{M}$ est relativement compact dans une surface de Riemann
$\widetilde{M}$ (on peut par exemple choisir pour $\widetilde{M}$ le double de
$M$), on peut consid\'{e}rer l'op\'{e}rateur $S$ qui \`{a} $v\in C^{\infty
}\left(  bM\right)  $ associe la fonction $Sv$ d\'{e}finie par%
\[
Sv:\widetilde{M}\backslash\overline{M}\ni q\mapsto\frac{i}{2}\int_{\partial
M}v\overline{\partial}g_{q}%
\]
Lorsque $v\in C^{\infty}\left(  bM\right)  $, $Sv$ est harmonique sur $M$, se
prolonge \`{a} $\overline{M}$ en fonction de classe $C^{\infty}$ et
d'apr\`{e}s un r\'{e}sultat de Sohotksy de 1873, v\'{e}rifie $v=Tv-Sv$ sur
$bM$. En outre, d'apr\`{e}s les travaux de Fredholm en 1900, l'\'{e}quation
int\'{e}grale $v=w+\left(  Sw\right)  \left\vert _{bM}\right.  $ admet une
unique solution qu'on note $Rv$ et le prolongement harmonique $Ev$ de $v$
\`{a} $M$ est $Ev=TRv$. En particulier, $G_{M}:\left(  q,z\right)  \mapsto
g\left(  q,z\right)  -E\left(  g_{z}\left\vert _{bM}\right.  \right)  \left(
q\right)  $ est la fonction de Green principale de $M$. Quand $g$ est une
fonction de Green principale, la formule~(\ref{F/ opT}) donne le prolongement
harmonique de $v$ \`{a} $M$.\medskip

\begin{definition}
Soit $\mathcal{Y}$ une courbe complexe ouverte, \'{e}ventuellement
singuli\`{e}re, d'un ouvert de $\mathbb{C}^{2}$. Une fonction de Green pour
$\mathcal{Y}$ est une fonction $g:\left(  \operatorname*{Reg}\mathcal{Y}%
\times\operatorname*{Reg}\mathcal{Y}\right)  \backslash\Delta
_{\operatorname*{Reg}\mathcal{Y}}\rightarrow\mathbb{R}$ sym\'{e}trique telle
que pour tout $q_{\ast}\in\operatorname*{Reg}\mathcal{Y}$, $g_{q_{\ast}%
}=g\left(  q_{\ast},.\right)  $ v\'{e}rifie $i\partial\overline{\partial
}g_{q_{\ast}}=\delta_{q_{\ast}}dV$ au sens des courant, $\delta_{q_{\ast}}$
\'{e}tant la mesure de Dirac port\'{e}e par $\left\{  q_{\ast}\right\}  $ et
$dV=i\partial\overline{\partial}\left\vert .\right\vert ^{2}$ - ceci implique
en particulier que $\partial g_{q_{\ast}}$ est une $\left(  1,0\right)
$-forme faiblement holomorphe sur $\mathcal{Y}\backslash\left\{  q_{\ast
}\right\}  $ au sens de~\cite{RoM1954}.

Lorsque $\mathcal{Y}$ est une surface de Riemann nodale ouverte, quotient
d'une surface de Riemann $S$ par une relation d'\'{e}quivalence et lorsque
$\pi$ est la projection canonique de $S$ sur $\mathcal{Y}$, une fonction de
Green simple pour $\mathcal{Y}$ est une fonction $g:\left(
\operatorname*{Reg}\mathcal{Y}\times\operatorname*{Reg}\mathcal{Y}\right)
\backslash\Delta_{\operatorname*{Reg}\mathcal{Y}}\rightarrow\mathbb{R}$ pour
laquelle il existe une fonction $\widetilde{g}:\left(  S\times S\right)
\backslash\Delta_{S}\rightarrow\mathbb{R}$ ayant les propri\'{e}t\'{e}s
suivantes :\ - lorsque $q_{\ast}\in\operatorname*{Reg}\mathcal{Y}$ et
$\left\{  s_{\ast}\right\}  =\pi^{-1}\left(  q_{\ast}\right)  $, $g_{q_{\ast}%
}=\pi_{\ast}\widetilde{g}_{s_{\ast}}$ au sens des courants - lorsque $q_{\ast
}\in\operatorname*{Sing}\mathcal{Y}$, $\mathcal{B}$ est une branche de
$\mathcal{Y}$ en $q_{\ast}$ et $\left\{  s_{\ast}\right\}  =\left(
\pi\left\vert _{\mathcal{B}}\right.  \right)  ^{-1}\left(  q_{\ast}\right)  $,
$g_{q}$ tend au sens des courants vers $\pi_{\ast}\widetilde{g}_{s_{\ast}}$
lorsque $q\in\mathcal{B}\backslash\left\{  q_{\ast}\right\}  $ tend vers
$q_{s}$.

Une fonction de Green principale pour une surface de Riemann nodale \`{a} bord
$\mathcal{Y}$ est une fonction de Green simple $g$ pour $\mathcal{Y}$ telle
que si $\mathcal{B}$ est une branche de bord de $\mathcal{Y}$, $g\left\vert
_{\mathcal{B}}\right.  $ se prolonge contin\^{u}ment \`{a} $\overline
{\mathcal{B}}$ par la valeur $0$ sur $\overline{\mathcal{B}}\cap b\mathcal{Y}%
$.\smallskip
\end{definition}

D\'{e}taillons maintenant la formule explicite de la proposition~17
de~\cite{HeG-MiV2014} \'{e}tablissant l'existence de fonctions Green pour une
famille \`{a} 1 param\`{e}tre de courbes complexes dont les singularit\'{e}s
\'{e}ventuelles sont quelconques. On se donne une courbe complexe
$\mathcal{Y}$ d'un ouvert de $\mathbb{C}^{2}$, $\Omega$ un voisinage de Stein
de $\mathcal{Y}$ dans $\mathbb{C}^{2}$, $\Phi$ une fonction holomorphe sur
$\Omega$ telle que $\mathcal{Y}=\left\{  \Phi=0\right\}  $ et $d\Phi\left\vert
_{\mathcal{Y}}\right.  \neq0$ puis un domaine strictement pseudoconvexe
$\Omega^{\ast}$ de $\mathbb{C}^{2}$ v\'{e}rifiant
\[
\mathcal{Y}_{0}=\mathcal{Y}\cap\Omega^{\ast}\subset\Omega,
\]
et enfin une fonction sym\'{e}trique $\Psi\in\mathcal{O}\left(  \Omega
\times\Omega,\mathbb{C}^{2}\right)  $ telle que pour tout $\left(
z,z^{\prime}\right)  \in\mathbb{C}^{2}$,%
\[
\Phi\left(  z^{\prime}\right)  -\Phi\left(  z\right)  =\left\langle
\Psi\left(  z^{\prime},z\right)  ,z^{\prime}-z\right\rangle
\]
o\`{u} $\left\langle v,w\right\rangle =v_{1}w_{1}+v_{2}w_{2}$ lorsque
$v,w\in\mathbb{C}^{2}$. On d\'{e}finit sur $\operatorname*{Reg}\mathcal{Y}$
une $\left(  1,0\right)  $-forme $\omega$ en posant%
\begin{align*}
\omega &  =\frac{-dz_{1}}{\partial\Phi/\partial z_{2}}~\text{sur}%
~\mathcal{Y}^{1}=\mathcal{Y}\cap\left\{  \partial\Phi/\partial z_{2}%
\neq0\right\} \\
\omega &  =\frac{+dz_{2}}{\partial\Phi/\partial z_{1}}~\text{sur}%
~\mathcal{Y}^{2}=\mathcal{Y}\cap\left\{  \partial\Phi/\partial z_{1}%
\neq0\right\}
\end{align*}
et on consid\`{e}re%
\[
k\left(  z^{\prime},z\right)  =\det\left[  \frac{\overline{z^{\prime}%
}-\overline{z}}{\left\vert z^{\prime}-z\right\vert ^{2}},\Psi\left(
z^{\prime},z\right)  \right]  .
\]
Lorsque $q_{\ast}\in\operatorname*{Reg}\mathcal{Y}_{0}$, \cite[prop.~17
]{HeG-MiV2014} \'{e}tablit que la formule
\begin{equation}
g_{q_{\ast}}\left(  q\right)  =\frac{1}{4\pi^{2}}\int_{q^{\prime}%
\in\mathcal{Y}_{0}}k\left(  q^{\prime},q\right)  k\left(  q_{\ast},q^{\prime
}\right)  ~i\omega\left(  q^{\prime}\right)  \wedge\overline{\omega}\left(
q^{\prime}\right)  . \label{F/ Green}%
\end{equation}
d\'{e}finit pour $\mathcal{Y}_{0}$ une fonction de Green au sens ci-dessus. En
outre, la preuve de~\cite[prop.~17 ]{HeG-MiV2014} apporte que si $q_{\ast}%
\in\operatorname*{Reg}\mathcal{Y}_{0}$
\[
\partial g_{q_{\ast}}=\widetilde{k}_{q_{\ast}}\omega
\]
o\`{u} $\widetilde{k}_{q_{\ast}}=\frac{1}{2\pi}k\left(  .,q_{\ast}\right)  $.
La proposition ci-dessous donne un compl\'{e}ment utile.

\begin{proposition}
\label{P/ GreenNodale}On suppose $\mathcal{Y}$ n'a que des singularit\'{e}s
nodales. Dans ce cas, la fonction $g$ d\'{e}finie par (\ref{F/ Green}) est une
fonction de Green simple pour $\mathcal{Y}$.
\end{proposition}

\begin{proof}
Commen\c{c}ons par prouver que $q_{\ast}$ \'{e}tant fix\'{e} dans
$\operatorname*{Reg}\mathcal{Y}_{0}$, $g_{q_{\ast}}$ se prolonge en fonction
harmonique usuelle le long des branches de $\mathcal{Y}_{0}\backslash\left\{
q_{\ast}\right\}  $. Puisque, $g_{q_{\ast}}$ est une distribution harmonique
sur $\mathcal{Y}_{0}\backslash\left\{  q_{\ast}\right\}  $, on sait
d\'{e}j\`{a} que $g_{q_{\ast}}\left\vert _{\operatorname*{Reg}\mathcal{Y}_{0}%
}\right.  $ est une fonction harmonique usuelle et gr\^{a}ce \`{a}~\cite[prop.
2]{HeG-MiV2012} que pour toute branche $\mathcal{B}$ de $\mathcal{Y}_{0}$ en
$q $, $g_{q_{\ast}}\left\vert _{\mathcal{B}}\right.  $ a au plus une
singularit\'{e} logarithmique isol\'{e}e en $q$ et donc que $\partial
g_{q_{\ast}}$ au plus un p\^{o}le simple en $q$. Fixons $q$ dans
$\operatorname*{Sing}\mathcal{Y}_{0}$ et $\mathcal{B}$ une branche de
$\mathcal{Y}_{0}$ en $q $. En diminuant $\mathcal{B}$ et en changeant
\'{e}ventuellement de coordonn\'{e}es, on se ram\`{e}ne au cas o\`{u} $q=0$ et
o\`{u} $\Phi$ est au voisinage de $0$ de la forme%
\[
\Phi\left(  z\right)  =\left(  z_{2}-\varphi\left(  z_{1}\right)  \right)
\Theta\left(  z\right)
\]
avec $\varphi$ holomorphe dans un disque $V=D\left(  0,r\right)  $ assez petit
et $\Theta\left\vert _{\mathcal{B}}\right.  $ ne s'annulant qu'en $0$. En
particulier, il existe une fonction $\theta$ holomorphe sur $V$ telle que
$\theta\left(  0\right)  \neq0$ et $\Theta\left(  z_{1},\varphi\left(
z_{1}\right)  \right)  =z_{1}^{\nu-1}\theta\left(  z_{1}\right)  $ lorsque
$z_{1}\in V$, $\nu$ d\'{e}signant le nombre de branches de $\mathcal{Y}_{0}$
en $q$. Sur $\mathcal{B}\backslash\left\{  q\right\}  $, on a donc
$\omega=\frac{dz_{1}}{\theta\left(  z_{1}\right)  z_{1}^{\nu-1}}$.
Consid\'{e}rons alors $\chi$ une $\left(  0,1\right)  $-forme \`{a} support
compact dans $\mathcal{B}$~; $\chi$ s'\'{e}crit donc $\xi d\overline{z_{1}}$
avec $\xi\in\mathcal{D}\left(  V\right)  $. D'o\`{u}, par d\'{e}finition,
\[
\left\langle \partial g_{q_{\ast}},\chi\right\rangle =\underset{\varepsilon
\downarrow0^{+}}{\lim}\int_{z_{1}\in V\backslash D\left(  0,\varepsilon
\right)  }\frac{\widehat{k}_{q_{\ast}}\left(  z_{1}\right)  \xi\left(
z_{1}\right)  }{\theta\left(  z_{1}\right)  z_{1}^{\nu-1}}idz_{1}\wedge
d\overline{z_{1}}%
\]
o\`{u} $\widehat{k}_{q_{\ast}}\left(  z_{1}\right)  =\widetilde{k}_{q_{\ast}%
}\left(  z_{1},\varphi\left(  z_{1}\right)  \right)  $. Ecrivons%
\[
\frac{\widehat{k}_{q_{\ast}}\left(  z_{1}\right)  \xi\left(  z_{1}\right)
}{\theta\left(  z_{1}\right)  }=\sum\limits_{\alpha+\beta<\nu-1}%
c_{\alpha,\beta}z_{1}^{\alpha}\overline{z_{1}}^{\beta}+\int_{0}^{1}%
\frac{\left(  1-t\right)  ^{\nu-2}}{\left(  \nu-2\right)  !}\left.  D^{\nu
-1}\left(  \widehat{k}_{q_{\ast}}\xi/\theta\right)  \right\vert _{tz_{1}%
}.z_{1}^{\nu-2}dt~idz_{1}\wedge d\overline{z_{1}}%
\]
o\`{u} l'expression $\left.  D^{p}f\right\vert _{w}.z^{p}$ se lit comme
\'{e}tant la valeur prise par la diff\'{e}rentielle totale d'ordre $p$ de $f$
en $w$ sur le vecteur $\left(  z,...,z\right)  $. Puisque $\int_{0}^{2\pi
}e^{i\theta\left(  \alpha-\beta-\nu+1\right)  }d\theta=0$ lorsque
$\alpha+\beta<\nu-1$, on obtient
\begin{equation}
\left\langle \partial g_{q_{\ast}},\chi\right\rangle =\int_{z_{1}\in V}%
\int_{0}^{1}\frac{\left(  1-t\right)  ^{\nu-2}}{\left(  \nu-2\right)
!}\left.  D^{\nu-1}\left(  \widehat{k}_{q_{\ast}}\xi/\theta\right)
\right\vert _{tz_{1}}.1^{\nu-1}dt~idz_{1}\wedge d\overline{z_{1}}
\label{F/ dg v1}%
\end{equation}
Par ailleurs, il existe $c\in\mathbb{C}$ et $h\in\mathcal{O}\left(  V\right)
$ telle que l'expression de $\partial g_{q_{\ast}}\left\vert _{\mathcal{B}%
}\right.  $ dans la coordonn\'{e}e $z_{1}$ soit $\frac{c}{z_{1}}dz_{1}%
+hdz_{1}$. D'o\`{u}%
\[
\left\langle \partial g_{q_{\ast}},\chi\right\rangle =\underset{\varepsilon
\downarrow0^{+}}{\lim}\int_{z_{1}\in V\backslash D\left(  0,\varepsilon
\right)  }\left(  \frac{c}{z_{1}}+h\left(  z_{1}\right)  \right)  \xi\left(
z_{1}\right)  idz_{1}\wedge d\overline{z_{1}}%
\]
Ecrivons $\xi\left(  z_{1}\right)  =\xi\left(  0\right)  +\xi_{1,0}z_{1}%
+\xi_{0,1}\overline{z_{1}}+\int_{0}^{1}\left(  1-t\right)  \left.  D^{2}%
\xi\right\vert _{tz_{1}}.z_{1}^{2}dt$. Il vient alors%
\begin{equation}
\left\langle \partial g_{q_{\ast}},\chi\right\rangle =\pi r^{2}\xi_{1,0}%
c+\int_{z_{1}\in V}h\left(  z_{1}\right)  \xi\left(  z_{1}\right)
idz_{1}\wedge d\overline{z_{1}} \label{F/ dg v2}%
\end{equation}
Comme (\ref{F/ dg v1}) ne comporte pas de d\'{e}rivation de la mesure de Dirac
en $0$, la comparaison avec (\ref{F/ dg v2}) impose $c=0$, ce qu'il fallait prouver.

Fixons maintenant $q_{s}\ $dans $\operatorname*{Sing}\mathcal{Y}$.
Consid\'{e}rons une branche $\mathcal{B}$ de $\mathcal{Y}$ en $q_{s}$
suffisamment petite pour qu'on puisse s'y donner une coordonn\'{e}e holomorphe
$z$ centr\'{e}e en $q_{s}$. La sym\'{e}trie de $g$ et ce qui pr\'{e}c\`{e}de
impliquent que lorsque $q_{\ast}\in\mathcal{B}$ tend vers $q_{s}$,
$g_{q_{\ast}}-\frac{1}{2\pi}\ln\left\vert z-z\left(  q_{\ast}\right)
\right\vert $ converge uniform\'{e}ment sur $\mathcal{B}$ vers une fonction
harmonique de la forme $g_{\mathcal{B},q_{s}}^{\mathcal{B}}-\frac{1}{2\pi}%
\ln\left\vert z\right\vert $ o\`{u} $g_{\mathcal{B},q_{s}}$ est harmonique sur
$\mathcal{B}\backslash\left\{  q_{s}\right\}  $. Pour la m\^{e}me raison, si
$\mathcal{B}^{\prime}$ est une autre branche de $\mathcal{Y}$ en $q_{s}$ et
$\mathcal{B}\ni q_{\ast}$ tend vers $q_{s}$, $g_{q_{\ast}}$ converge
uniform\'{e}ment sur $\mathcal{B}^{\prime}$ vers une une fonction harmonique
$g_{\mathcal{B},q_{s}}^{\mathcal{B}^{\prime}}$. Enfin, l'expression
explicite~(\ref{F/ Green}) de $g$ donne que lorsque $q_{\ast}\in\mathcal{B}$
tend vers $q_{s}$ et que $\mathcal{B}^{\prime}$ est une branche de
$\mathcal{Y}$ relativement compacte dans $\mathcal{Y}\backslash\left\{
q_{s}\right\}  $, $g_{q_{\ast}}$ converge uniform\'{e}ment sur $\mathcal{B}%
^{\prime}$ vers une fonction harmonique $g_{\mathcal{B},q_{s}}^{\mathcal{B}%
^{\prime}}$. Quand $\mathcal{B}^{\prime}$ parcoure l'ensemble des branches de
$\mathcal{Y}$, ces fonctions $g_{\mathcal{B},q_{s}}^{\mathcal{B}^{\prime}}$ se
recollent en une fonction $g_{\mathcal{B},q_{s}}$ qui est harmonique sur
$\mathcal{B}^{\prime}\backslash\left\{  q_{s}\right\}  $ pour toute branche
$\mathcal{B}^{\prime}$ de $\mathcal{Y}$, dont la restriction \`{a}
$\mathcal{B}$ pr\'{e}sente une singularit\'{e} logarithmique en $q_{s}$ et
telle que $g_{q_{\ast}}\underset{\mathcal{B}\ni q_{\ast}}{\rightarrow
}g_{\mathcal{B},q_{s}}$ au sens des courants.

Proc\'{e}dant ainsi pour tous les points singuliers de $\mathcal{Y}$, on
constate que $g$ est une fonction de Green simple pour $\mathcal{Y}$.
\end{proof}

\begin{corollary}
\label{C/ GreenNodale2}Soit $\left(  M,\sigma\right)  $ une structure de
conductivit\'{e} de dimension deux. On se donne, ce qui est toujours possible,
une structure de conductivit\'{e} $(\widetilde{M},\widetilde{\sigma})$ de
dimension $2$ qui prolonge simplement $\left(  M,\sigma\right)  $, ce qui
signifie que $M\subset\subset\widetilde{M}$, $\widetilde{\sigma}\left\vert
_{M}\right.  =\sigma$ et que $\widetilde{\sigma}\left\vert _{p}\right.
=Id_{T_{p}^{\ast}\overline{\widetilde{M}}}$ pour tout $p\in b\widetilde{M}$.
On note alors $F:\widetilde{M}\rightarrow\mathbb{C}^{2}$ l'application obtenue
en appliquant le th\'{e}or\`{e}me~\ref{T/ HS} \`{a} $(\widetilde{M}%
,\widetilde{\sigma})$, on pose $\mathcal{Y}=F\left(  \widetilde{M}\right)  $
et on fixe un voisinage de Stein $\Omega$ de $\mathcal{Y}$ dans $\mathbb{C}%
^{2}$. Enfin, $\mathcal{M}=F\left(  M\right)  $ \'{e}tant relativement
compacte dans $\mathcal{Y}$, on peut se donner un domaine strictement
pseudoconvexe $\Omega^{\ast}$ de $\mathbb{C}^{2}$ v\'{e}rifiant $\mathcal{M}%
\subset\subset\mathcal{Y}_{0}=\mathcal{Y}\cap\Omega^{\ast}\subset\Omega$. On
note $g$ la fonction d\'{e}finie par (\ref{F/ Green}). Alors, $F^{\ast
}g\left\vert _{\overline{M}\times\overline{M}\backslash\Delta_{M}}\right.  $
est une fonction de Green pour $\left(  M,c_{\sigma}\right)  $.
\end{corollary}

\begin{proof}
Puisque $F:M\rightarrow\mathcal{M}$ est une normalisation $\left(  c_{\sigma
},c_{Q}\right)  $-analytique, $h=F^{\ast}g$ est bien d\'{e}finie sur
$\overline{M}_{\operatorname{reg}}\times\overline{M}_{\operatorname{reg}%
}\backslash\Delta_{\overline{M}_{\operatorname{reg}}}$ o\`{u} $\overline
{M}_{\operatorname{reg}}=F^{-1}\left(  \operatorname*{Reg}\overline{Q}\right)
$, sym\'{e}trique et pour tout $x\in\overline{M}$, $h_{x}=h\left(  .,x\right)
$ est harmonique sur $\overline{M}_{\operatorname{reg}}\backslash
bM\cup\left\{  x\right\}  $, continue sur $\overline{M}_{\operatorname{reg}%
}\backslash\left\{  x\right\}  $ et v\'{e}rifie $i\partial^{\sigma}%
\overline{\partial^{\sigma}}h=\delta_{x}dV$ o\`{u} $dV$ est une forme volume
pour $M$. Lorsque $p\in F^{-1}\left(  \operatorname*{Sing}\mathcal{\overline
{\mathcal{Y}}}\right)  \cap M$ et $U$ est un voisinage ouvert connexe de $p$
dans $M$, $B=F\left(  U\right)  $ est une branche int\'{e}rieure de
$\overline{\mathcal{M}}$ en $q=F\left(  p\right)  $ et on peut poser
$h_{p}=F^{\ast}g_{p,B}$. La proposition~\ref{P/ GreenNodale} entra\^{\i}ne que
$h$ ainsi construite est une fonction de Green pour $M$.
\end{proof}

Le corollaire~\ref{C/ GreenNodale2} permet aussi de compl\'{e}ter le
th\'{e}or\`{e}me~\ref{T/ plgt1} en donnant une formule explicite $F^{\ast
}\theta_{c}^{\mathcal{M}}f_{\ast}u$.~

\begin{corollary}
\label{C/ GreenNodale3}Les hypoth\`{e}ses et les notations sont celles du
th\'{e}or\`{e}me~\ref{T/ plgt1} et du corollaire~\ref{C/ GreenNodale2}.
$\mathcal{M}$ admet une fonction de Green principale et si $g$ est une telle
fonction, pour tout $u\in C^{\infty}\left(  bM\right)  $, $F^{\ast}\theta
_{c}^{\mathcal{M}}f_{\ast}u$ est donn\'{e} par la formule%
\[
F^{\ast}\theta_{c}^{\mathcal{M}}f_{\ast}u=\left(  F^{\ast}\partial
\widehat{f_{\ast}u}\right)  \left\vert _{bM}\right.  ,~\widehat{f_{\ast}%
u}:\operatorname*{Reg}\overline{\mathcal{M}}\ni q\mapsto\left\vert
\begin{array}
[c]{l}%
\frac{i}{2}\int_{\partial\mathcal{M}}\left(  f_{\ast}u\right)  \overline
{\partial}g_{q}~si~q\in\mathcal{M}\\
\left(  f_{\ast}u\right)  \left(  q\right)  ~si~q\in b\mathcal{M}%
\end{array}
\right.
\]

\end{corollary}

\begin{proof}
Ceci d\'{e}coule directement du corollaire~\ref{C/ GreenNodale2} et du
proc\'{e}d\'{e} de construction des fonctions de Green principale rappel\'{e}
au d\'{e}but de cette section.
\end{proof}

\noindent\textbf{Remarque. }Le corollaire~\ref{C/ GreenNodale2} permet aussi
de d\'{e}terminer la fonction de Green principale de $\left(  M,c_{\sigma
}\right)  $ gr\^{a}ce \`{a} la m\'{e}thode d\'{e}crite au d\'{e}but de cette section.

Cette section fournit donc un proc\'{e}d\'{e} pour d\'{e}terminer
l'op\'{e}rateur $N_{d}^{c_{\sigma}}$ \`{a} partir de $N_{d}^{\sigma}$ et
ram\`{e}ne le probl\`{e}me initial \`{a} celui o\`{u} le coefficient de
conductivit\'{e} est constant et donc \`{a} la situation \'{e}tudi\'{e}e
dans~\cite{HeG-MiV2007, HeG-MiV2012} \`{a} laquelle sont consacr\'{e}es les
sections~\ref{S/ coeffCst} et~\ref{S/ RcoeffCst}.

\subsection{Conductivit\'{e} de coefficient constant\label{S/ coeffCst}}

Dans cette section, on revient sur l'unicit\'{e} de la structure complexe dans
le cas d'une conductivit\'{e} de coefficient constant. Comme il est not\'{e}
dans la section~\ref{S/ condaniso}, cela permet de compl\'{e}ter le
th\'{e}or\`{e}me de Henkin-Santacesaria ainsi que la preuve du
th\'{e}or\`{e}me~1 de \cite{HeG-MiV2007} qui est le th\'{e}or\`{e}me
ci-dessous pour $n=2$. Notons que l'hypoth\`{e}se faite sur les fonctions
$u_{j}$ est g\'{e}n\'{e}riquement v\'{e}rifi\'{e}e (voir~\cite{HeG-MiV2007,
HeG-MiV2012}).

\begin{theorem}
[Henkin-Michel, 2007]\label{T/ unicite}Soient $M$ et $M^{\prime}$, deux
surfaces de Riemann \`{a} bord lisses bord\'{e}es par la m\^{e}me courbe
r\'{e}elle $\gamma$. On pose $\partial=d-\overline{\partial}$ (resp.
$\partial^{\prime}=d-\overline{\partial^{\prime}}$), $\overline{\partial}$
(resp. $\overline{\partial^{\prime}}$) \'{e}tant l'op\'{e}rateur de
Cauchy-Riemann de $M$ (resp. $M^{\prime}$). Si $u\in C^{\infty}\left(
\gamma\right)  $, on note $\widetilde{u}$ (resp. $\widehat{u}$) le
prolongement harmonique de $u$ \`{a} $M$ (resp. $M^{\prime}$) et on pose
$\theta u=\left(  \partial\widetilde{u}\right)  \left\vert _{\gamma}\right.  $
(resp. $\theta^{\prime}u=\left(  \partial^{\prime}\widehat{u}\right)
\left\vert _{\gamma}\right.  $)~; $\theta$ (resp. $\theta^{\prime}$) est aussi
l'op\'{e}rateur $\theta_{c}^{\sigma}$ d\'{e}fini par (\ref{F/ theta}) lorsque
$\sigma$ est la conductivit\'{e} associ\'{e}e \`{a} la structure complexe de
$M$ (resp. $M^{\prime}$).

On se donne $u_{0},...,u_{n}\in C^{\infty}\left(  \gamma\right)  $ o\`{u}
$n\in\mathbb{N}^{\ast}$, on suppose que pour tout $j\in\left\{
0,...,n\right\}  $, $\theta u_{j}=\theta^{\prime}u_{j}$, l'application
$\left[  \theta u\right]  =\left[  \theta u_{0}:\cdots:\theta u_{n}\right]
=\left[  \theta^{\prime}u\right]  $ est bien d\'{e}finie, r\'{e}alise un
plongement de $\gamma$ dans $\left\{  \left[  w_{0}:\cdots:w_{n}\right]
\in\mathbb{CP}_{n};~w_{0}\neq0\right\}  $ et on suppose en outre que $\left[
\partial\widetilde{u}\right]  $ (resp. $\left[  \partial^{\prime}%
\widehat{u}\right]  $) est bien d\'{e}finie sur $M$ (resp. $M^{\prime}$) et
prolonge m\'{e}romorphiquement $\left[  \theta u\right]  $ (resp. $\left[
\theta^{\prime}u\right]  $) \`{a} $M$ (resp. $M^{\prime}$). Dans ces
conditions, il existe un isomorphisme de surfaces de Riemann \`{a} bord de
$\overline{M}$ sur $\overline{M^{\prime}}$ dont la restriction \`{a} $\gamma$
est l'identit\'{e}.
\end{theorem}

L'une des \'{e}tapes de la preuve de ce th\'{e}or\`{e}me
utilise~\cite{HeG-MiV2015} dont les lemmes~11 \`{a} 14 avaient \'{e}t\'{e}
\'{e}crits au d\'{e}part par l'auteur de ces lignes pour donner une preuve
compl\`{e}te du th\'{e}or\`{e}me ci-dessus.

On note $\left(  U_{\ell}\right)  $ et $\left(  U_{\ell}^{\prime}\right)  $
les extensions harmoniques de $u$ \`{a} $M$ et $M^{\prime}$ respectivement.
Par hypoth\`{e}se, $F=\left[  \partial U\right]  :\overline{M}\longrightarrow
\mathbb{CP}_{n}$ et $F^{\prime}=\left[  \partial U^{\prime}\right]
:\overline{M^{\prime}}\longrightarrow\mathbb{CP}_{n}$ sont bien d\'{e}finies,
co\"{\i}ncident sur $\gamma$ et $f=F\left\vert _{\gamma}\right.  =F^{\prime
}\left\vert _{\gamma}\right.  $ plonge $\gamma$ dans $\left\{  w_{0}%
\neq0\right\}  $ o\`{u} $w_{0},...,w_{n}$ sont les coordonn\'{e}es
homog\`{e}nes standards de $\mathbb{CP}_{n}$. On munit $\delta=f\left(
\gamma\right)  $ de l'orientation de $\gamma$ transport\'{e}e par $f$. Les
hypoth\`{e}ses de r\'{e}gularit\'{e} faites sur $M$ et $M^{\prime}$ impliquent
que $F$ et $F^{\prime}$ sont de classe $C^{1}$. On pose%
\begin{align*}
Y  &  =F\left(  M\right)  \backslash\delta,~\Gamma=F^{-1}\left(
\delta\right)  ,\\
\widetilde{M}  &  =M\backslash\Gamma,~~\widetilde{F}=F\left\vert
_{M\backslash\Gamma}^{\mathbb{CP}_{n}\backslash\delta}\right.  ,\\
\overline{M}_{r}  &  =\left\{  dF\neq0\right\}  ~~\&~~M_{s}=\left\{
dF=0\right\}
\end{align*}
Puisque $f$ est un plongement de $\gamma$ dans $\left\{  w_{0}\neq0\right\}  $
qui est isomorphe \`{a} $\mathbb{C}^{n}$, il existe un voisinage $G$ ouvert de
$\gamma$ dans $\overline{M}$ tel que $F_{G}=F\left\vert _{G}\right.  $ soit un
plongement de $G$ dans $\mathbb{C}^{2}$; l'orientation de $\delta$ est donc
aussi celle induite par celle naturelle de $G$. Lorsque $A$ est un espace
topologique, on note $CC\left(  A\right)  $ l'ensemble des composantes
connexes de $A$. Si $A\subset\overline{M}$ et $B\subset F\left(  A\right)  $,
on note $\nu\left(  F,A,B\right)  $ le degr\'{e} de $F\left\vert _{A}%
^{B}\right.  $ quand il existe. On adopte pour $M^{\prime}$ des notations
similaires \`{a} celle prises pour $M$. La notation $\mathcal{D}_{p,q}\left(
U\right)  $ d\'{e}signe l'espace des $\left(  p,q\right)  $-formes de classe
$C^{\infty}$ \`{a} support compact dans un ouvert $U$ d'une vari\'{e}t\'{e}
complexe. La notation $\mathcal{H}^{d}\left(  E\right)  $ d\'{e}signe la
mesure $d$-dimensionnelle de Hausdorff d'un ensemble $E$ quand ceci a un sens.

\begin{lemma}
\label{L CC}$\Gamma\backslash\gamma$ est un compact de $M$ et $Y$\textit{\ est
une courbe complexe de }$\mathbb{CP}_{n}\backslash\delta$.
\end{lemma}

\begin{proof}
Puisque $F_{G}$ plonge $G$ dans $\mathbb{C}^{2}$, $\Gamma\cap G=\gamma$ et
$\Gamma\backslash\gamma=\Gamma\cap\left(  \overline{M}\backslash G\right)  $
est un compact contenu dans $M$. En particulier, $\widetilde{M}=M\backslash
\Gamma$ est une surface de Riemann ouverte. Par construction, $\widetilde{F}$
est propre car si $L$ est un compact de $\mathbb{CP}_{n}\backslash\delta$,
$\widetilde{F}{}^{-1}\left(  L\right)  $ est un compact de $\overline{M}$ qui
ne rencontre pas $\Gamma$ et donc est un compact de $\widetilde{M}$. Par un
th\'{e}or\`{e}me de Remmert, inutile dans le cas tr\`{e}s simple $n=1$,
$Y=\widetilde{F}(\widetilde{M})$ est un sous-ensemble analytique de
$\mathbb{CP}_{n}\backslash\delta$.
\end{proof}

\begin{lemma}
\label{L Struct-Img}$F_{\ast}\left[  M\right]  $ est un courant normal positif
port\'{e} par $\overline{Y}$ et $dF_{\ast}\left[  M\right]  =\left[
\delta\right]  $.
\end{lemma}

\begin{proof}
Si $\chi$ est une forme lisse \`{a} support compact de $\mathbb{CP}_{n}$,%
\[
\left\langle F_{\ast}\left[  M\right]  ,\chi\right\rangle =\int_{M}F^{\ast
}\chi.
\]
$F_{\ast}\left[  M\right]  $ est donc un courant de bidegr\'{e} $\left(
1,1\right)  $ port\'{e} par $\overline{F\left(  M\right)  }$ c'est-\`{a}-dire
$\overline{Y}$. Il est positif car si $\chi\in\mathcal{D}_{1,1}\left(
\mathbb{CP}_{n}\right)  $ est positive, $\left(  F^{\ast}\chi\right)
\left\vert _{M}\right.  $ est une $\left(  1,1\right)  $-forme positive de $M$
car $F$ est holomorphe et donc $\left\langle F_{\ast}\left[  M\right]
,\chi\right\rangle \geqslant0$. Soit $\xi\in C^{\infty}\left(  \mathbb{CP}%
_{n}\right)  $ tel que $\chi=\xi\omega_{FS}$ o\`{u} $\omega_{FS}=\frac{i}%
{2\pi}\partial\overline{\partial}\ln\left\vert w\right\vert ^{2}$ est la
$\left(  1,1\right)  $-forme qui induit la m\'{e}trique de Fubini-Study. On a
alors
\[
\left\vert \left\langle F_{\ast}\left[  M\right]  ,\chi\right\rangle
\right\vert \leqslant\int_{M}\left\vert \xi\right\vert F^{\ast}\omega
_{FS}\leqslant\left\Vert \xi\right\Vert _{\infty}\int_{M}F^{\ast}\omega_{FS}%
\]
Comme $\left\Vert \chi\right\Vert =\underset{p\in\mathbb{CP}_{n}}{\sup
}\left\Vert \chi_{p}\right\Vert $ et%
\begin{align*}
\left\Vert \chi_{p}\right\Vert  &  =\underset{s,t\in T_{p}\mathbb{CP}%
_{n},~\left\Vert s\right\Vert _{FS}=\left\Vert t\right\Vert _{FS}=1}{\max
}\left\vert \chi_{p}.\left(  s,t\right)  \right\vert \\
&  =\left\vert \xi\left(  p\right)  \right\vert \underset{s,t\in
T_{p}\mathbb{CP}_{n},~\left\Vert s\right\Vert _{FS}=\left\Vert t\right\Vert
_{FS}=1}{\max}\left\vert \left(  \omega_{FS}\right)  _{p}.\left(  s,t\right)
\right\vert =\left\vert \xi\left(  p\right)  \right\vert ,
\end{align*}
on obtient que la masse de $F_{\ast}\left[  M\right]  $ est finie et au plus
$\int_{M}F^{\ast}\omega_{FS}$. Si $\chi\in\mathcal{D}\left(  \mathbb{CP}%
_{n}\right)  $,
\[
\left\langle dF_{\ast}\left[  M\right]  ,\chi\right\rangle =\left\langle
F_{\ast}\left[  M\right]  ,d\chi\right\rangle =\int_{M}F^{\ast}d\chi=\int%
_{M}dF^{\ast}\chi=\int_{\gamma}F^{\ast}\chi=\left\langle F_{\ast}\left[
\gamma\right]  ,\chi\right\rangle
\]
Autrement dit, $dF_{\ast}\left[  M\right]  =F_{\ast}\left[  \gamma\right]
=\left[  \delta\right]  $. En particulier, la masse de $dF_{\ast}\left[
M\right]  $ est finie$~$; $F_{\ast}\left[  M\right]  $ est un courant normal
port\'{e} par $\overline{Y}$.
\end{proof}

\begin{lemma}
\label{L CourantPos}$F_{\ast}\left[  M\right]  \left\vert _{\mathbb{CP}%
_{n}\backslash\delta}\right.  $ est une cha\^{\i}ne holomorphe positive de
$\mathbb{CP}_{n}\backslash\delta$ port\'{e}e par $Y$.
\end{lemma}

\begin{proof}
Etant donn\'{e} que $T=F_{\ast}\left[  M\right]  $ est port\'{e} par
$\overline{Y}$ et que $Y=\overline{Y}\backslash\delta$, $S=T\left\vert
_{\mathbb{CP}_{n}\backslash\delta}\right.  $ est un courant normal et donc
localement rectifiable de $\mathbb{CP}_{n}\backslash\delta$, sans bord et
port\'{e} par $Y$. D'apr\`{e}s le th\'{e}or\`{e}me de structure~2.1
de~\cite{HaR1977}, il existe donc $\left(  n_{j}\right)  _{1\leqslant
j\leqslant N}\in\mathbb{Z}^{\mathbb{N}}$ tel que $S=%
{\displaystyle\sum\limits_{1\leqslant j\leqslant N}}
n_{j}\left[  Y_{j}\right]  $ o\`{u} $\left(  Y_{j}\right)  $ est la famille
des composantes irr\'{e}ductibles de $Y$. $S$ \'{e}tant par ailleurs un
courant positif d'apr\`{e}s le lemme~\ref{L Struct-Img}, les $n_{j}$ sont des
entiers naturels.
\end{proof}

\begin{lemma}
\label{L UniciteIm}$F_{\ast}\left[  M\right]  =F_{\ast}\left[  M^{\prime
}\right]  $ et $Y^{\prime}=Y$.
\end{lemma}

\begin{proof}
D'apr\`{e}s le lemme~\ref{L Struct-Img}, le courant $T=F_{\ast}\left[
M\right]  -F_{\ast}^{\prime}\left[  M^{\prime}\right]  $ est un courant normal
sans bord de bidegr\'{e} $\left(  1,1\right)  $ port\'{e} par $\overline
{Y}\cup\overline{Y^{\prime}}$. Il est par cons\'{e}quent de la forme $%
{\displaystyle\sum\limits_{1\leqslant j\leqslant N}}
n_{j}\left[  Z_{j}\right]  $ o\`{u} $\left(  n_{j}\right)  \in\left(
\mathbb{Z}^{\ast}\right)  ^{\mathbb{N}} $ et les $Z_{j}$ sont des courbes
complexes compactes irr\'{e}ductibles de $\mathbb{CP}_{n}$ contenues dans
$\overline{Y}\cup\overline{Y^{\prime}}$. Soit $Z$ l'une de ces courbes.
$Z\cap\delta\neq\varnothing$ car sinon $F^{-1}\left(  Z\right)  $ est une
courbe complexe compacte contenue dans $M$ ou $M^{\prime}$, ce qui est exclu.
L'une des composantes connexes $\beta$ de $\delta$ est donc contenue dans $Z$;
on munit $\beta$ de l'orientation induite par $\delta$. $\beta$ \'{e}tant
lisse, il existe dans $Z$ une surface de Riemann (lisse) $B$ telle que
$B\backslash\beta$ est contenue dans $\left(  \mathbb{CP}_{n}\backslash
\delta\right)  \cap\operatorname*{Reg}\overline{Y}\cap\operatorname*{Reg}%
\overline{Y^{\prime}}$ et n'a que deux composantes connexes $B^{-}$ et $B^{+}$.

Par construction, $B^{-}$ est une surface de Riemann ouverte connexe contenue
dans la courbe complexe $Y\cup Y^{\prime}$ et donc l'un au moins des deux
nombres $\mathcal{H}^{2}\left(  B^{-}\cap Y\right)  $ ou $\mathcal{H}%
^{2}\left(  B^{-}\cap Y^{\prime}\right)  $ est strictement positif, par
exemple $\mathcal{H}^{2}\left(  B^{-}\cap Y\right)  >0$. Puisque $B^{-}$ est
connexe, ceci implique$^{(}$\footnote{$\bigskip$Puisque $B^{-}\cap
\delta=\varnothing$, $B^{-}=\left(  B^{-}\cap Y\right)  \cup\left(
B^{-}\backslash\overline{Y}\right)  $. $B^{-}\cap Y$ est un ouvert de $B^{-}$
car par construction, $B^{-}\subset\operatorname*{Reg}\overline{Y}%
\cap\operatorname*{Reg}\overline{Y^{\prime}}$. Il est non vide par
hypoth\`{e}se. Donc $B^{-}=B^{-}\cap Y\subset Y$.}$^{)}$ que $B^{-}\subset Y$.
Etant donn\'{e} que $\beta$ est contenu dans les bords de $Y$ et $B$, on en
d\'{e}duit quitte \`{a} diminuer $B$, que $Y\cap B\subset Z$ et donc que
$Y\cap B\subset B^{-}\cup B^{+}$.

Supposons que $\mathcal{H}^{2}\left(  B^{+}\cap Y\right)  =0$. Alors, comme
$B\subset\operatorname*{Reg}\overline{Y}$, $B^{+}\cap Y=\varnothing$, $Y\cap
B=B^{-}$ et, par force, $B^{+}\subset Y^{\prime}$. Supposons en outre que
$\mathcal{H}^{2}\left(  B^{-}\cap Y^{\prime}\right)  =0$, alors, quitte \`{a}
diminuer $B$, on a de la m\^{e}me fa\c{c}on qu'auparavant $Y^{\prime}\cap
B=B^{+}$ et donc $d\left[  Y\right]  =-d\left[  Y^{\prime}\right]  $ pr\`{e}s
de $\beta$. Ceci est incompatible avec le fait que $F_{\ast}\left[  M\right]
$ et $F_{\ast}^{\prime}\left[  M^{\prime}\right]  $ sont deux cha\^{\i}nes
holomorphes positives de $\mathbb{CP}_{n}\backslash\delta$ port\'{e}es
respectivement par $Y$ et $Y^{\prime}$. Par cons\'{e}quent, $\mathcal{H}%
^{2}\left(  B^{-}\cap Y^{\prime}\right)  >0$ et donc $B^{-}\subset Y^{\prime}
$. D'o\`{u} $B\subset Y^{\prime}$ puis $Z\subset Y^{\prime}$, ce qui est de
nouveau une contradiction. Revenant \`{a} notre premi\`{e}re supposition, on
en d\'{e}duit que $\mathcal{H}^{2}\left(  B^{+}\cap Y\right)  >0$ et donc
$B\subset Y$, ce qui est de nouveau absurde. Le lemme est prouv\'{e}.
\end{proof}

\begin{lemma}
\label{L/ degre}Lorsque $y\in\overline{Y}$, $\overline{M_{y}}=F^{-1}\left(
\left\{  y\right\}  \right)  $ est un ensemble fini et $\nu:\overline
{Y}:y\mapsto\operatorname{Card}\overline{M}_{y}$ est born\'{e}e.
\end{lemma}

\begin{proof}
Supposons que $F^{-1}\left(  \left\{  y\right\}  \right)  $ est infini pour un
un certain $y\in\overline{Y}$. Si $F^{-1}\left(  \left\{  y\right\}  \right)
$ poss\`{e}de un point d'accumulation dans $M$, $F=y$ sur une composante
connexe de $M$ et donc sur un ouvert non vide de $\gamma$. Dans le cas
contraire, $F^{-1}\left(  \left\{  y\right\}  \right)  $ poss\`{e}de un point
d'accumulation dans $\gamma$ et $dF$ s'annule en ce point. Dans les deux cas,
ceci contredit que $F\left\vert _{\gamma}\right.  $ est un plongement.

Supposons que $\nu$ ne soit pas born\'{e}e. Il existe alors $\left(
y_{m}\right)  \in\overline{Y}^{\mathbb{N}}$ telle que $\left(  \nu_{m}\right)
=\left(  \nu\left(  y_{m}\right)  \right)  $ a $+\infty$ comme limite et
$\left(  y_{m}\right)  $ converge vers $y_{\ast}\in\overline{Y}$. Puisque
$\overline{M}$ est compacte, il existe dans $\overline{M}^{\mathbb{N}}$ une
suite convergente de limite $x_{\ast}^{0}\in F^{-1}\left(  \left\{  y_{\ast
}\right\}  \right)  $ et une extractrice $\varphi:\mathbb{N}\rightarrow
\mathbb{N}$ telle que $y_{\varphi\left(  m\right)  }=F\left(  x_{m}\right)  $
pour tout $m\in\mathbb{N}$. Si $dF\left\vert _{x_{\ast}^{0}}\right.  \neq0$,
il existe un voisinage ouvert $U_{0}$ de $x_{\ast}^{0}$ dans $\overline{M}$
tel que $V_{0}=F\left(  U_{0}\right)  $ est une surface de Riemann (\`{a} bord
si $x_{\ast}^{0}\in\gamma$) et $F\left\vert _{U_{0}}^{V_{0}}\right.  $ est un
biholomorphisme (de surfaces de Riemann \`{a} bord si $x_{\ast}^{0}\in\gamma
$)~; on pose $m_{\ast}^{0}=1$ dans ce cas. Si $dF\left\vert _{x_{\ast}^{0}%
}\right.  =0$, $x_{0}^{\ast}\notin\gamma$ et on peut choisir des
coordonn\'{e}es holomorphes $\left(  \zeta_{1},...,\zeta_{n}\right)  $ pour
$\mathbb{CP}_{n}$ au voisinage de $y_{\ast}$ tel que l'ordre d'annulation
$m_{\ast}$ de $\left(  d\left(  \zeta_{1}\circ F\right)  ,...,d\left(
\zeta_{n}\circ F\right)  \right)  $ en $x_{\ast}^{0}$ est aussi celui de
$d\left(  \zeta_{1}\circ F\right)  $ en $x_{\ast}^{0}$. Dans ce cas, il existe
un voisinage ouvert $U_{0}$ de $x_{\ast}^{0}$ dans $M$ tel que si $y\in
V_{0}=F\left(  U_{0}\right)  $, $\zeta_{1}\left(  F\left(  y\right)  \right)
$ a exactement $m_{\ast}^{0}$ ant\'{e}c\'{e}dents par $\zeta_{1}\circ F$ dans
$U_{0}$, deux \`{a} deux distincts si $y\neq y_{\ast}$~; si $y\in
V_{0}=F\left(  U_{0}\right)  $, $y$ a donc au moins un ant\'{e}c\'{e}dent par
$F$ dans $U_{0}$ et au plus $m_{\ast}^{0}$.

Supposons que nous disposons dans $F^{-1}\left(  y_{\ast}\right)  $ de $k+1$
points deux \`{a} deux distincts $x_{\ast}^{0},...,x_{\ast}^{k}$ et de
voisinages ouverts $U_{0},...,U_{k}$ de ces points dans $\overline{M}$ tels
que pour tout $j\in\left\{  1,...,k\right\}  $, $1\leqslant\operatorname{Card}%
F^{-1}\left(  y_{\ast}\right)  \cap U_{j}\leqslant m_{\ast}^{j}$ et
$U_{j}\subset\overline{M}\backslash V_{j-1}$ o\`{u} $V_{j-1}%
=\underset{1\leqslant\ell\leqslant j-1}{\cup}U_{\ell}$. Alors
$\operatorname{Card}F^{-1}\left(  y_{\ast}\right)  \cap V_{k}\leqslant
\sum\limits_{0\leqslant j\leqslant k}m_{\ast}^{j}$ et puisque $\overline
{M}\backslash V_{k+1}$ est compact, on peut donc trouver une extractrice
$\varphi:\mathbb{N}\rightarrow\mathbb{N}$ telle que pour tout $m\in\mathbb{N}%
$, $F^{-1}\left(  y_{\varphi\left(  m\right)  }\right)  \cap\left(
\overline{M}\backslash V_{k+1}\right)  $ contient au moins un point
$x_{m}^{k+1}$ qui, lorsque $m$ tend vers $+\infty$ , tend vers un point
$x_{\ast}^{k+1}\in F^{-1}\left(  \left\{  y_{\ast}\right\}  \right)  $. Comme
pr\'{e}c\'{e}demment, on peut alors trouver un entier $m_{\ast}^{k+1}$ et un
voisinage $U_{k+1}$ de $x_{\ast}^{k+1}$ dans $\overline{M}$ tels que
$1\leqslant\operatorname{Card}F^{-1}\left(  y_{\ast}\right)  \cap
U_{k}\leqslant m_{\ast}^{k+1}$.

On construit ainsi par r\'{e}currence une suite $\left(  x_{\ast}^{k}\right)
_{k\in\mathbb{N}}$ de points deux \`{a} deux distincts de $\overline{M}_{y}$,
ce qui est impossible. $\nu$ est donc born\'{e}e.
\end{proof}

\begin{lemma}
\label{L/ hol1}Soit $h\in\mathcal{O}\left(  M\right)  \cap C^{0}\left(
\overline{M}\right)  $. Alors $F_{\ast}h$ est holomorphe et born\'{e}e sur
$\operatorname*{Reg}Y$ et $F^{\prime\ast}F_{\ast}h=\left(  F_{\ast}h\right)
\circ F^{\prime}\in\mathcal{O}\left(  M^{\prime}\right)  \cap C^{0}\left(
\overline{M^{\prime}}\right)  $
\end{lemma}

\begin{proof}
Par d\'{e}finition $F_{\ast}h$ est la fonction d\'{e}finie sur $Y$ par
$\left(  F_{\ast}h\right)  \left(  y\right)  =\sum\limits_{x\in F^{-1}\left(
y\right)  }h\left(  x\right)  $. Soit $y_{\ast}\in\left(  \operatorname*{Reg}%
Y\right)  \backslash F\left(  \left\{  dF=0\right\}  \right)  $. Posons
$F^{-1}\left(  y_{\ast}\right)  =\left\{  x_{\ast1},...x_{\ast k}\right\}  $
o\`{u} $k=\nu\left(  y\right)  $. Il existe un voisinage $B$ de $y$ dans
$\operatorname*{Reg}Y$ tel que pour tout $j\in\left\{  1,...,k\right\}  $, il
existe un voisinage $A_{j}$ de $x_{\ast j}$ dans $M$ pour lequel
$F_{j}=F\left\vert _{A_{j}}^{B}\right.  $ est un biholomorphisme. Supposons
que $\left(  y_{\nu}\right)  \in B^{\mathbb{N}}$ converge vers $y_{\ast}$ et
$\operatorname{Card}F^{-1}\left\{  y_{n}\right\}  \geqslant k$ pour tout $n$.
Pour chaque $n\in\mathbb{N}$, il existe donc $a_{n}\in M\backslash\left\{
F_{1}^{-1}\left(  y_{n}\right)  ,...,F_{k}^{-1}\left(  y_{n}\right)  \right\}
$ tel que $F\left(  a_{n}\right)  =y_{n}$. Quitte \`{a} consid\'{e}rer une
sous-suite, $\left(  a_{n}\right)  $ converge vers un point $a$ de
$\overline{M}$ qui v\'{e}rifie $F\left(  a\right)  =y_{\ast}$. Etant donn\'{e}
que $y\in Y=F\left(  M\right)  \backslash F\left(  bM\right)  $, $a\notin bM$
et il existe $j\in\left\{  1,...,k\right\}  $ tel que $a=x_{\ast j}$. Pour $n$
assez grand, $a_{n}$ et $F_{j}^{-1}\left(  y_{n}\right)  $ sont alors deux
points distincts de $A_{j}$ qui ont la m\^{e}me image par $F$, \`{a} savoir
$y_{n}$. Ceci est absurde. Donc, $F_{\ast}h=\sum\limits_{1\leqslant j\leqslant
k}h\circ F_{j}^{-1}$ est holomorphe au voisinage de. De plus, $\left\vert
F_{\ast}h\right\vert \leqslant k\left\Vert h\right\Vert _{\infty}$ et
$k=\nu\left(  y\right)  $. $F_{\ast}h$ est donc born\'{e}e d'apr\`{e}s le
lemme~\ref{L/ degre}. Etant donn\'{e} que $\left(  \operatorname*{Reg}%
Y\right)  \cap F\left(  \left\{  dF=0\right\}  \right)  $ est fini, $F_{\ast
}h$ se prolonge holomorphiquement \`{a} $\operatorname*{Reg}Y$. Ceci implique
que $F^{\prime\ast}F_{\ast}h=\left(  F_{\ast}h\right)  \circ F^{\prime}$ est
holomorphe et est born\'{e}e sur $M^{\prime}\backslash F^{\prime-1}\left(
\operatorname*{Sing}Y\right)  $. Comme $F^{\prime-1}\left(
\operatorname*{Sing}Y\right)  $ est un ensemble fini, $F^{\prime\ast}F_{\ast
}h$ se prolonge holomorphiquement \`{a} $M^{\prime}$.
\end{proof}

\begin{lemma}
\label{L/ hol3}Si $\omega^{\prime}\in C^{1,0}\left(  \overline{M^{\prime}%
}\right)  \cap\Omega^{1,0}\left(  M^{\prime}\right)  $, il existe $\omega\in
C^{1,0}\left(  \overline{M}\right)  \cap\Omega^{1,0}\left(  M\right)  $ telle
que $\omega\left\vert _{\gamma}\right.  =\omega^{\prime}\left\vert _{\gamma
}\right.  $.
\end{lemma}

\begin{proof}
Il s'agit de constater que $\omega^{\prime}\left\vert _{\gamma}\right.  $
v\'{e}rifie la condition des moments quand $\gamma$ est vue comme le bord de
$M$. Soit donc $h\in\mathcal{O}\left(  M\right)  \cap C^{0}\left(
\overline{M}\right)  $. D'apr\`{e}s le lemme~\ref{L/ hol1}, $g=F^{\prime\ast
}F_{\ast}h\in\mathcal{O}\left(  M^{\prime}\right)  \cap C^{0}\left(
\overline{M^{\prime}}\right)  $. Puisque $f_{\ast}\left[  \gamma\right]
=\left[  \delta\right]  $,
\begin{align*}
\int_{\gamma}h\omega^{\prime}  &  =\int_{\gamma}F^{\ast}F_{\ast}\left(
h\omega^{\prime}\right)  =\int_{\delta}F_{\ast}\left(  h\omega^{\prime}\right)
\\
&  =\int_{\gamma}\left(  F^{\prime\ast}F_{\ast}\right)  \left(  h\omega
^{\prime}\right)  =\int_{M^{\prime}}d\left(  F^{\prime\ast}F_{\ast}\right)
\left(  h\omega^{\prime}\right)  =0.
\end{align*}
car $F^{\prime\ast}F_{\ast}h\in\mathcal{O}\left(  M^{\prime}\right)  \cap
C^{0}\left(  \overline{M^{\prime}}\right)  $ et $\omega^{\prime}\in
\Omega^{1,0}\left(  M^{\prime}\right)  $.
\end{proof}

\begin{proof}
[Preuve du th\'{e}or\`{e}me~\ref{T/ unicite}]\ref{T/ unicite} Puisque par
hypoth\`{e}se $\left[  \left(  \partial U_{\ell}\right)  _{0\leqslant
\ell\leqslant n}\right]  $ est une application bien d\'{e}finie de
$\overline{M}$ dans $\mathbb{CP}_{n}$, on peut utiliser le lemme~12
d'adjonction de~\cite{HeG-MiV2015} qui bien, qu'\'{e}crit pour le cas
particulier $n=2$, s'applique sans changement pour $n$ quelconque dans
$\mathbb{N}^{\ast}$~: il existe des fonctions $U_{n+1},...,U_{N}$ harmoniques
sur $M$ et continues sur $\overline{M}$ telles que $\left[  \left(  \partial
U_{\ell}\right)  _{0\leqslant\ell\leqslant N}\right]  $ est un plongement de
$M$ dans $\mathbb{CP}_{N}$. De m\^{e}me, il existe des fonctions
$U_{N+1}^{\prime},...,U_{N^{\prime}}^{\prime}$ harmoniques sur $M^{\prime}$ et
continues sur $\overline{M^{\prime}}$ telles que $\left[  \left(  \partial
U_{\ell}^{\prime}\right)  _{\ell\in\left\{  0,..,n,N+1,..,N^{\prime}\right\}
}\right]  $ est un plongement de $M^{\prime}$ dans $\mathbb{CP}_{n+N^{\prime
}-N}$. Lorsque $\ell\in\left\{  N+1,...,N+N^{\prime}\right\}  $, le
lemme~\ref{L/ hol3} donne que $\left(  \partial U_{\ell}^{\prime}\right)
\left\vert _{\gamma^{\prime}}\right.  $ se prolonge \`{a} $M$ en une $\left(
1,0\right)  $-forme holomorphe $\Sigma_{\ell}$. De m\^{e}me, lorsque $\ell
\in\left\{  n+1,...,N\right\}  $, $\left(  \partial U_{\ell}\right)
\left\vert _{\gamma}\right.  $ se prolonge\ \`{a} $M^{\prime}$ en une $\left(
1,0\right)  $-forme holomorphe $\Sigma_{\ell}^{\prime}$. Consid\'{e}rons
alors
\begin{align*}
\Sigma &  =\left(  \partial U_{0},...,\partial U_{n},\partial U_{n+1}%
...,\partial U_{N},\Sigma_{N+1},...,\Sigma_{N+N^{\prime}}\right)
\overset{d\acute{e}f}{=}\left(  \Sigma_{\ell}\right)  _{0\leqslant
\ell\leqslant L}\\
\Sigma^{\prime}  &  =\left(  \partial U_{0}^{\prime},..,\partial U_{n}%
^{\prime},\Sigma_{n+1}^{\prime},...,\Sigma_{N^{\prime}}^{^{\prime}},\partial
U_{N+1}^{\prime},...,\partial U_{N+N^{\prime}}^{\prime}\right)
\overset{d\acute{e}f}{=}\left(  \Sigma_{\ell}^{\prime}\right)  _{0\leqslant
\ell\leqslant L}%
\end{align*}
Par construction $\Sigma$ et $\Sigma^{\prime}$ co\"{\i}ncident sur $\gamma$.
Notons $\left(  w_{\ell}\right)  _{0\leqslant\ell\leqslant L}$ les
coordonn\'{e}es naturelles de $\mathbb{C}^{L+1}$. Lorsque $0\leqslant
\ell_{\ast}\leqslant n$, $\left[  \Sigma\right]  \left\vert _{\left\{
\partial U_{\ell}\neq0\right\}  }\right.  $ s'\'{e}crit $\left(  \partial
U_{\ell}/\partial U_{\ell_{\ast}}\right)  _{\ell\neq\ell^{\ast}}$ dans les
coordonn\'{e}es naturelles de $\mathbb{C}^{L}$ identifi\'{e} \`{a} $\left\{
w_{\ell_{\ast}}\neq0\right\}  $. Notons $p_{\ell_{\ast}}$ la projection
naturelle de $\mathbb{C}^{L}$ sur $\mathbb{C}^{N}$, $\left(  z_{\ell}\right)
_{\ell\neq\ell_{\ast}}\mapsto\left(  z_{\ell}\right)  _{0\leqslant
\ell\leqslant N\text{,~}\ell\neq\ell_{\ast}}$. L'application $\left(  \partial
U_{\ell}/\partial U_{\ell_{\ast}}\right)  _{0\leqslant\ell\leqslant
N\text{,~}\ell\neq\ell_{\ast}}$ est par construction un plongement de
$\left\{  \partial U_{\ell}\neq0\right\}  $ dans $\mathbb{C}^{N}$. $\left[
\Sigma\right]  $ est par ailleurs injective car $\overline{M}%
=\underset{0\leqslant\ell\leqslant n}{\cup}\left\{  \partial U_{\ell}%
\neq0\right\}  $ et car une relation de la forme $\left[  \Sigma\right]
\left(  x\right)  =\left[  \Sigma\right]  \left(  y\right)  $ impose
$y\in\underset{\left(  \partial U_{\ell}\right)  _{x}\neq0}{\cap}\left\{
\partial U_{\ell}\neq0\right\}  $. $\left[  \Sigma\right]  $ est donc un
plongement de $\overline{M}$ dans $\mathbb{CP}_{L}$. De m\^{e}me, $\left[
\Sigma^{\prime}\right]  $ est un plongement de $M^{\prime}$ dans
$\mathbb{CP}_{L}$. Remarquant que la preuve du lemme~\ref{L Struct-Img}
n'utilise pas que $F$ est une application canonique, c'est-\`{a}-dire de la
forme $\left[  \partial U\right]  $, ou en utilisant le lemme~8 de
\cite{HeG-MiV2015} qui montre que $\Sigma$ et $\Sigma^{\prime}$ sont
forc\'{e}ment de ce type, on en conclut que $\Sigma\left(  M\right)
=\Sigma^{\prime}\left(  M^{\prime}\right)  $ puis que $M$ et $M^{\prime}$ sont
rendues isomorphes par une application dont la restriction \`{a} $\gamma$ est l'identit\'{e}.
\end{proof}

\section{Reconstruction}

\subsection{Conductivit\'{e} isotrope\label{S/ ReconsCondiso}}

Lorsque $M$ est un domaine de $\mathbb{R}^{2}$ est muni de la structure
complexe induite par la m\'{e}trique euclidienne standard de $\mathbb{R}^{2}$,
c'est-\`{a}-dire quand $M$ est muni d'une conductivit\'{e} isotrope $\sigma$,
on sait que $\sigma$ est enti\`{e}rement d\'{e}termin\'{e}e par son
op\'{e}rateur de Dirichlet-Neumann. Cette unicit\'{e} est \'{e}tablie pour une
conductivit\'{e} r\'{e}elle analytique dans \cite{KoR-VoM1984}. Pour une
conductivit\'{e} isotrope lisse, un proc\'{e}d\'{e} de reconstruction
effective a \'{e}t\'{e} donn\'{e} dans~\cite{NoR1988} par Novikov et pour une
conductivit\'{e} de classe $L^{\infty}$ par Nachman dans~\cite{NaA1996}. Une
autre preuve de ce r\'{e}sultat a \'{e}t\'{e} \'{e}crite par Gutarts dans
\cite{GuB2007} pour une conductivit\'{e} lisse. Lorsque $M$ est une surface de
Riemann connexe dont le genre est connu, Henkin et Novikov dans~\cite[th.
1.2]{HeG-NoR2011}~g\'{e}n\'{e}ralisent et corrigent les r\'{e}sultats de
reconstruction d'une conductivit\'{e} isotrope de~\cite{HeG-MiV2008}. L'aspect
n\'{e}cessairement technique du principal r\'{e}sultat de~\cite[th.
1.2]{HeG-NoR2011} ne nous permet de n'en citer ici qu'un r\'{e}sum\'{e}.

\begin{theorem}
[Henkin-Novikov, 2011]\label{T/ HN}Soit $\left(  M,\sigma\right)  $ une
structure de conductivit\'{e} de genre $g$ avec $\sigma$ de classe $C^{3}$.
Alors $\sigma$ peut \^{e}tre reconstruite \`{a} partir de l'op\'{e}rateur de
Dirichlet-Neumann $N_{d}^{\sigma}$ en r\'{e}solvant $g$ \'{e}quations de type
Fredholm associ\'{e}es \`{a} $g$ donn\'{e}es g\'{e}n\'{e}riques de
$N_{d}^{\sigma}$ puis en r\'{e}solvant $g$ syst\`{e}mes explicites qui, dans
le cas o\`{u} $M$ est un domaine de $\left\{  z\in\mathbb{C}^{2};~P\left(
z\right)  =0\right\}  $, $P\in\mathbb{C}_{N}\left[  X\right]  $, sont des
syst\`{e}mes lin\'{e}aires de $N\left(  N-1\right)  $ \'{e}quations \`{a}
$N\left(  N-1\right)  $ inconnues.
\end{theorem}

\subsection{Conductivit\'{e} anisotrope\label{S/ RcoeffCst}}

Le th\'{e}or\`{e}me de Henkin-Santacesaria cit\'{e} dans la
section~\ref{S/ condaniso} permet pour une structure de conductivit\'{e}
$\left(  M,\sigma\right)  $ de construire dans $\mathbb{C}^{2}$ une courbe
complexe nodale qui est l'image par une immersion $\sigma$-holomorphe de
$\overline{M}$ dont la restriction \`{a} $bM$ est un plongement.

A l'aide du th\'{e}or\`{e}me~\ref{T/ plgt1} et du
corollaire~\ref{C/ GreenNodale3}, on peut reconstruire l'op\'{e}rateur de
Dirichlet-Neumann de la structure complexe $\left(  M,c_{\sigma}\right)  $
sous-jacente \`{a} $\left(  M,\sigma\right)  $. Ceci permet de produire une
courbe complexe $S$ de $\mathbb{CP}_{3}$ qui est isomorphe \`{a} $\left(
M,c_{\sigma}\right)  $ et ram\`{e}ne le probl\`{e}me initial \`{a} celui
o\`{u} le coefficient de conductivit\'{e} est constant, c'est-\`{a}-dire au
cas particulier \'{e}tudi\'{e} dans~\cite{HeG-MiV2007}. Si cette \'{e}tape est
franchie de fa\c{c}on constructive, le th\'{e}or\`{e}me de Henkin-Novikov
rappel\'{e} dans la section~\ref{S/ ReconsCondiso} permet de reconstruire la
conductivit\'{e} $\sigma$ elle-m\^{e}me. Le probl\`{e}me initial est alors
enti\`{e}rement r\'{e}solu s'il s'agit de produire une structure de
conductivit\'{e} abstraite dont le bord orient\'{e} et l'op\'{e}rateur de
Dirichlet-Neumann sont prescrits.\medskip

Dans cette section, on aborde la reconstruction effective de $S$, que sans
perte de g\'{e}n\'{e}ralit\'{e}, on suppose \^{e}tre un domaine relativement
compact d'une surface de Riemann ouverte $\widetilde{S}$ de $\mathbb{CP}_{3}$.
Pour un choix g\'{e}n\'{e}rique du quadruplet $\left(  u_{0},u_{1},u_{2}%
,u_{3}\right)  $ de fonctions utilis\'{e}es dans le
th\'{e}or\`{e}me~\ref{T/ plgt1}, on peut supposer que les projections $\pi
_{3}:\left(  w_{0}:w_{1}:w_{2}:w_{3}\right)  \mapsto\left(  w_{0}:w_{1}%
:w_{2}\right)  $ et $\pi_{2}:\left(  w_{0}:w_{1}:w_{2}:w_{3}\right)
\mapsto\left(  w_{0}:w_{1}:w_{3}\right)  $ immergent $\widetilde{S}$ dans
$\mathbb{CP}_{2}$ sur des courbes nodales $\widetilde{S}_{3}$ et
$\widetilde{S}_{2}$ telles que $\pi_{3}^{-1}\left(  \operatorname*{Sing}%
\widetilde{S}_{3}\right)  \cap\pi_{2}^{-1}\left(  \operatorname*{Sing}%
\widetilde{S}_{2}\right)  \cap\widetilde{S}=\varnothing$. D\`{e}s lors, pour
obtenir un atlas de $S$, il suffit d'en produire pour $\overline{S}%
\cap\operatorname*{Reg}\widetilde{S}_{2}$ et $\overline{S}\cap
\operatorname*{Reg}\widetilde{S}_{3}$, c'est-\`{a}-dire pour une surface de
Riemann nodale \`{a} bord $Q$ plong\'{e}e dans $\mathbb{CP}_{2}$ et qui est un
domaine relativement compact d'une surface de Riemann nodale ouverte
$\widetilde{Q}$ de $\mathbb{CP}_{2}$ et dont le bord orient\'{e} $\partial Q$
est connu. Notons que le genre des courbes nodales $S_{3}$ et $S_{2}$ est le
m\^{e}me que celui de $S$ et donc de $M$.

La m\'{e}thode que nous proposons dans la section~\ref{S/ algorithme}
d\'{e}coule de l'analyse des indicatrices de Cauchy-Fantapi\'{e} introduites
dans la section ci-dessous et de la caract\'{e}risation des sommes d'ondes de
chocs donn\'{e}e dans la section~\ref{S/ fabOM}. Cette m\'{e}thode diff\`{e}re
sensiblement de celle propos\'{e}e par Agaltsov-Henkin~\cite{AgA-HeG2015} pour
des cas particuliers.

\subsection{Indicatrices de Cauchy-Fantapi\'{e}}

Sans perte de g\'{e}n\'{e}ralit\'{e}, on suppose que $bQ\subset\left\{
w_{0}w_{1}w_{2}\neq0\right\}  $ ce qui permet de consid\'{e}rer%
\begin{equation}
\rho=\underset{w\in bQ}{\max}\left\vert \frac{w_{2}}{w_{1}}\right\vert
\label{F/ rhoZ}%
\end{equation}
Jusqu'\`{a} la fin de cette article, on fait en outre l'hypoth\`{e}se
g\'{e}n\'{e}rique et donc peu restrictive que%
\[
\left(  0:1:0\right)  \notin Q_{\infty}=Q\pitchfork\left\{  w_{0}=0\right\}
\subset\operatorname*{Reg}Q
\]
o\`{u} $\pitchfork$ d\'{e}signe une intersection transverse. L'hypoth\`{e}se
$Q_{\infty}\subset\operatorname*{Reg}Q$ n'est pas n\'{e}cessaire mais elle
simplifie les calculs. Toutefois, nous indiquons pour certaines formules une
version pour le cas o\`{u} $Q_{\infty}\cap\operatorname*{Sing}Q\neq
\varnothing$. Dans la situation r\'{e}guli\`{e}re, on peut prendre
$u_{0}=\frac{w_{0}}{w_{2}}$ comme coordonn\'{e}e pour $Q$ au voisinage des
points de $Q_{\infty} $ et il existe pour chaque $q\in Q_{\infty}$ une
fonction $g^{q}$ holomorphe au voisinage de $0$ dans $\mathbb{C}$ telle qu'au
voisinage de $q$ dans $\mathbb{CP}_{2}$, $Q$ co\"{\i}ncide avec $\left\{
\left(  u_{0}:u_{1}:1\right)  ;~u_{1}=g^{q}\left(  u_{0}\right)  \right\}  $.
On note alors $(\Sigma g_{\nu}^{q}u_{0}^{\nu})$ la s\'{e}rie de Talyor de
$g^{q}$ en $0$. Pour $q\in Q_{\infty}$, on a donc%
\[
q=\left(  0:g_{0}^{q}:1\right)  \overset{d\acute{e}f}{=}\left(  0:b_{q}%
:1\right)  .
\]
On pose aussi%
\[
R_{\infty}=\left\{  -1/b_{q};~q\in Q_{\infty}\right\}  ~~~\text{et}%
~~~E_{\infty}=\mathbb{C}\times R_{\infty}%
\]
On note $U$ l'ouvert form\'{e} par les points $z=\left(  x,y\right)  $ de
$\mathbb{C}^{2}$ tels que $bQ$ ne rencontre par $L_{z}=\left\{  w\in
\mathbb{CP}_{2};~xw_{0}+yw_{1}+w_{2}=0\right\}  $. L'hypoth\`{e}se $\left(
0:1:0\right)  \notin Q_{\infty}$ assure que $L_{z}\cap Q\subset\left\{
w_{0}\neq0\right\}  $ lorsque $y\neq0$. On d\'{e}signe par
$U_{\operatorname{reg}}$ l'ouvert de $\mathbb{C}^{2}$ form\'{e} par les points
$z$ de $U$ tels que pour tout $q\in Q\cap L_{z}$, $L_{z}$ coupe
transversalement $Q$ en $q$~; $U_{\operatorname{sing}}=U\backslash
U_{\operatorname{reg}}$ est un sous-ensemble analytique de $U$~; quand $E$ est
une partie de $U$, on pose $E_{\operatorname{reg}}=E\cap U_{\operatorname{reg}%
} $ et $E_{\operatorname{sing}}=E\cap U_{\operatorname{sing}}$. L'ensemble $Z$
d\'{e}finit ci-dessous joue un r\^{o}le essentiel~:
\begin{equation}
Z=\underset{\mathbb{C}\backslash\rho\overline{\mathbb{D}}}{\cup}m\left(
y\right)  \mathbb{D}\times\left\{  y\right\}  \label{F/ Z}%
\end{equation}
o\`{u} $\mathbb{D}=D\left(  0,1\right)  $ et
\[
m:\mathbb{C}\backslash\rho\overline{\mathbb{D}}\ni y\mapsto\underset{w\in
bQ}{\min}\left\vert \left(  y\frac{w_{1}}{w_{0}}+\frac{w_{2}}{w_{0}}\right)
\right\vert
\]
Notons que par construction, $Z\subset U$ et que $\underset{\left\vert
y\right\vert \rightarrow\infty}{\lim}m\left(  y\right)  =+\infty$ puisque
$bQ\subset\left\{  w_{1}\neq0\right\}  $. En g\'{e}n\'{e}ral $Q$ n'est pas
affine et donc $Q\cap L_{z}\neq\varnothing$ lorsque $z\in U$ mais le lemme
ci-dessous assure que le proc\'{e}d\'{e} de reconstruction amorc\'{e} avec la
proposition~\ref{P/ DH1997} aboutit bien \`{a} la connaissance compl\`{e}te de
$Q$.\label{a}

\begin{lemma}
\label{L/ Qexhaustion}Pour tout $w_{\ast}\in Q\cap\left\{  w_{0}\neq0\right\}
$ et tout $R\in\mathbb{R}_{+}^{\ast}$, il existe $z\in Z_{\operatorname{reg}%
}\cap\left(  \mathbb{C}\times\mathbb{C}\backslash R\overline{\mathbb{D}%
}\right)  $ tel que $w_{\ast}\in L_{z}$.
\end{lemma}

\begin{proof}
Soient $R\in\mathbb{R}_{+}^{\ast}$ et $w_{\ast}\in Q$ tel que $w_{\ast0}\neq
0$. Posons $\zeta_{\ast}=\left(  \frac{w_{\ast1}}{w_{\ast0}},\frac{w_{\ast2}%
}{w_{\ast0}}\right)  $. Les points $z=\left(  x,y\right)  $ de $\mathbb{C}%
^{2}$ tels que $w_{\ast}\in L_{z}$ forment la droite $L_{w_{\ast}}^{\ast}$
d'\'{e}quation $x+y\zeta_{\ast1}+\zeta_{\ast2}=0$. Si $L_{w_{\ast}}^{\ast
}\left(  R\right)  =L_{w_{\ast}}^{\ast}\cap\left(  \mathbb{C}\times
\mathbb{C}\backslash R\overline{\mathbb{D}}\right)  $ ne rencontre pas $U$,
c'est que pour tout $y\in\mathbb{C}\backslash R\overline{\mathbb{D}}$,
$\left(  -y\zeta_{\ast1}-\zeta_{\ast2},y\right)  \notin U$, ce qui implique
l'existence dans $bQ$ d'un \'{e}l\'{e}ment $w=\left(  1:\zeta_{1}:\zeta
_{2}\right)  $ tel que $w\in L_{z}$, c'est-\`{a}-dire tel que $\left(
-y\zeta_{\ast1}-\zeta_{\ast2}\right)  +y\zeta_{1}+\zeta_{2}=0$ soit encore
$y=-\frac{\zeta_{\ast2}-\zeta_{2}}{\zeta_{\ast1}-\zeta_{1}} $. Etant donn\'{e}
que $bQ$ est une courbe r\'{e}elle, $\mathbb{C}\backslash R\overline
{\mathbb{D}}$ ne peut \^{e}tre contenu dans l'image de $bQ$ par $\zeta
\mapsto-\frac{\zeta_{\ast2}-\zeta_{2}}{\zeta_{\ast1}-\zeta_{1}}$. Ainsi,
$L_{w_{\ast}}^{\ast}\left(  R\right)  \cap U$ est un ouvert non vide de
$L_{w_{\ast}}^{\ast}$. Notons que lorsque $z\in L_{w^{\ast}}^{\ast}\cap U$,
$Q\cap L_{z}$ est une partie finie non vide de $Q_{0}=Q\cap\left\{  w_{0}%
\neq0\right\}  $ car $L_{z}\cap bQ=\varnothing$, $w_{\ast}\in L_{z}$ et
$\left(  0:1:0\right)  \notin Q_{\infty}$.

Recouvrons $Q_{0}$ par une famille localement finie $\mathcal{B}$ de branches
de $Q$. Pour chaque branche $B\in\mathcal{B}$, on se donne une fonction
holomorphe $f$ d\'{e}finie sur un ouvert $V_{B}$ de $\mathbb{C}^{2}$ tel que
$V_{B}=\left\{  \left(  1:\zeta_{1}:\zeta_{2}\right)  ;~\zeta\in
V_{B}~\&~f_{B}\left(  \zeta\right)  =0\right\}  $. Notons $E\left(  R\right)
$ l'ensemble des $z$ de $L_{w_{\ast}}^{\ast}\left(  R\right)  $ tels que
$L_{z} $ et $Q$ soient tangentes en un point de $L_{z}\cap Q$. Un point $z$ de
$\mathbb{C}^{2}$ appartient \`{a} $E\left(  R\right)  $ lorsque $\left\vert
y\right\vert >R$ et qu'il existe $B\in\mathcal{B}$ et $\zeta\in V_{B}$
v\'{e}rifiant les conditions%
\begin{align*}
f_{B}\left(  \zeta\right)   &  =0,~~x+y\zeta_{\ast1}+\zeta_{\ast2}%
,~~x+y\zeta_{1}+\zeta_{2}=0,\\
\frac{\partial f_{B}}{\partial\zeta_{2}}\left(  \zeta\right)   &
\neq0,~y=\frac{\partial f_{B}/\partial\zeta_{1}}{\partial f_{B}/\partial
\zeta_{2}}\left(  \zeta\right)  ,~~x=-\frac{\partial f_{B}/\partial\zeta_{1}%
}{\partial f_{B}/\partial\zeta_{2}}\left(  \zeta\right)  \zeta_{\ast1}%
-\zeta_{\ast2}%
\end{align*}
Lorsque $\zeta\neq\zeta_{\ast}$, ceci impose $\zeta_{\ast1}\neq\zeta_{1}$ et
$-\frac{\partial f_{B}/\partial\zeta_{1}}{\partial f_{B}/\partial\zeta_{2}%
}\left(  \zeta\right)  =\frac{\zeta_{\ast2}-\zeta_{2}}{\zeta_{\ast1}-\zeta
_{1}}$. Les points $\zeta$ v\'{e}rifiant cette \'{e}quation forment donc un
sous-ensemble analytique $C_{B}$ de $B$. A ce titre, $C_{B}$ est soit discret,
soit \'{e}gal \`{a} $B$.

Supposons que $C_{B}=B$ pour un \'{e}l\'{e}ment $B$ de $\mathcal{B}$. Alors
$\partial f_{B}/\partial\zeta_{2}$ ne s'annule pas et on peut trouver
localement une fonction $\varphi_{B}$ telle que $f_{B}\left(  \zeta\right)
=0$ si et seulement si $\zeta_{2}=\varphi_{B}\left(  \zeta_{1}\right)  $. La
fonction $\varphi_{B}$ v\'{e}rifie alors $\varphi^{\prime}\left(  \zeta
_{1}\right)  +\frac{1}{\zeta_{\ast1}-\zeta_{1}}\varphi\left(  \zeta
_{1}\right)  =\frac{\zeta_{\ast2}}{\zeta_{\ast1}-\zeta_{1}}$, c'est-\`{a}-dire
$\left(  \frac{1}{\zeta_{1}-\zeta_{\ast1}}\varphi\left(  \zeta_{1}\right)
\right)  ^{\prime}=\left(  \frac{\zeta_{\ast2}}{\zeta_{1}-\zeta_{\ast1}%
}\right)  ^{\prime}$. D'o\`{u} $\varphi\left(  \zeta_{1}\right)  =\left(
\zeta_{1}-\zeta_{\ast1}\right)  c+\zeta_{\ast2}$ o\`{u} $c$\ est une
constante. Dans ce cas, $B$ est un ouvert de la droite d'\'{e}quation
$\zeta_{2}=\left(  \zeta_{1}-\zeta_{\ast1}\right)  c+\zeta_{\ast2}$ et il
suffit de prendre $y\neq c$ pour que $L_{z}$ ne coupe pas $B $
tangentiellement. Lorsque $C_{B}$ est une partie discr\`{e}te de $B$,
l'ensemble $E\left(  R,B\right)  $ des \'{e}l\'{e}ments $z$ de $L_{w_{\ast}%
}^{\ast}\left(  R\right)  $ tels que $L_{z}$ et $B$ soient tangentes en un
point de $L_{z}\cap B$ est contenu, du fait des relations ci-dessus, dans un
ensemble discret. Puisque $\mathcal{B}$ est localement finie, on en d\'{e}duit
de l'\'{e}tude de ces deux cas que $L_{w_{\ast}}^{\ast}\left(  R\right)  $
rencontre $Z_{\operatorname{reg}}\cap\left(  \mathbb{C}\times\mathbb{C}%
\backslash R\overline{\mathbb{D}}\right)  $.
\end{proof}

On utilise comme dans~\cite{HeG-MiV2007} les indicatrices de
Cauchy-Fantapi\'{e} de $Q$ qui sont les fonctions $G_{k}$, $k\in\mathbb{N}$,
d\'{e}finies sur $U$ par%
\begin{equation}
G_{k}\left(  z\right)  =\frac{1}{2\pi i}\int_{\partial Q}\Omega_{z}%
^{k},~\Omega_{z}^{k}=\left(  \frac{w_{1}}{w_{0}}\right)  ^{k}\frac
{1}{~x+y\frac{w_{1}}{w_{0}}+\frac{w_{2}}{w_{0}}}d\left(  x+y\frac{w_{1}}%
{w_{0}}+\frac{w_{2}}{w_{0}}\right)  \label{F/ Gk}%
\end{equation}

Le point d\'{e}part de la reconstruction est un r\'{e}sultat de
Dolbeault-Henkin essentiellement contenu dans \cite{DoP-HeG1997}; leur preuve
qui consiste \`{a} utiliser la formule de Stokes s'applique sans changement au
cas o\`{u} $Q_{\infty}$ contient des noeuds de $\overline{Q}$. Dans cet
\'{e}nonc\'{e} et ensuite, on utilise les notations suivantes lorsque
$h_{1},...,h_{p}$ sont des fonctions \`{a} valeurs dans $\mathbb{C}$ et
$k\in\mathbb{N}$,
\[
N_{h,k}=\sum\limits_{1\leqslant j\leqslant p}h_{j}^{k}\text{, }S_{h,k}%
=\sum\limits_{1\leqslant j_{1}<\cdots<j_{p}\leqslant k}h_{j_{1}}\cdots
h_{j_{k}.}%
\]
Les identit\'{e}s de Newton sont que pour tout $k\in\mathbb{N}^{\ast}$,
\begin{align}
N_{h,k}  &  =\left(  -1\right)  ^{k-1}kS_{h,k}+\sum\limits_{1\leqslant
j\leqslant k-1}\left(  -1\right)  ^{k-j-1}S_{h,j}N_{h,k-j}\label{F/ SN}\\
S_{h,k}  &  =\frac{\left(  -1\right)  ^{k-1}}{k}N_{h,k}+\frac{1}{k}%
\sum\limits_{1\leqslant j\leqslant k-1}\left(  -1\right)  ^{j-1}%
S_{h,j}N_{h,k-j} \label{F/ NS}%
\end{align}
On note $\mathbb{C}\left[  X,Y\right)  $ l'ensemble des \'{e}l\'{e}ments de
$\mathbb{C}\left(  X,Y\right)  $ qui sont des polyn\^{o}mes en $X$.
$\mathbb{C}_{k}\left[  X,Y\right)  $ d\'{e}signe l'anneau des polyn\^{o}mes en
$X$ de degr\'{e} au plus $k$ dont les coefficients sont des fractions
rationnelles en $Y $ de degr\'{e} compris entre $-1$ et $-k$. Une onde de choc
est par d\'{e}finition une fonction $h$ holomorphe sur un ouvert de
$\mathbb{C}^{2}$ telle que dans le syst\`{e}me $\left(  x,y\right)  $ des
coordonn\'{e}es standards%
\begin{equation}
\frac{\partial h}{\partial y}=h\frac{\partial h}{\partial x} \label{F/ OdC}%
\end{equation}

\begin{proposition}
[Dolbeault-Henkin, 1997]\label{P/ DH1997}Soient $z_{\ast}\in
U_{\operatorname{reg}}\backslash E_{\infty}$ et $p=\operatorname{Card}\left(
L_{z_{\ast}}\cap Q\right)  $. Si $U_{\ast}$ est un voisinage suffisamment
petit de $z_{\ast}$ dans $U_{\operatorname{reg}}$, il existe des ondes de choc
$h_{1},...,h_{p}$ sur $U_{\ast}$ dont les images sont deux \`{a} deux
distinctes telles que pour tout $z\in U_{\ast}$,
\[
L_{z}\cap Q=\left\{  \left(  1:h_{j}\left(  z\right)  :-x-yh_{j}\left(
z\right)  \right)  ;~1\leqslant j\leqslant p\right\}  .
\]
En outre, pour tout $k\in\mathbb{N}$, il existe $P_{k}\in\mathbb{C}_{k}\left[
X,Y\right)  $ tel que pour tout $z\in U_{\ast}$
\begin{equation}
G_{k}\left(  z\right)  =N_{h,k}\left(  z\right)  +P_{k}\left(  z\right)  .
\label{F/ G=N+P}%
\end{equation}
De plus, notant $\eta$ l'injection naturelle de $Q$ dans $\mathbb{CP}_{2}$,
$P_{k}=\sum\limits_{q\in Q_{\infty}}\operatorname*{Res}\left(  \eta^{\ast
}\Omega_{z}^{k},q\right)  $ et
\[
\frac{\partial P_{k}}{\partial Y}=\frac{k}{k+1}\frac{\partial P_{k+1}%
}{\partial X}.
\]

\end{proposition}

D'un point de vue pratique, la principale difficult\'{e} est de d\'{e}terminer
le nombre $p$ et les polyn\^{o}mes $P_{k}$. \cite{AgA-HeG2015} contient une
m\'{e}thode quand $p\in\left\{  1,2\right\}  $. Pour celle que propos\'{e}e
dans la section~\ref{S/ algorithme}, il faut commencer par apporter des
pr\'{e}cisions sur $p$ les polyn\^{o}mes $P_{k}$. Le lemme ci-dessous prouve
les polyn\^{o}mes $P_{k}$ ne d\'{e}pendent que des germes de $Q$ en les points
de $Q_{\infty}$.

\begin{lemma}
\label{L/ Pk}$P_{0}=-q_{\infty}$ o\`{u} $q_{\infty}=\operatorname{Card}%
Q_{\infty}$ et lorsque $k\in\mathbb{N}^{\ast}$, posant $P_{k}=\sum
\limits_{0\leqslant m\leqslant k}p_{k,m}X^{m}$,
\begin{equation}
p_{k,k}=\frac{1}{\left(  k-1\right)  !}p_{1,1}^{\left(  k-1\right)
}~~\&~~p_{k,m}=\frac{k}{m!\left(  k-m\right)  }p_{k-m,0}^{\left(  m\right)
},~~m\in\left\{  0,..,k-1\right\}  \label{F/ pkm}%
\end{equation}
En outre, si on pose%
\begin{equation}
B_{\infty}=\prod\limits_{q\in Q_{\infty}}\left(  1+Yb_{q}\right)
\label{F/ Binfini}%
\end{equation}
Alors%
\begin{align}
p_{1,1}  &  =\sum\limits_{q\in Q_{\infty}}\frac{b_{q}}{1+Yb_{q}}%
=\frac{B_{\infty}^{\prime}}{B_{\infty}}\label{F/ p11}\\
p_{1,0}  &  =-\sum\limits_{q\in Q_{\infty}}\frac{g_{1}^{q}}{1+Yb_{q}}%
=\frac{p_{1,0,1}}{B_{\infty}}\label{F/ p10}\\
p_{k,0}  &  =\sum\limits_{j=1}^{k}\sum\limits_{q\in Q_{\infty}}\frac
{p_{k,0,j}^{q}}{\left(  1+Yb_{q}\right)  ^{j}}=\sum\limits_{j=1}^{k}%
\frac{p_{k,0,j}}{B_{\infty}^{j}},k\in\mathbb{N} \label{F/ pk0}%
\end{align}
o\`{u} les $p_{k,0,j}^{q}$ sont des polyn\^{o}mes universels en les
coefficients du jet d'ordre $k-j+1$ de $Q$ en $q$ et $p_{k,0,j}=\sum
\limits_{q}p_{k,0,j}^{q}\underset{q^{\prime}\neq q}{\Pi}\left(
1+Yb_{q^{\prime}}\right)  ^{j}$. En particulier, $P_{k}$ ne d\'{e}pend pas de
$z_{\ast}$ et est enti\`{e}rement d\'{e}termin\'{e} par les $k\left(
q_{\infty}+1\right)  $ nombres $b_{q}$, $p_{k,0,j}^{q}$, $\left(  q,j\right)
\in Q_{\infty}\times\left\{  1,...,k\right\}  $
\end{lemma}

\noindent\textbf{Remarque. }Dans le cas o\`{u} $Q_{\infty}\cap
\operatorname*{Sing}Q\neq\varnothing$, la formule~(\ref{F/ Binfini}) devient
$B_{\infty}=\prod\limits_{q\in Q_{\infty}}\left(  1+Yb_{q}\right)
^{\nu\left(  q\right)  }$ o\`{u} $\nu\left(  q\right)  $ d\'{e}signe le nombre
de branches de $Q$ en $q$, (\ref{F/ p11}) reste inchang\'{e}e et dans
(\ref{F/ p10}), il faut remplacer $g_{1}^{q}$ par $\sum\limits_{B}g_{1}^{B,q}$
o\`{u} la sommation se fait sur un jeu complet de branches de $Q$ en $q$ et
$g_{1}^{B,q}=\left(  g^{B}\right)  ^{\prime}\left(  0\right)  $, $g^{B}$
d\'{e}signant la fonction holomorphe telle qu'au voisinage de $0$, une
\'{e}quation de la branche $B$ est $u_{1}=g_{1}^{B}\left(  u_{0}\right)  $.

\begin{proof}
Supposons que (\ref{F/ pkm}) soit v\'{e}rifi\'{e}e pour un entier $k$ non nul.
Alors%
\begin{align*}
P_{k+1}  &  =P_{k+1}\left(  0,Y\right)  +\frac{k+1}{k}\left(  \sum
\limits_{0\leqslant m\leqslant k-1}p_{k,m}^{\prime}\frac{X^{m+1}}{m+1}%
+p_{k,k}^{\prime}\frac{X^{k+1}}{k+1}\right) \\
&  =p_{k+1,0}+\sum\limits_{0\leqslant m\leqslant k-1}\frac{k+1}{\left(
m+1\right)  !\left(  k-m\right)  }p_{k-m,0}^{\left(  m+1\right)  }%
X^{m+1}+\frac{1}{k!}p_{1,1}^{\left(  k\right)  }X^{k+1}\\
&  =p_{k+1,0}+\sum\limits_{1\leqslant m\leqslant k}\frac{k+1}{\left(
m+1\right)  !\left(  k+1-m\right)  }p_{k+1-m,0}^{\left(  m\right)  }%
X^{m}+\frac{1}{k!}p_{1,1}^{\left(  k\right)  }%
\end{align*}
Ce qui prouve (\ref{F/ pkm}) par r\'{e}currence.

Soit maintenant $k\in\mathbb{N}$. Dans les coordonn\'{e}es affines $\left(
u_{0},u_{1}\right)  =\left(  \frac{w_{0}}{w_{2}},\frac{w_{1}}{w_{2}}\right)  $
de $\mathbb{CP}_{2}$, $\Omega_{z}^{k}$ s'\'{e}crit%
\[
\Omega_{z}^{k}=\left(  \frac{u_{1}}{u_{0}}\right)  ^{k}\frac{d\frac
{xu_{0}+yu_{1}+1}{u_{0}}}{\frac{xu_{0}+yu_{1}+1}{u_{0}}}=\left(  \frac{u_{1}%
}{u_{0}}\right)  ^{k}\left(  \frac{xdu_{0}+ydu_{1}}{xu_{0}+yu_{1}+1}%
-\frac{du_{0}}{u_{0}}\right)  .
\]
On fixe un point $q$ de $Q_{\infty}$ et pour simplifier les notations, on
\'{e}crit $g$ au lieu de $g^{q}$ (et donc $g_{\nu}$ pour $g_{\nu}^{q}$) et $u$
pour $u_{0}$ pour quelque temps. Au voisinage de $q$ dans $Q$, la forme
$\eta^{\ast}\Omega_{z}^{k}$ s'\'{e}crit dans la coordonn\'{e}e $u$ sous la
forme%
\[
\eta^{\ast}\Omega_{z}^{k}=\left(  \frac{\left(  x+yg^{\prime}\right)  g^{k}%
}{u^{k}\left(  1+xu+yg\right)  }-\frac{g^{k}}{u^{k+1}}\right)  du.
\]
Notant par $\left\langle f,u^{\nu}\right\rangle $ le coefficient de $u^{\nu}$
dans la s\'{e}rie de Taylor en $0$ d'une fonction $f$ holomorphe au voisinage
de $0$, on obtient
\[
P_{k}^{q}\left(  z\right)  \overset{d\acute{e}f}{=}\operatorname*{Res}\left(
\eta^{\ast}\Omega_{z}^{k},q\right)  =\operatorname*{Res}\left(  \frac{\left(
x+yg^{\prime}\right)  g^{k}}{\left(  1+xu+yg\right)  u^{k}},0\right)
-\left\langle g^{k},u^{k}\right\rangle .
\]
En particulier $P_{0}^{q}\left(  z\right)  =-1$ et donc $P_{0}%
=-\operatorname{Card}Q_{\infty}$. Supposons d\'{e}sormais $k\geqslant1$.
Notons alors que
\[
\frac{yg^{\prime}g^{k}}{1+xu+yg}=\left(  1-\frac{1+xu}{1+xu+yg}\right)
g^{\prime}g^{k-1}%
\]
et que si $g^{\prime}g^{k-1}=\sum\limits_{n\in\mathbb{N}}\alpha_{n}u^{n}$,
$\frac{1}{k}\left(  g^{k}-g_{0}^{k}\right)  =\sum\limits_{n\in\mathbb{N}%
^{\ast}}\frac{\alpha_{n-1}}{n}u^{n}$, ce qui livre
\[
\operatorname*{Res}\left(  \frac{g^{k}}{u^{k+1}}du,0\right)  =k\frac
{\alpha_{k-1}}{k}=\operatorname*{Res}\left(  \frac{g^{\prime}g^{k-1}}{u^{k}%
}du,0\right)
\]
Par cons\'{e}quent,%
\begin{align*}
P_{k}^{q}\left(  z\right)   &  =\operatorname*{Res}\left(  \frac{xg^{k}%
}{\left(  1+xu+yg\right)  u^{k}},0\right)  -\operatorname*{Res}\left(
\frac{1+xu}{\left(  1+xu+yg\right)  u^{k}}g^{\prime}g^{k-1},0\right) \\
&  =\operatorname*{Res}\left(  \frac{x\left(  g-ug^{\prime}\right)
-g^{\prime}}{\left(  1+xu+yg\right)  u^{k}}g^{k-1},0\right)  .
\end{align*}
Puisque $g-g_{0}=O\left(  u\right)  $, il vient par ailleurs pour $u$ assez
petit%
\[
\frac{1}{1+xu+yg}=\frac{\left(  1+yg_{0}\right)  ^{-1}}{1+\frac{xu+y\left(
g-g_{0}\right)  }{1+yg_{0}}}=\sum\limits_{n\in\mathbb{N}^{\ast}}\frac{\left(
-1\right)  ^{n-1}}{\left(  1+yg_{0}\right)  ^{n}}\left[  xu+y\left(
g-g_{0}\right)  \right]  ^{n-1}.
\]
Or pour tout $n\in\mathbb{N}^{\ast}$%
\begin{align*}
&  \left[  x\left(  g-ug^{\prime}\right)  -g^{\prime}\right]  g^{k-1}\left[
xu+y\left(  g-g_{0}\right)  \right]  ^{n-1}\\
&  =\sum\limits_{m=0}^{n-1}C_{n-1}^{m}g^{k-1}\left(
\begin{array}
[c]{l}%
x^{m+1}y^{n-1-m}\left(  g-ug^{\prime}\right)  \left(  g-g_{0}\right)
^{n-1-m}u^{m}\\
-g^{\prime}x^{m}y^{n-1-m}\left(  g-g_{0}\right)  ^{n-1-m}u^{m}%
\end{array}
\right) \\
&  =%
\begin{array}
[c]{l}%
\sum\limits_{m=1}^{n}C_{n-1}^{m-1}x^{m}y^{n-m}g^{k-1}\left(  g-ug^{\prime
}\right)  \left(  g-g_{0}\right)  ^{n-m}u^{m-1}\\
-\sum\limits_{m=0}^{n-1}C_{n-1}^{m}x^{m}y^{n-1-m}g^{\prime}g^{k-1}\left(
g-g_{0}\right)  ^{n-1-m}u^{m}%
\end{array}
\\
&  =%
\begin{array}
[c]{l}%
-y^{n-1}g^{\prime}g^{k-1}\left(  g-g_{0}\right)  ^{n-1}+x^{n}g^{k-1}\left(
g-ug^{\prime}\right)  u^{n-1}\\
+\sum\limits_{m=1}^{n-1}x^{m}y^{n-1-m}\left(
\begin{array}
[c]{l}%
yC_{n-1}^{m-1}g^{k-1}\left(  g-ug^{\prime}\right)  \left(  g-g_{0}\right) \\
-C_{n-1}^{m}g^{\prime}g^{k-1}u
\end{array}
\right)  \left(  g-g_{0}\right)  ^{n-1-m}u^{m-1}%
\end{array}
\end{align*}
Donc
\begin{align*}
&  P_{k}^{q}\left(  z\right) \\
&  =-\sum\limits_{n=1}^{k}\frac{\left(  -1\right)  ^{n}}{\left(
1+yg_{0}\right)  ^{n}}\left(
\begin{array}
[c]{l}%
-y^{n-1}\left\langle g^{\prime}g^{k-1}\left(  g-g_{0}\right)  ^{n-1}%
,u^{k-1}\right\rangle +x^{n}\left\langle g^{k-1}\left(  g-ug^{\prime}\right)
,u^{k-n}\right\rangle \\
+\sum\limits_{m=1}^{n-1}x^{m}y^{n-1-m}\left\langle \left(
\begin{array}
[c]{l}%
yC_{n-1}^{m-1}g^{k-1}\left(  g-ug^{\prime}\right)  \left(  g-g_{0}\right) \\
-C_{n-1}^{m}g^{\prime}g^{k-1}u
\end{array}
\right)  \left(  g-g_{0}\right)  ^{n-1-m},u^{k-m}\right\rangle
\end{array}
\right)
\end{align*}
D'o\`{u} $P_{k}^{q}\left(  z\right)  =\sum\limits_{m=0}^{k}p_{k,m}^{q}\left(
y\right)  x^{m}$ avec
\[
p_{k,0}^{q}=\sum\limits_{n=1}^{k}\frac{\left(  -1\right)  ^{n}Y^{n-1}}{\left(
1+Yg_{0}\right)  ^{n}}\left\langle g^{\prime}g^{k-1}\left(  g-g_{0}\right)
^{n-1},u^{k-1}\right\rangle ,~~~p_{k,k}^{q}=\frac{\left(  -1\right)
^{k+1}g_{0}^{k}}{\left(  1+Yg_{0}\right)  ^{k}}%
\]
et pour $1\leqslant m\leqslant k-1$,%
\begin{align*}
&  p_{k,m}^{q}\\
&  =-\sum\limits_{n=m+1}^{k}\frac{\left(  -1\right)  ^{n}Y^{n-1-m}}{\left(
1+Yg_{0}\right)  ^{n}}\left\langle \left(  YC_{n-1}^{m-1}g^{k-1}\left(
g-ug^{\prime}\right)  \left(  g-g_{0}\right)  +C_{n-1}^{m}g^{\prime}%
g^{k-1}u\right)  \left(  g-g_{0}\right)  ^{n-1-m},u^{k-m}\right\rangle
\end{align*}
En particulier,
\[
p_{1,0}^{q}=\frac{-g_{1}}{1+Yg_{0}}~~\&~~p_{1,1}^{q}=\frac{g_{0}}{1+Yg_{0}}.
\]
Par ailleurs, pour tout $n\in\mathbb{N}$,
\[
\frac{\left(  -1\right)  ^{n}Y^{n-1}}{\left(  1+Yg_{0}\right)  ^{n}}%
=\frac{\left(  -1\right)  ^{n}}{g_{0}^{n-1}\left(  1+Yg_{0}\right)  }\left(
1-\frac{1}{1+Yg_{0}}\right)  ^{n-1}=\left(  -1\right)  ^{n}\sum\limits_{j=1}%
^{n}\frac{\left(  -1\right)  ^{j-1}C_{n-1}^{j-1}g_{0}^{-\left(  n-1\right)  }%
}{\left(  1+Yg_{0}\right)  ^{j}},
\]
D'o\`{u}
\begin{align*}
p_{k,0}^{q}  &  =\sum\limits_{j=1}^{k}\frac{\left(  -1\right)  ^{j-1}}{\left(
1+Yg_{0}\right)  ^{j}}\sum\limits_{n=j}^{k}\frac{\left(  -1\right)
^{n}C_{n-1}^{j-1}}{g_{0}^{n-1}}\left\langle g^{\prime}g^{k-1}\left(
g-g_{0}\right)  ^{n-1},u^{k-1}\right\rangle \\
&  =\frac{-g_{1}^{k}}{\left(  1+Yg_{0}\right)  ^{k}}+\sum\limits_{j=1}%
^{k-1}\frac{\left(  -1\right)  ^{j-1}}{\left(  1+Yg_{0}\right)  ^{j}}%
\sum\limits_{n=j}^{k}\frac{\left(  -1\right)  ^{n}C_{n-1}^{j-1}}{g_{0}^{n-1}%
}\left\langle g^{\prime}g^{k-1}\left(  g-g_{0}\right)  ^{n-1},u^{k-1}%
\right\rangle
\end{align*}
On remarque que $\left\langle g^{\prime}g^{k-1}\left(  g-g_{0}\right)
^{k-1},u^{k-1}\right\rangle =g_{1}g_{0}^{k-1}g_{1}^{k-1}=g_{1}^{k}g_{0}^{k-1}
$ et que%
\begin{align*}
&  \left\langle g^{\prime}g^{k-1}\left(  g-g_{0}\right)  ^{k-2},u^{k-1}%
\right\rangle \\
&  =\left(  g_{1}+2g_{2}u+O\left(  u^{2}\right)  \right)  \left(  g_{0}%
+g_{1}u+O\left(  u^{2}\right)  \right)  ^{k-1}\left(  g_{1}u+g_{2}%
u^{2}+O\left(  u^{3}\right)  \right)  ^{k-2}\\
&  =\left(  g_{1}+2g_{2}u\right)  \left(  g_{0}^{k-1}+\left(  k-1\right)
g_{0}^{k-2}g_{1}u\right)  \left(  g_{1}^{k-2}u^{k-2}+\left(  k-2\right)
g_{1}^{k-3}g_{2}u^{k-1}\right)  +O\left(  u^{k}\right) \\
&  =\left(  g_{1}g_{0}^{k-1}+\left(  2g_{2}g_{0}^{k-1}+\left(  k-1\right)
g_{0}^{k-2}g_{1}^{2}\right)  u\right)  \left(  g_{1}^{k-2}u^{k-2}+\left(
k-2\right)  g_{1}^{k-3}g_{2}u^{k-1}\right)  +O\left(  u^{k}\right) \\
&  =g_{1}^{k-1}g_{0}^{k-1}u^{k-2}+\left[  g_{1}g_{0}^{k-1}\left(  k-2\right)
g_{1}^{k-3}g_{2}+\left[  2g_{2}g_{0}^{k-1}+\left(  k-1\right)  g_{0}%
^{k-2}g_{1}^{2}\right]  g_{1}^{k-2}\right]  u^{k-1}+O\left(  u^{k}\right) \\
&  =g_{1}^{k-1}g_{0}^{k-1}u^{k-2}+\left[  kg_{0}^{k-1}g_{1}^{k-2}g_{2}+\left(
k-1\right)  g_{0}^{k-2}g_{1}^{k}\right]  u^{k-1}+O\left(  u^{k}\right)
\end{align*}
ce qui donne%
\[
\frac{\left(  -1\right)  ^{k}C_{k-1}^{k-1}}{g_{0}^{k-1}}\left\langle
g^{\prime}g^{k}\left(  g-g_{0}\right)  ^{n-1},u^{k-1}\right\rangle =\left(
-1\right)  ^{k}g_{1}^{k}%
\]
et%
\begin{align*}
&  \sum\limits_{n=k-1}^{k}\frac{\left(  -1\right)  ^{n}C_{n-1}^{k-2}}%
{g_{0}^{n-1}}\left\langle g^{\prime}g^{k}\left(  g-g_{0}\right)
^{n-1},u^{k-1}\right\rangle \\
&  =\frac{\left(  -1\right)  ^{k}\left(  k-1\right)  }{g_{0}^{k-1}%
}\left\langle g^{\prime}g^{k}\left(  g-g_{0}\right)  ^{k-1},u^{k-1}%
\right\rangle +\frac{\left(  -1\right)  ^{k-1}}{g_{0}^{k-2}}\left\langle
g^{\prime}g^{k}\left(  g-g_{0}\right)  ^{k-2},u^{k-1}\right\rangle \\
&  =\left(  -1\right)  ^{k}\left(  k-1\right)  g_{1}^{k}+\frac{\left(
-1\right)  ^{k-1}}{g_{0}^{k-2}}\left[  kg_{0}^{k-1}g_{1}^{k-2}g_{2}+\left(
k-1\right)  g_{0}^{k-2}g_{1}^{k}\right] \\
&  =\left(  -1\right)  ^{k}\left(  \left(  k-1\right)  g_{1}^{k}-\left[
kg_{0}g_{1}^{k-2}g_{2}+\left(  k-1\right)  g_{1}^{k}\right]  \right)
=-k\left(  -1\right)  ^{k}g_{0}g_{1}^{k-2}g_{2}%
\end{align*}
D'o\`{u}
\[
p_{k,0}^{q}=\frac{-g_{1}^{k}}{\left(  1+Yg_{0}\right)  ^{k}}+\frac
{-kg_{0}g_{1}^{k-2}g_{2}}{\left(  1+Yg_{0}\right)  ^{k-1}}+\sum\limits_{j=1}%
^{k-2}\frac{p_{k,0,j}^{q}}{\left(  1+Yg_{0}\right)  ^{j}}%
\]
avec%
\[
p_{k,0,j}^{q}=\left(  -1\right)  ^{j-1}\sum\limits_{n=j}^{k}\frac{\left(
-1\right)  ^{n}C_{n-1}^{j-1}}{g_{0}^{n-1}}\left\langle g^{\prime}%
g^{k-1}\left(  g-g_{0}\right)  ^{n-1},u^{k-1}\right\rangle
\]
En sommant sur $q$ d\'{e}crivant $Q_{\infty}$ les relations obtenues, on
produit celles annonc\'{e}es dans le lemme.
\end{proof}

Une cons\'{e}quence non vitale pour notre but mais somme toute remarquable,
est le r\'{e}sultat de confinement suivant.

\begin{corollary}
\label{C/ confinementZeroBinfini}L'ensemble $R_{\infty}$ des racines du
polyn\^{o}me $B_{\infty}$ d\'{e}fini par la formule (\ref{F/ Binfini}) est
contenu dans $\rho\overline{\mathbb{D}}$.
\end{corollary}

\begin{proof}
Pour \'{e}tablir le confinement annonc\'{e}, il faut rentrer un peu dans le
d\'{e}tail de la preuve de la proposition~\ref{P/ DH1997}. Notons $v=\left(
\frac{w_{1}}{w_{0}},\frac{w_{2}}{w_{0}}\right)  $ le couple des
coordonn\'{e}es naturelles de $\left\{  w_{0}\neq0\right\}  $. Soit $z_{\ast
}=\left(  x_{\ast},y_{\ast}\right)  \in U$. La droite $L_{z_{\ast}}$ rencontre
donc $Q$ en au moins un point et puisque $L_{z_{\ast}}\cap bQ=\varnothing$,
$L_{z_{\ast}}\cap Q$ ne poss\`{e}de pas de point d'accumulation dans $bQ$ ni
dans $Q$ car sinon $Q$ contiendrait un ouvert non vide de la droite
$L_{z_{\ast}}$ et celle-ci aurait \`{a} rencontrer $bQ$. $L_{z_{\ast}}\cap Q $
est donc un ensemble fini, contenu dans $\left\{  w_{0}\neq0\right\}  $ car
$\left(  0:1:0\right)  \notin Q$. Si $w_{\ast}=\left(  1:v_{\ast1}:v_{\ast
2}\right)  \in Q$, on pose $D\left(  w_{\ast},\varepsilon\right)  =\left\{
\left(  1:v_{1}:v_{2}\right)  ;~\left\vert v-v_{\ast}\right\vert
<\varepsilon\right\}  $ et $D_{Q}\left(  w_{\ast},\varepsilon\right)  =Q\cap
D\left(  w_{\ast},\varepsilon\right)  $. S\'{e}lectionnons $\varepsilon
\in\mathbb{R}_{+}^{\ast}$ tel que les disques $\overline{D\left(  q_{\ast
},\varepsilon\right)  }$, $q_{\ast}\in L_{z_{\ast}}\cap Q$, soient deux \`{a}
disjoints. Il existe alors un voisinage $U_{\ast}$ de $z_{\ast}$ tel que
$L_{z}\cap Q\subset Q_{\varepsilon}\underset{q_{\ast}\in L_{z_{\ast}}\cap
Q}{\cup}D\left(  q_{\ast},\varepsilon\right)  $ lorsque $z\in U_{\ast}$. Par
cons\'{e}quent, lorsque $z\in U_{\ast}$, on peut appliquer$^{(}$\footnote{Dans
le cas o\`{u} $Q$ a des singularit\'{e}s quelconques, celle-ci reste valable
mais ce n'est pas \'{e}l\'{e}mentaire comme dans le cas nodal.}$^{)}$ la
formule de Stokes \`{a} la forme $\Omega_{z}^{1}$ et $Q\backslash
Q_{\varepsilon}$. On obtient avec les notations de la
proposition~\ref{P/ DH1997}
\[
G_{1}\left(  z\right)  =\sum\limits_{q_{\ast}\in L_{z_{\ast}}\cap Q}\frac
{1}{2\pi i}\int_{\partial D_{Q}\left(  q_{\ast},\varepsilon\right)  }%
\Omega_{z}^{1}+\sum\limits_{q\in Q_{\infty}}\operatorname*{Res}\left(
\eta^{\ast}\Omega_{z}^{k},q\right)
\]
En prenant $\varepsilon$ suffisamment petit, chaque int\'{e}grale ci-dessus se
d\'{e}compose en une somme d'int\'{e}grales de la forme $\int_{B\cap\partial
D\left(  q_{\ast},\varepsilon\right)  }\Omega_{z}^{1}$ o\`{u} $B$ est une
branche de $Q$ en $q_{\ast}$. Les th\'{e}or\`{e}mes classiques de
r\'{e}gularit\'{e} s'appliquant \`{a} ces derni\`{e}res, on en d\'{e}duit que
$z\mapsto\sum\limits_{q\in Q_{\infty}}\operatorname*{Res}\left(  \eta^{\ast
}\Omega_{z}^{k},q\right)  $ est holomorphe au voisinage de $z_{\ast}$. Par
ailleurs, les fractions rationnelles $P_{k}$ ne d\'{e}pendant pas du point de
$U_{\operatorname{reg}}$ auquel on applique la proposition~\ref{P/ DH1997}, on
sait aussi que si $z=\left(  x,y\right)  \in U_{\operatorname{reg}}$,
$\sum\limits_{q\in Q_{\infty}}\operatorname*{Res}\left(  \eta^{\ast}\Omega
_{z}^{k},q\right)  =P_{1}\left(  z\right)  =\frac{A\left(  y\right)
}{B_{\infty}\left(  y\right)  }+x\frac{B_{\infty}^{\prime}\left(  y\right)
}{B_{\infty}\left(  y\right)  }$ o\`{u} $A\in\mathbb{C}\left[  Y\right]  $.
Ceci force $B\left(  y_{\ast}\right)  \neq0$. Etant donn\'{e} que $Z\subset
U$, on en d\'{e}duit que $B_{\infty}\subset\rho\overline{\mathbb{D}}$.
\end{proof}

\noindent\textbf{Remarque. }En fait, $B_{\infty}$ est contenu dans l'image de
$\mathbb{C}\backslash\overline{U}$ par la projection $\mathbb{C}^{2}\ni\left(
x,y\right)  \mapsto y$.

\subsection{D\'{e}veloppement des indicatrices}

La forme des fractions $P_{k}$ d\'{e}gag\'{e}e par le lemme~\ref{L/ Pk}
sugg\`{e}re d'\'{e}tudier les fonctions $G_{k}$ sur le domaine $Z$ d\'{e}finit
par (\ref{F/ Z}).

\begin{lemma}
\label{L/ serieLdeG} On note $\delta$ l'entier relatif $\frac{1}{2\pi i}%
\int_{\partial Q}\frac{d\left(  w_{1}/w_{0}\right)  }{w_{1}/w_{0}}$. Pour tout
$k\in\mathbb{N}$, $G_{k}$ admet sur $Z$ un d\'{e}veloppement en s\'{e}rie de
Laurent de la forme
\begin{equation}
G_{k}\left(  x,y\right)  =\sum\limits_{m\in\mathbb{N}}\frac{G_{k,m}\left(
x\right)  }{y^{m}}=\left(  -1\right)  ^{k}\delta\frac{x^{k}}{y^{k}}%
+\sum\limits_{m\in\mathbb{N}^{\ast}}\frac{\widetilde{G}_{k,m}\left(  x\right)
}{y^{m}} \label{F/ devGk}%
\end{equation}
avec convergence normale sur tout compact de $Z$ et o\`{u} pour tout
$m\in\mathbb{N}^{\ast}$, $\widetilde{G}_{k,m}=\sum\limits_{0\leqslant
n<m}G_{k,m}^{n}X^{n}$ est un polyn\^{o}me degr\'{e} au plus $m-1$,
$G_{k,0}=\delta_{k,0}\delta$, $G_{k,m}=\delta_{k,m}\left(  -1\right)
^{m}\delta^{m}X^{m}+\widetilde{G}_{k,m}\overset{d\acute{e}f}{=}\sum
\limits_{0\leqslant n\leqslant m}G_{k,m}^{n}X^{n}$. En particulier
\begin{equation}
G_{1}\left(  x,y\right)  =\frac{G_{1,1}^{0}-\delta x}{y}+\sum
\limits_{m\geqslant2}\frac{G_{k,m}\left(  x\right)  }{y^{m}} \label{F/ devG1}%
\end{equation}
avec $G_{1,1}^{0}=\frac{-1}{2\pi i}\int_{\partial Q}\frac{w_{2}}{w_{1}}%
d\frac{w_{1}}{w_{0}}+\frac{1}{2\pi i}\int_{\partial Q}\left(  \frac{w_{2}%
}{w_{0}}\right)  {}^{2}d\frac{w_{2}}{w_{0}}$.
\end{lemma}

\begin{proof}
Fixons $k$ dans $\mathbb{N}$. Soit $y\in\mathbb{C}\backslash\rho
\overline{\mathbb{D}}$. La fonction $G_{k}\left(  .,y\right)  $ est holomorphe
sur $D_{y}=D\left(  0,m\left(  y\right)  \right)  $ et se d\'{e}veloppe en
s\'{e}rie enti\`{e}re sous la forme%
\[
\forall x\in D_{y},~G_{k}\left(  x,y\right)  =\sum\limits_{n\in\mathbb{N}%
}G_{k}^{n}\left(  y\right)  x^{n}%
\]
avec pour tout $n\in\mathbb{N}$%
\[
G_{k}^{n}\left(  y\right)  =\frac{1}{n!}\frac{1}{2\pi i}\left(  \frac
{\partial^{n}}{\partial x^{n}}\int_{\left(  \partial Q\right)  _{0}}z_{1}%
^{k}\frac{ydz_{1}+dz_{2}}{x+yz_{1}+z_{2}}\right)  \left\vert _{x=0}\right.
=\frac{\left(  -1\right)  ^{n}}{2\pi i}\int_{\left(  \partial Q\right)  _{0}%
}z_{1}^{k}\frac{ydz_{1}+dz_{2}}{\left(  yz_{1}+z_{2}\right)  ^{n+1}}.
\]
o\`{u} $\left(  \partial Q\right)  _{0}$ est l'image de $\partial Q$ dans les
coordonn\'{e}es naturelles $\left(  z_{1},z_{2}\right)  =\left(  \frac{w_{1}%
}{w_{0}},\frac{w_{2}}{w_{0}}\right)  $ de $\left\{  w_{0}\neq0\right\}  $.
Lorsque $w\in bQ$, $\left\vert \frac{w_{2}}{w_{1}}\right\vert \leqslant
\rho<\left\vert y\right\vert $ et donc $\left\vert \frac{z_{2}}{yz_{1}%
}\right\vert \leqslant\frac{\rho}{\left\vert y\right\vert }<1$ ce qui permet
d'\'{e}crire%
\[
\frac{1}{\left(  yz_{1}+z_{2}\right)  ^{n+1}}=\sum\limits_{m\in\mathbb{N}%
}C_{-n-1}^{m}\frac{z_{2}^{m}}{\left(  yz_{1}\right)  ^{n+1+m}}%
\]
o\`{u} $C_{-n-1}^{m}=\frac{\left(  -1\right)  ^{m}A_{n+m}^{m}}{m!}$ et avec
convergence normale sur tout compact de $\mathbb{C}\backslash\rho
\overline{\mathbb{D}}\times bQ$. D'o\`{u}%
\begin{align*}
G_{k}^{n}\left(  y\right)   &  =\left(  -1\right)  ^{n}\sum\limits_{m\in
\mathbb{N}}\frac{\left(  -1\right)  ^{m}A_{n+m}^{m}}{m!}\frac{1}{2\pi i}%
\int_{\left(  \partial Q\right)  _{0}}\frac{z_{1}^{k}z_{2}^{m}\left(
ydz_{1}+dz_{2}\right)  }{\left(  yz_{1}\right)  ^{n+1+m}}\\
&  =\sum\limits_{m\in\mathbb{N}}\frac{\left(  -1\right)  ^{n+m}A_{n+m}^{m}%
}{m!2\pi i}\int_{\left(  \partial Q\right)  _{0}}\left(  \frac{1}{y^{n+m}%
}z_{1}^{k-n-1-m}z_{2}^{m}dz_{1}+\frac{1}{y^{n+1+m}}z_{1}^{k-n-1-m}z_{2}%
^{m}dz_{2}\right)
\end{align*}
avec convergence normale sur tout compact de $\mathbb{C}\backslash
\rho\overline{\mathbb{D}}$. Etant donn\'{e} que $z_{1}^{k-n-1}dz_{1}=\frac
{1}{k-n}dz_{1}^{k-n}$ lorsque $k\neq n$, il vient $G_{k}^{n}\left(  y\right)
=\sum\limits_{m\geqslant n}\frac{G_{k,m}^{n}}{y^{m}}$ avec
\[
G_{k,n}^{n}=\frac{\left(  -1\right)  ^{n}\delta_{k,n}}{2\pi i}\int_{\left(
\partial Q\right)  _{0}}\frac{dz_{1}}{z_{1}}=\left(  -1\right)  ^{n}%
\delta_{k,n}g
\]
et
\[
G_{k,m}^{n}=\frac{\left(  -1\right)  ^{m}A_{m}^{m-n}}{m!2\pi i}\int_{\left(
\partial Q\right)  _{0}}z_{1}^{k-m-1}z_{2}^{m-n}dz_{1}-\frac{\left(
-1\right)  ^{m}A_{m-1}^{m-n-1}}{\left(  m-n-1\right)  !2\pi i}\int_{\left(
\partial Q\right)  _{0}}z_{1}^{k-m}z_{2}^{m-n+1}dz_{2}%
\]
lorsque $m>n$. La s\'{e}rie double de terme g\'{e}n\'{e}ral $G_{k,m}^{n}%
x^{n}y^{-m}$ \'{e}tant normalement convergente sur les compacts de $Z$, on
obtient%
\begin{align*}
G_{k}\left(  x,y\right)   &  =\sum\limits_{n\in\mathbb{N}}G_{k}^{n}\left(
y\right)  x^{n}=\sum\limits_{n\in\mathbb{N}}\left(  \frac{\left(  -1\right)
^{n}\delta_{k,n}g}{y^{n}}+\sum\limits_{m>n}\frac{G_{k,m}^{n}}{y^{m}}\right)
x^{n}\\
&  =\left(  -1\right)  ^{k}g\frac{x^{k}}{y^{k}}+\sum\limits_{m\in
\mathbb{N}^{\ast}}\frac{1}{y^{m}}\sum\limits_{0\leqslant n<m}G_{k,m}^{n}x^{n}%
\end{align*}
avec $G_{k,m}=\sum\limits_{0\leqslant n<m}G_{k,m}^{n}X^{n}$. D'o\`{u} les
relations annonc\'{e}es dans le lemme et%
\[
G_{1}\left(  x,y\right)  =-\delta\frac{x}{y}+\sum\limits_{m\in\mathbb{N}%
^{\ast}}\frac{\widetilde{G}_{1,m}\left(  x\right)  }{y^{m}}=\frac{G_{1,1}%
^{0}-\delta x}{y}+\sum\limits_{m\geqslant2}\frac{G_{1,m}\left(  x\right)
}{y^{m}}%
\]
avec $\widetilde{G}_{1,m}=\sum\limits_{0\leqslant n<m}G_{1,m}^{n}X^{n}$
\end{proof}

\begin{corollary}
\label{C/ pCst}Le nombre $p$ de fonctions $h_{1},...,h_{p}$ qui interviennent
dans la proposition~\ref{P/ DH1997} est le m\^{e}me pour tous les points de
$Z_{\operatorname{reg}}\backslash E_{\infty}$~: $p=\delta+q_{\infty} $ o\`{u}
$q_{\infty}=\operatorname{Card}Q_{\infty}$.
\end{corollary}

\noindent\textbf{Remarques. 1. }Dans le cas o\`{u} $Q_{\infty}\cap
\operatorname*{Sing}Q\neq\varnothing$, $q_{\infty}=\sum\limits_{q\in
Q_{\infty}}\nu\left(  q\right)  $.

\noindent\textbf{2. }Le corollaire~\ref{C/ genre via dp(u)} donne que pour une
connexion de Chern $D$ appropri\'{e}e calculable \`{a} partir de
l'op\'{e}rateur de Dirichlet-Neumann,
\begin{equation}
q_{\infty}=\frac{1}{2\pi i}\int_{\partial Q}\frac{D\partial u_{0}}{\partial
u_{0}}+2g\left(  M\right)  -2+c. \label{F/ qinfini-g}%
\end{equation}
o\`{u} $c$ est le nombre de composantes connexes de $bM$. Si le genre de $M$
est connu, cette formule livre la valeur de $q_{\infty}$ puis celle de $p$.

\begin{proof}
Notons provisoirement $p\left(  z\right)  $ le nombre de fonctions
$h_{1},...,h_{p\left(  z\right)  }$ qui interviennent dans la
proposition~\ref{P/ DH1997} quand $z\in U_{\operatorname{reg}}$. Puisque
$P_{0}=-q_{\infty}$, on sait que $G_{0}\left(  z\right)  =p\left(  z\right)
-q_{\infty}$ et donc que $p $ est une fonction \`{a} valeurs dans $\mathbb{Z}$
continue sur l'ensemble connexe $Z_{\operatorname{reg}}\backslash Z_{\infty}$.
Elle est donc constante et puisque $G_{0}\left(  z\right)  =\delta
+\sum\limits_{m\in\mathbb{N}^{\ast}}\frac{G_{0,m}\left(  x\right)  }{y^{m}}$
lorsque $z=\left(  x,y\right)  \in Z_{\operatorname{reg}}$, on en d\'{e}duit
que $\delta=p-q_{\infty}$ et que $G_{0,m}=0$ pour tout $m\in\mathbb{N}^{\ast}$.
\end{proof}

\begin{corollary}
\label{L/ serieLdeN}Les notations et les hypoth\`{e}ses sont celles de la
proposition~\ref{P/ DH1997}. Pour tout $k\in\mathbb{N}^{\ast}$, $N_{h,k}$ se
prolonge \`{a} $Z$ en une fonction holomorphe $N_{Q,k}$ qui ne d\'{e}pend pas
de $z_{\ast}$ et qui se d\'{e}veloppe en s\'{e}rie de Laurent sous la forme
$N_{Q,k}\left(  x,y\right)  =\sum\limits_{m\in\mathbb{N}^{\ast}}\frac
{N_{k,m}^{Q}\left(  x\right)  }{y^{m}}$ o\`{u} les $N_{k,m}^{Q}$ sont des
polyn\^{o}mes de degr\'{e} au plus $m$. En outre, pour tout $z$ dans
$Z_{\operatorname{reg}}$, il existe des ondes de choc $h_{1}^{z},...,h_{p}%
^{z}$ dont les images sont deux \`{a} deux distinctes et telles que pour
$z^{\prime}$ suffisamment voisin de $z$, $\left(  N_{Q,k}\left(  z^{\prime
}\right)  \right)  _{k\in\mathbb{N}}=\left(  N_{h^{z},k}\left(  z^{\prime
}\right)  \right)  _{k\in\mathbb{N}}$ et $L_{z^{\prime}}\cap Q=\left\{
\left(  1:h_{j}\left(  z^{\prime}\right)  :-x-yh_{j}\left(  z^{\prime}\right)
\right)  ;~1\leqslant j\leqslant p\right\}  $.
\end{corollary}

\begin{proof}
Soit $k\in\mathbb{N}$. On sait que $N_{h,k}=G_{k}-P_{k}$ sur $U_{\ast}$ et
gr\^{a}ce au lemme~\ref{L/ Pk} et au
corollaire~\ref{C/ confinementZeroBinfini} que $P_{k}$ est une fraction
rationnelle qui ne d\'{e}pend pas de $z_{\ast}$ et qui est d\'{e}finie sur
$Z$. Par cons\'{e}quent $N_{Q,k}=G_{k}-P_{k}$ prolonge $N_{h,k}$ en une
fonction holomorphe sur $Z$. En appliquant la proposition~\ref{P/ DH1997} et
le corollaire~\ref{C/ pCst} en un point $z$ arbitraire de
$Z_{\operatorname{reg}}\backslash E_{\infty} $, on obtient des ondes de chocs
$h_{1}^{z},...,h_{p}^{z}$ ayant les propri\'{e}t\'{e}s annonc\'{e}es. Par
ailleurs, le lemme~\ref{L/ Pk} donne aussi que
\[
P_{k}=\sum\limits_{0\leqslant k\leqslant m}p_{k,m}X^{m}=\frac{1}{\left(
k-1\right)  !}p_{1,1}^{\left(  k-1\right)  }X^{k}+\sum\limits_{0\leqslant
k\leqslant m}\frac{k}{m!\left(  k-m\right)  }p_{k-m,0}^{\left(  m\right)
}X^{m}%
\]
avec $p_{1,1}=\sum\limits_{q\in Q_{\infty}}\frac{b_{q}}{1+Yb_{q}}$ et
$p_{\nu,0}=\sum\limits_{j=1}^{\nu}\sum\limits_{q\in Q_{\infty}}\frac
{p_{\nu,0,j}^{q}}{\left(  1+Yb_{q}\right)  ^{j}}$. Pour $\left\vert
y\right\vert >\beta$, on obtient donc%
\begin{align*}
p_{1,1}\left(  y\right)   &  =\sum\limits_{n\in\mathbb{N}^{\ast}}\frac{\left(
-1\right)  ^{n-1}}{Y}\sum\limits_{q\in Q_{\infty}}b_{q}^{n-1}=\sum
\limits_{n\in\mathbb{N}^{\ast}}\frac{\left(  -1\right)  ^{n-1}S_{b,n-1}}%
{Y^{n}}\\
p_{\nu,0}\left(  y\right)   &  =\sum\limits_{j=1}^{\nu}\sum\limits_{n\in
\mathbb{N}^{\ast}}\frac{\left(  j-1\right)  !\left(  -1\right)  ^{n+j-1}%
}{Y^{n+j-1}}\sum\limits_{q\in Q_{\infty}}b_{q}^{-n}p_{\nu,0,j}^{q}%
=\sum\limits_{m\in\mathbb{N}^{\ast}}\frac{p_{\nu,0}^{\infty,m}}{Y^{m}}%
\end{align*}
avec $p_{\nu,0}^{\infty,m}=\left(  -1\right)  ^{m}\sum\limits_{\left(
n,j\right)  \in\mathbb{N}^{\ast}\times\left\{  1,...,\nu\right\}
,~n+j=m+1}\left(  j-1\right)  !\sum\limits_{q\in Q_{\infty}}b_{q}^{-n}%
p_{\nu,0,j}^{q}$. Il suffit alors de combiner ces formules avec le
lemme~\ref{L/ serieLdeG} pour obtenir les d\'{e}veloppements annonc\'{e}s.
\end{proof}

\subsection{Une gen\`{e}se des ondes de chocs multiples\label{S/ fabOM}}

On se donne $A,B\in\mathbb{C}\left[  Y\right]  $ avec $\deg A<r=\deg B$, $B$
valant $1$ en $0$ et dont les racines sont confin\'{e}es dans $\rho
\overline{\mathbb{D}}$ puis on d\'{e}finit $P\in\mathbb{C}\left[  X,Y\right)
$ et $N$ par les relations%
\[
P\left(  X,Y\right)  =\frac{A\left(  Y\right)  }{B\left(  Y\right)  }%
+\frac{B^{\prime}\left(  Y\right)  }{B\left(  Y\right)  }X~~\&~~N=G_{1}-P.
\]
Il s'agit dans cette section de savoir quand $N$ est une onde de chocs
multiples, c'est-\`{a}-dire une somme d'onde de choc. Le
th\'{e}or\`{e}me~4\ de~\cite{HeG-MiV2007} donne une caract\'{e}risation de
telles sommes mais dans cet article, on en utilise une plus adapt\'{e}e \`{a}
la situation pr\'{e}sente. Ces deux caract\'{e}risations correspondent plus ou
moins \`{a} privil\'{e}gier l'une ou l'autre des variables $x$ ou $y$ et
reposent sur le lemme~16 de~\cite{HeG-MiV2007} rappel\'{e} ci-dessous.

\begin{lemma}
[Henkin-Michel, 2007]\label{L/ EqDiffSym}Soit $D$ un domaine de $\mathbb{C}%
^{2}$ et $h_{1},...,h_{d}\in\mathcal{O}\left(  D\right)  $ des fonctions
mutuellement distinctes. Alors chaque $h_{j}$ est une onde de choc si et
seulement si les fonctions $\Sigma_{h,k}=\left(  -1\right)  ^{k}S_{h,k}$
satisfont au syst\`{e}me suivant%
\begin{equation}
\Sigma_{d}\frac{\partial\Sigma_{1}}{\partial x}+\frac{\partial\Sigma_{d}%
}{\partial y}=0\hspace{0.5cm}\&\hspace{0.5cm}\Sigma_{k}\frac{\partial
\Sigma_{1}}{\partial x}+\frac{\partial\Sigma_{k}}{\partial y}=\frac
{\partial\Sigma_{k+1}}{\partial x}~,~~1\leqslant k\leqslant d-1.
\label{F/ EqSym0}%
\end{equation}

\end{lemma}

Le corollaire suivant est contenu dans la preuve de la proposition~17
de~\cite{HeG-MiV2007}.

\begin{corollary}
\label{C/ p-onde}Soient $D$ un domaine de $\mathbb{C}^{2}$, $N\in
\mathcal{O}\left(  D\right)  $ et $d\in\mathbb{N}^{\ast}$. Il existe
localement des ondes de choc $h_{1},...,h_{d}$ mutuellement distinctes telles
que $N=h_{1}+\cdots+h_{d}$ si et seulement si il existe dans $\mathcal{O}%
\left(  D\right)  $ des fonctions $s_{1},...,s_{d}$ telles que $s_{1}=-N$ et
\begin{equation}
-s_{d}\frac{\partial N}{\partial x}+\frac{\partial s_{d}}{\partial
y}=0,\hspace{0.5cm}-s_{k}\frac{\partial N}{\partial x}+\frac{\partial s_{k}%
}{\partial y}=\frac{\partial s_{k+1}}{\partial x}~,~~1\leqslant k\leqslant
d-1, \label{F/ EqSym1}%
\end{equation}
et si le discriminant du polyn\^{o}me $\Sigma=T^{d}+s_{1}T^{d-1}+\cdots
+s_{d}\in\mathcal{O}\left(  D\right)  \left[  T\right]  $ n'est pas
identiquement nul sur $D$. Dans ce cas, on dit que $N$ est une $d$-onde de chocs.
\end{corollary}

\begin{proof}
Si $N$ est la somme de $d$ ondes de choc $h_{1},...,h_{d}$ mutuellement
distinctes, $s_{1}=-S_{h,1}=-N_{1}$ et les fonctions $s_{k}=\left(  -1\right)
^{k}S_{h,k}$, $2\leqslant k\leqslant d$ v\'{e}rifient (\ref{F/ EqSym1})
d'apr\`{e}s le lemme~\ref{L/ EqDiffSym}. R\'{e}ciproquement, supposons
$d\geqslant2$ et que $\left(  s_{1},...,s_{d}\right)  $ soit solution de
(\ref{F/ EqSym1}), $s_{1}=-N_{1}$ et que le discriminant $\Delta$ de $S$ n'est
pas nul. Pour tout $z_{\ast}\in\left\{  \Delta\neq0\right\}  $, $T_{z_{\ast}}$
n'a que des racines simples et il existe un voisinage $U_{\ast}$ de $z_{\ast}$
dans $D$ ainsi que des fonctions holomorphes $h_{1},...,h_{d}$ telles que pour
tout $z\in U_{\ast}$, les racines de $\Sigma_{z}$ sont les $d$ nombres
complexes $h_{1}\left(  z\right)  ,...,h_{d}\left(  z\right)  $. Le
lemme~\ref{L/ EqDiffSym} donne alors que $h_{1},...,h_{d}$ sont des ondes de choc.
\end{proof}

La d\'{e}finition ci-dessous introduit un op\'{e}rateur
int\'{e}gro-diff\'{e}rentiel adapt\'{e} \`{a} la r\'{e}solution du
syst\`{e}me~(\ref{F/ EqSym1}).

\begin{definition}
\label{D/ PJH}On s\'{e}lectionne dans le cercle unit\'{e} $\mathbb{T}$ de
$\mathbb{C}$ une direction $\tau$ et dans $\Omega_{\rho,\tau}=\mathbb{C}%
\backslash\left(  \rho\overline{\mathbb{D}}\cup\mathbb{R}_{+}\tau\right)  $,
on fixe un point $\omega$. On d\'{e}finit alors un op\'{e}rateur de
primivitisation $\mathcal{P}=\mathcal{P}_{\rho,\tau,\omega}$ sur $\Omega
_{\rho,\tau}$ en associant \`{a} une fonction holomorphe sur $\Omega
_{\rho,\tau}$ sa primitive qui s'annule en $\omega$ et on pose
\[
J=\mathcal{P}_{0,\tau,\omega}\frac{1}{Id_{\Omega_{0,\tau}}}%
\]
On fixe $\rho_{\omega}\in\left]  \rho,+\infty\right[  $ tel que $m\left(
y\right)  \geqslant m\left(  \omega\right)  =m_{\omega}$ lorsque $\left\vert
y\right\vert \geqslant\rho_{\omega}$ puis on pose $D_{\omega}=D\left(
0,\frac{1}{2}m_{\omega}\right)  $, $Z_{\omega}=D_{\omega}\times\left(
\mathbb{C}\backslash\rho_{\omega}\overline{\mathbb{D}}\right)  $ et
$Z_{\omega,\tau}=Z_{\omega}\backslash\mathbb{R}_{+}\tau$. On d\'{e}finit les
op\'{e}rateurs int\'{e}gro-diff\'{e}rentiels $\mathcal{D}$ et $\mathcal{E}$
agissant sur $\mathcal{O}\left(  Z_{\omega,\tau}\right)  $ par%
\[
\mathcal{D}=e^{-H}\frac{\partial}{\partial x}e^{H}=\frac{\partial}{\partial
x}+\frac{\partial H}{\partial x}~~\&~~\mathcal{E}=\mathcal{P}_{\rho_{\omega
},\tau,\omega}\circ\mathcal{D}%
\]
o\`{u} les fonctions $H$ et $\widetilde{H}$ sont d\'{e}finies sur $Z$ par%
\[
H=-\delta\otimes J-\sum\limits_{m\geqslant1}\frac{G_{1,m+1}^{\prime}\left(
x\right)  }{m}\frac{1}{y^{m}}=-\delta\otimes J+\widetilde{H}.
\]

\end{definition}

Etant donn\'{e} que d'apr\`{e}s (\ref{F/ devG1}), $G_{1}\left(  x,y\right)
=\frac{G_{1,1}^{0}-\delta x}{y}+\sum\limits_{m\geqslant2}\frac{G_{k,m}\left(
x\right)  }{y^{m}}$ sur $Z$, on obtient que sur $Z$%
\[
\frac{\partial H}{\partial y}=\frac{\partial G_{1}}{\partial x}%
\]
et aussi, puisque $\delta\in\mathbb{Z}$,%
\begin{equation}
e^{H}=\omega^{\delta}\left(  1\otimes Y^{-\delta}\right)  e^{\widetilde{H}}
\label{F/ exp(-H)}%
\end{equation}
Ainsi, $e^{H}$ se prolonge holomorphiquement \`{a} $Z$ ind\'{e}pendamment de
la coupure faite selon $\mathbb{R}_{+}\tau$. Une autre cons\'{e}quence de
(\ref{F/ exp(-H)}) est que $\delta$ est donn\'{e} pour tout $x\in\mathbb{C}$
par la formule%
\begin{equation}
\delta=\underset{\left\vert y\right\vert \rightarrow+\infty}{\lim}\frac
{\ln\left\vert e^{-H\left(  x,y\right)  }\right\vert }{\ln\left\vert
y\right\vert }. \label{F/ delta}%
\end{equation}

\begin{proposition}
\label{P/ s=>Mu}Soient $s_{1},...,s_{d}\in\mathcal{O}\left(  Z_{\omega
}\right)  $. Alors $\left(  s_{1},...,s_{d}\right)  $ est solution de
(\ref{F/ EqSym1}) avec $N=G_{1}-P$ si et seulement si il existe dans
$\mathcal{O}\left(  D_{\omega}\right)  $ des fonctions $\mu_{1},...,\mu_{d}$
telles que pour tout $\tau\in\mathbb{T}$, le syst\`{e}me ci-dessous est
v\'{e}rifi\'{e} sur $Z_{\omega,\tau}$
\begin{equation}
\left(  1\otimes B\right)  s_{k}=\left(  \mathcal{E}^{0}\left(  \mu_{k}%
\otimes1\right)  +\cdots+\mathcal{E}^{d-k}\left(  \mu_{d}\otimes1\right)
\right)  e^{H},~d\geqslant s\geqslant1 \label{F/ systMu}%
\end{equation}

\end{proposition}

\begin{proof}
Afin de simplifier les \'{e}critures dans cette preuve, on \'{e}crit $B$ pour
$1\otimes B$ et $\mu_{j}\otimes1$ pour $\mu_{j}$ . Puisque $N=G_{1}-\frac
{A}{B}-x\otimes\frac{B^{\prime}}{B}$, on remarque que si $s\in\mathcal{O}%
\left(  Z_{\omega,\tau}\right)  $,
\begin{align*}
B\left(  -s\frac{\partial N}{\partial x}+\frac{\partial s}{\partial y}\right)
&  =s\left(  -B\frac{\partial G_{1}}{\partial x}+B^{\prime}\right)
+B\frac{\partial s}{\partial y}\\
&  =-\left(  Bs\right)  \frac{\partial G_{1}}{\partial x}+\frac{\partial
Bs}{\partial y}=e^{H}\frac{\partial e^{-H}Bs}{\partial y}%
\end{align*}
Comme $e^{H}$ se prolonge holomorphiquement \`{a} $Z_{\omega}$
ind\'{e}pendamment de $\tau$, $\left(  s_{1},...,s_{d}\right)  $ est solution
de (\ref{F/ EqSym1}) si et seulement si les \'{e}quations%
\begin{equation}
\frac{\partial e^{-H}Bs_{d}}{\partial y}=0\hspace{0.5cm}\&\hspace{0.5cm}%
\frac{\partial e^{-H}Bs_{k}}{\partial y}=e^{-H}\frac{\partial Bs_{k+1}%
}{\partial x}~,~~1\leqslant k\leqslant d-1 \label{F/ EqSym2}%
\end{equation}
sont satisfaites sur $Z_{\omega}$. La premi\`{e}re est \'{e}quivalente \`{a}
l'existence d'une fonction $\mu_{d}$ d\'{e}finie sur $D_{\omega}$ telle que
pour tout $\left(  x,y\right)  \in Z_{\omega}$,
\[
B\left(  y\right)  s_{d}\left(  x,y\right)  =\mu_{d}\left(  x\right)
e^{H\left(  x,y\right)  }%
\]
Une telle fonction $\mu_{d}$ est en fait holomorphe sur $D_{\omega}$ puisque
pour tout $y\in\Omega_{\tau}\backslash\rho\overline{\mathbb{D}}$ tel que
$B\left(  y\right)  \neq0$, elle serait donn\'{e}e sur $D\left(  0,m_{\omega
}\right)  $ par la formule $\mu_{d}=s_{d}\left(  .,y\right)  \frac{e^{H\left(
.,y\right)  }}{B\left(  y\right)  }$. Supposons que pour $k\in\left\{
1,...,d-1\right\}  $, $\mu_{d},...,\mu_{k}\in\mathcal{O}\left(  D_{\omega
}\right)  $ v\'{e}rifient sur $Z_{\omega,\tau}$
\[
Bs_{j}=\left(  \mathcal{E}^{0}\left(  \mu_{j}\right)  +\cdots+\mathcal{E}%
^{d-j}\left(  \mu_{d}\right)  \right)  e^{H}%
\]
lorsque $d\geqslant j\geqslant k+1$. L'\'{e}quation $\frac{\partial}{\partial
y}\left(  Bs_{k}e^{-H}\right)  =e^{-H}\frac{\partial}{\partial x}\left(
Bs_{k+1}\right)  $ est alors \'{e}quivalente \`{a} l'existence de $\mu_{k}%
\in\mathcal{O}\left(  D_{\omega}\right)  $ telle que pour tout $\left(
x,y\right)  \in Z_{\omega,\tau}$,%
\[
B\left(  y\right)  s_{k}\left(  x,y\right)  e^{-H\left(  x,y\right)  }=\mu
_{k}\left(  x\right)  +\left(  e^{-H}\frac{\partial}{\partial x}\left(
Bs_{k+1}\right)  \right)  _{0,-1}\left(  x,y\right)  .
\]
Par hypoth\`{e}se,%
\[
Bs_{k+1}=\sum\limits_{k+1\leqslant j\leqslant d}\left(  \mathcal{E}^{j-k-1}%
\mu_{j}\right)  e^{H}.
\]
Donc%
\[
\left(  e^{-H}\left(  Bs_{k+1}\right)  _{x}\right)  _{0,-1}=\sum
\limits_{k+1\leqslant j\leqslant d}\left(  e^{-H}\frac{\partial}{\partial
x}\left(  e^{H}\mathcal{E}^{j-k-1}\mu_{j}\right)  \right)  _{0,-1}%
=\sum\limits_{k+1\leqslant j\leqslant d}\mathcal{E}^{j-k}\mu_{j}.
\]
Par cons\'{e}quent, l'\'{e}quation $\frac{\partial}{\partial y}\left(
Bs_{k}e^{-H}\right)  =e^{-H}\frac{\partial}{\partial x}\left(  Bs_{k+1}%
\right)  $ est \'{e}quivalente \`{a} l'existence de $\mu_{k}\in\mathcal{O}%
\left(  D_{\omega}\right)  $ telle que pour tout $\left(  x,y\right)  \in
Z_{\omega,\tau}$,
\[
B\left(  y\right)  s_{k}\left(  x,y\right)  =\left(  \mu_{k}\left(  x\right)
+\mathcal{E}^{1}\mu_{k+1}\left(  x,y\right)  +\cdots+\mathcal{E}^{d-k}\mu
_{d}\left(  x,y\right)  \right)  e^{H\left(  x,y\right)  }%
\]

\end{proof}

La proposition~\ref{P/ s=>Mu} donne un proc\'{e}d\'{e} simple pour construire
a priori des fonctions susceptibles d'\^{e}tre des ondes de chocs multiples.
On note $\mathbb{C}_{\mathcal{B}}\left[  Y\right]  $ l'ensemble des
polyn\^{o}mes $B\in\mathbb{C}\left[  Y\right]  $ valant $1$ en $0$ et dont les
racines sont confin\'{e}es dans $\rho\overline{\mathbb{D}}$.

\begin{corollary}
\label{C/ constructionOndeMult}Soient $\mu_{1},...,\mu_{d}\in\mathcal{O}%
\left(  D_{\omega}\right)  $. On d\'{e}finit des fonctions $s_{k}\left(
\mu,B\right)  $, $1\leqslant k\leqslant d$, holomorphes sur $Z_{\omega,\tau}$
par les formules
\begin{gather*}
\mathcal{E}_{k}\left(  \mu\right)  =\sum\limits_{j=k}^{d}\mathcal{E}%
^{j-k}\left(  \mu_{j}\otimes1\right) \\
s_{k}\left(  \mu,B\right)  =\frac{e^{\widetilde{H}}}{1\otimes\left(
Y/\omega\right)  ^{\delta}B}\mathcal{E}_{k}\left(  \mu\right)  ,~1\leqslant
k\leqslant d,\\
S\left(  \mu,B\right)  =T^{d}+s_{1}\left(  \mu,B\right)  T^{d-1}+\cdots
+s_{d}\left(  \mu,B\right)  \in\mathcal{O}\left(  D_{\omega}\right)  \left[
T\right]
\end{gather*}
Alors l'application $\mathcal{O}\left(  D_{\omega}\right)  ^{d}\times
\mathbb{C}_{\mathcal{B}}\left[  Y\right]  \ni\left(  \mu,B\right)
\mapsto\left(  s_{k}\left(  \mu,B\right)  \right)  _{1\leqslant k\leqslant d}$
est injective. En outre $-s_{1}\left(  \mu,B\right)  $ est une $d$-onde de
chocs sur $Z_{\omega,\tau}$ si et seulement si%
\[
-s_{1}\left(  \mu,B\right)  =G_{1}-P
\]
et le discriminant $\Delta\left(  \mu,B\right)  $ de $S\left(  \mu,B\right)  $
n'est pas identiquement nul. Quand cette condition est v\'{e}rifi\'{e}e, on
dit que $-s_{1}\left(  \mu,B\right)  $ est une $\left(  d,B\right)  $-onde de chocs.
\end{corollary}

\begin{proof}
Supposons que $\left(  \mu,B\right)  $ et $\left(  \nu,C\right)  $ sont deux
\'{e}l\'{e}ments de $\mathcal{O}\left(  \mathbb{C}\right)  ^{d}\times
\mathbb{C}_{\mathcal{B}}\left[  Y\right]  $ tels que $\left(  s_{k}\left(
\mu,B\right)  \right)  _{1\leqslant k\leqslant n}=\left(  s_{k}\left(
\nu,C\right)  \right)  _{1\leqslant k\leqslant d}$. Alors sur $Z_{\omega,\tau
}$%
\[
\frac{\mu_{d}\otimes1}{1\otimes\left(  Y/\omega\right)  ^{\delta}B}=\frac
{\nu_{d}\otimes1}{1\otimes\left(  Y/\omega\right)  ^{\delta}C}%
\]
Comme $B,C\in\mathbb{C}_{\mathcal{B}}\left[  Y\right]  $, ceci impose $B=C$ et
$\mu_{d}=\nu_{d}$. Supposons que $\mu_{j}=\nu_{j}$ lorsque $d\geqslant
j\geqslant k>1$. La relation $s_{k-1}\left(  \mu,B\right)  =s_{k-1}\left(
\nu,C\right)  $ s'\'{e}crit alors $\mathcal{E}_{k-1}\left(  \mu\right)
=\mathcal{E}_{k-1}\left(  \nu\right)  $ et celle-ci livre imm\'{e}diatement
$\mu_{k-1}=\nu_{k-1}$. D'o\`{u} $\mu=\nu$.

Puisque $e^{H}=\left(  1\otimes\left(  Y/\omega\right)  ^{-\delta}\right)
e^{\widetilde{H}}$, il r\'{e}sulte de la proposition~\ref{P/ s=>Mu} que
$\left(  s_{k}\left(  \mu,B\right)  \right)  _{1\leqslant k\leqslant d}$
v\'{e}rifie le syst\`{e}me (\ref{F/ EqSym1}). Lorsque $-s_{1}\left(
\mu,B\right)  =G_{1}-P$, le syst\`{e}me (\ref{F/ EqSym1}) est identique au
syst\`{e}me (\ref{F/ EqSym0}) et $\Delta\left(  \mu,B\right)  \neq0$ assure
que $-s_{1}\left(  \mu,B\right)  $ est la somme de $d$ ondes de chocs deux
\`{a} deux diff\'{e}rentes dont les fonctions sym\'{e}triques sont les
$\left(  -1\right)  ^{k}s_{k}\left(  \mu,B\right)  $.
\end{proof}

La proposition ci-dessous donne une d\'{e}composition des op\'{e}rateurs
it\'{e}r\'{e}s $\mathcal{E}^{k}$.

\begin{proposition}
\label{P/ Ekj}On d\'{e}finit des fonctions $\mathcal{E}_{k,k},...,\mathcal{E}%
_{k,0}$ holomorphes sur $\mathcal{Z}_{\omega,\tau}$ pour tout $k\in\mathbb{N}
$ par les relations suivantes%
\begin{gather*}
\mathcal{E}_{k,k}=1\otimes\frac{\left(  Y-\omega\right)  ^{k}}{k!}%
,~~\mathcal{E}_{k+1,0}=\mathcal{E}^{k}\mathcal{P}\frac{\partial H}{\partial
x}\\
\mathcal{E}_{k+1,j}=\mathcal{PE}_{k,j-1}+\mathcal{EE}_{k,j},~1\leqslant
j\leqslant k
\end{gather*}
o\`{u} $\mathcal{E}_{k,\nu}=0$ si $\nu<0$. Alors pour tout $f\in
\mathcal{O}\left(  D_{\omega}\right)  $,%
\[
\mathcal{E}^{k}\left(  f\otimes1\right)  =\sum\limits_{0\leqslant j\leqslant
k}\left(  f^{\left(  j\right)  }\otimes1\right)  \mathcal{E}_{k,j}%
\]
En outre, les fonctions $\mathcal{E}_{k,j}$ ont le long de $\mathbb{R}\tau$ au
plus des discontinuit\'{e}s de premi\`{e}re esp\`{e}ce.
\end{proposition}

\begin{proof}
Par d\'{e}finition pour tout $f\in\mathcal{O}\left(  D_{\omega}\right)  $,
$\mathcal{D}\left(  f\otimes1\right)  =f^{\prime}\otimes1+\left(
f\otimes1\right)  \frac{\partial H}{\partial x}$ et donc $\mathcal{E}f=\left(
f^{\prime}\otimes1\right)  \mathcal{E}_{1,1}+\left(  f\otimes1\right)
\mathcal{E}_{1,0}$ avec $\mathcal{E}_{1,1}=1\otimes\left(  Y-\omega\right)  $
et $\mathcal{E}_{1,0}=\mathcal{P}H$. Par ailleurs pour tout $\left(
x,y\right)  \in Z_{\omega,\tau}$
\[
H\left(  x,y\right)  =G_{1,-1}^{\leqslant3}\left(  x,y\right)  -\frac
{G_{2}^{\prime}\left(  x\right)  }{y}+G_{1}^{\prime}\left(  x\right)
J_{-1}\left(  x,ny\right)
\]
avec $G_{1,-1}^{\leqslant3}\left(  x,y\right)  =\sum\limits_{n\leqslant3}%
\frac{G_{n}^{\prime}\left(  x\right)  /\left(  1-n\right)  }{y^{n-1}}$.
$\mathcal{P}G_{1,-1}^{\leqslant3}\left(  \left(  x,y\right)  \right)
=\sum\limits_{n\leqslant3}\frac{G_{n}^{\prime}\left(  x\right)  /\left(
1-n\right)  \left(  2-n\right)  }{y^{n-2}}$ se prolonge holomorphiquement
\`{a} $Z_{\omega}$. Puisque $J_{-1}$ a une discontinuit\'{e} de premi\`{e}re
esp\`{e}ce le long de $\mathbb{R}\tau$, il en est de m\^{e}me de
$\mathcal{P}J_{-1}$. Par cons\'{e}quent, $\mathcal{E}_{1,0}=\mathcal{P}H$ a au
plus des discontinuit\'{e}s de premi\`{e}re esp\`{e}ce le long de
$\mathbb{R}\tau$. Supposons le r\'{e}sultat du lemme vrai pour $k\in
\mathbb{N}^{\ast}$ donn\'{e}. Alors pour $f\in\mathcal{O}\left(  D_{\omega
}\right)  $%
\begin{align*}
&  \mathcal{E}^{k+1}\left(  f\otimes1\right) \\
&  =%
{\displaystyle\sum\limits_{0\leqslant j\leqslant k}}
\mathcal{P}\frac{\partial}{\partial x}\left(  f^{\left(  j\right)  }%
\otimes1\right)  \mathcal{E}_{k,j}+\mathcal{P}\left(  \frac{\partial
H}{\partial x}%
{\displaystyle\sum\limits_{0\leqslant j\leqslant k}}
\left(  f^{\left(  j\right)  }\otimes1\right)  \mathcal{E}_{k,j}\right) \\
&  =%
{\displaystyle\sum\limits_{0\leqslant j\leqslant k}}
\mathcal{P}\left(  \left(  f^{\left(  j+1\right)  }\otimes1\right)
\mathcal{E}_{k,j}+\left(  f^{\left(  j\right)  }\otimes1\right)
\frac{\partial\mathcal{E}_{k,j}}{\partial x}\right)  +%
{\displaystyle\sum\limits_{0\leqslant j\leqslant k}}
\left(  f^{\left(  j\right)  }\otimes1\right)  \mathcal{P}\left(
\mathcal{E}_{k,j}\frac{\partial H}{\partial x}\right) \\
&  =%
{\displaystyle\sum\limits_{0\leqslant j\leqslant k}}
\left(  f^{\left(  j+1\right)  }\otimes1\right)  \mathcal{PE}_{k,j}+%
{\displaystyle\sum\limits_{0\leqslant j\leqslant k}}
\left(  f^{\left(  j\right)  }\otimes1\right)  \mathcal{P}\frac{\partial
\mathcal{E}_{k,j}}{\partial x}+%
{\displaystyle\sum\limits_{0\leqslant j\leqslant k}}
\left(  f^{\left(  j\right)  }\otimes1\right)  \mathcal{P}\left(
\mathcal{E}_{k,j}\frac{\partial H}{\partial x}\right)
\end{align*}
ce qui donne la formule attendue avec%
\begin{align*}
\mathcal{E}_{k+1,k+1}  &  =\mathcal{PE}_{k,k}=\frac{1}{k!}\mathcal{P}\left(
1\otimes\left(  Y-\omega\right)  ^{k}\right)  =\frac{1}{\left(  k+1\right)
!}1\otimes\left(  Y-\omega\right)  ^{k+1},\\
\mathcal{E}_{k+1,j}  &  =\mathcal{PE}_{k,j-1}+\mathcal{P}\left(
\frac{\partial\mathcal{E}_{k,j}}{\partial x}+\mathcal{E}_{k,j}\frac{\partial
H}{\partial x}\right)  =\mathcal{PE}_{k,j-1}+\mathcal{EE}_{k,j},~1\leqslant
j\leqslant k,\\
\mathcal{E}_{k+1,0}  &  =\mathcal{P}\left(  \frac{\partial\mathcal{E}_{k,0}%
}{\partial x}+\mathcal{E}_{k,0}\frac{\partial H}{\partial x}\right)
=\mathcal{EE}_{k,0}=\mathcal{EE}^{k-1}\mathcal{P}\frac{\partial H}{\partial
x}=\mathcal{E}^{k}\mathcal{P}\frac{\partial H}{\partial x}.
\end{align*}
Le m\^{e}me type de raisonnement que celui fait pour $k=1$ livre que les
$\mathcal{E}_{k+1,j}$ ont au plus des discontinuit\'{e}s de premi\`{e}re
esp\`{e}ce le long de $\mathbb{R}\tau$.
\end{proof}

\subsection{Syst\`{e}mes diff\'{e}rentiels}

Comme dans la section pr\'{e}c\'{e}dente, on se donne $A,B\in\mathbb{C}\left[
Y\right]  $ avec $\deg A<r=\deg B$ et $B=\underset{1\leqslant q\leqslant
r}{\Pi}\left(  1+s_{q}Y\right)  =\sum\limits_{0\leqslant j\leqslant r}%
\beta_{j}Y^{j}$ tel que $\left\{  B=0\right\}  \subset\rho\overline
{\mathbb{D}}$. On d\'{e}finit $P\in\mathbb{C}\left[  X,Y\right)  $ et
$N\in\mathcal{O}\left(  Z\right)  $ par les relations%
\[
P=1\otimes\frac{A}{B}+X\otimes\frac{B^{\prime}}{B}=f_{1,0}\left(  Y\right)
+f_{1,1}\left(  Y\right)  X~~~\text{et~~~}N=G_{1}-P.
\]
Puisque le degr\'{e} $q_{\infty}$ de $B_{\infty}$ et le nombre $p$ des ondes
de chocs sont li\'{e}s par la relation $p=q_{\infty}+\delta$ o\`{u} $\delta$
est connu gr\^{a}ce \`{a} (\ref{F/ delta}), on pose%
\[
d=r+\delta.
\]
Le corollaire~\ref{C/ constructionOndeMult} permet de fabriquer \`{a} partir
d'un jeu $\mu_{1},..,\mu_{d}$ de fonctions holomorphes sur $D_{\omega}$ un
polyn\^{o}me $T^{d}+\sum\limits_{1\leqslant k\leqslant d}s_{k}\left(
\mu,B\right)  T^{d-k}$ dont les racines sont sur $\left\{  \Delta\left(
\mu,B\right)  \neq0\right\}  $ des ondes de chocs pourvu que $-s_{1}\left(
\mu,B\right)  =N=G_{1}-P$. La proposition ci-dessous montre que cette
condition est \'{e}quivalente \`{a} ce que les $\mu_{j}$ v\'{e}rifient un
syst\`{e}me d'\'{e}quations lin\'{e}aires avec second membre explicite.

\begin{proposition}
\label{P/ systLinMu}Soit $\mu\in\mathcal{O}\left(  D_{\omega}\right)  ^{d}$.

1. On note $(\sum\limits_{n\in\mathbb{Z}}K_{n}^{0}\left(  B,A,x\right)  y^{n})
$ le d\'{e}veloppement en s\'{e}rie de Laurent par rapport \`{a} $y$ de
$\left[  R-\left(  1\otimes B\right)  G_{1}\right]  e^{-H}$. Alors
$G_{1}+s_{1}\left(  \mu,B\right)  =R$ sur $Z_{\omega}$ si et seulement si
$\mu$ est solution sur $D_{\omega}$ du syst\`{e}me diff\'{e}rentiel%
\begin{equation}
\forall n\in\mathbb{Z},~\sum\limits_{0\leqslant m<j\leqslant d}c_{j,m}%
^{0,n}\mu_{j}^{\left(  m\right)  }=K_{n}^{0}\left(  B,A,.\right)  \label{E0}%
\end{equation}
o\`{u} pour $n\in\mathbb{Z}$, $0\leqslant m<j\leqslant d$ et $R>\rho_{\omega}%
$,%
\[
c_{j,m}^{0,n}=\frac{1}{2\pi i}\int_{R\mathbb{T}}\mathcal{E}_{j-1,m}\frac
{dy}{y^{n+1}}%
\]
En outre les fonctions $K_{n}^{0}\left(  B,A,.\right)  $ sont nulles pour
$n\geqslant d$ et lin\'{e}aires par rapport \`{a} $\left(  A,B\right)  $ avec
des coefficients qui ne d\'{e}pendent que de $G$ .

2. L'\'{e}quation $\frac{\partial G_{1}}{\partial x}+\frac{\partial
s_{1}\left(  \mu,B\right)  }{\partial x}=1\otimes\frac{B^{\prime}}{B}$ est
\'{e}quivalente \`{a} ce que $\mu$ est solution sur $D_{\omega}$ du
syst\`{e}me diff\'{e}rentiel
\begin{equation}
\forall n\in\mathbb{Z},~\sum\limits_{j=1}^{d}\sum\limits_{m=0}^{j}%
c_{j,m}^{1,n}\mu_{j}^{\left(  m\right)  }=K_{n}^{1}\left(  B,.\right)
\label{E1}%
\end{equation}
o\`{u} pour $n\in\mathbb{Z}$, $1\leqslant j\leqslant d$, $0\leqslant
m\leqslant j$ et $R>\rho$,%
\begin{gather*}
K_{n}^{1}\left(  B,.\right)  =\frac{\omega^{g}}{2\pi i}\int_{R\mathbb{T}%
}\left(  B^{\prime}\left(  y\right)  -B\left(  y\right)  \frac{\partial G_{1}%
}{\partial x}\left(  .,y\right)  \right)  \frac{e^{\widetilde{H}\left(
.,y\right)  }dy}{y^{n+1+\delta}}\\
c_{j,m}^{1,n}=\frac{1}{2\pi i}\int_{R\mathbb{T}}\mathcal{E}_{j,m}^{1}\frac
{dy}{y^{n+1}},
\end{gather*}
les fonctions $\mathcal{E}_{j,m}^{1}$ \'{e}tant d\'{e}termin\'{e}es par les
relations%
\begin{gather*}
\mathcal{E}_{j,j}^{1}=\mathcal{E}_{j-1,j-1}=1\otimes\frac{\left(
Y-\omega\right)  ^{j-1}}{\left(  j-1\right)  !},~~\mathcal{E}_{j,0}%
^{1}=\mathcal{DE}_{j-1,0}\\
\mathcal{EE}_{j,m}^{1}=\mathcal{E}_{j-1,m-1}+\widetilde{\mathcal{D}%
}\mathcal{E}_{j-1,m},~1\leqslant m\leqslant j,
\end{gather*}
En outre, $K_{n}^{1}\left(  B,.\right)  =0$ lorsque $n\geqslant d$.

3. On suppose $\frac{\partial^{2}G_{1}}{\partial x^{2}}\neq0$ et on pose
$B=\sum\limits_{0\leqslant j\leqslant r}\beta_{j}Y^{j}$. L'\'{e}quation
$\frac{\partial^{2}G_{1}}{\partial^{2}x}+\frac{\partial^{2}s_{1}\left(
\mu,B\right)  }{\partial x^{2}}=0$ est \'{e}quivalente \`{a} ce que $\mu$ est
solution sur $D_{\omega}$ du syst\`{e}me diff\'{e}rentiel%
\begin{equation}
\forall n\in\mathbb{Z},~\beta_{n}\mathbf{1}_{\left\{  \delta,...,d\right\}
}\left(  n\right)  =\sum\limits_{j=1}^{d}\sum\limits_{m=0}^{j+1}c_{j,m}%
^{2,n}\mu_{j}^{\left(  m\right)  }. \label{E2}%
\end{equation}
o\`{u}
\[
c_{j,m}^{2,n}=\frac{1}{2\pi i}\int_{R\mathbb{T}}\frac{\mathcal{E}_{j,m}^{2}%
}{\partial^{2}G_{1}/\partial x^{2}}\left(  .,y\right)  \frac{dy}{y^{n+1}},
\]
les fonctions $\mathcal{E}_{j,m}^{2}$ \'{e}tant d\'{e}termin\'{e}es par les
formules suivantes
\[
\mathcal{E}_{j,m}^{2}=\mathcal{E}_{j-1,m-2}+2\mathcal{DE}_{j-1,m-1}%
+\frac{\partial^{2}\mathcal{E}_{j-1,m}}{\partial x^{2}}+2\frac{\partial
\widetilde{H}}{\partial x}\frac{\partial\mathcal{E}_{j-1,m}}{\partial
x}+\left(  \left(  \frac{\partial\widetilde{H}}{\partial x}\right)  ^{2}%
+\frac{\partial^{2}\widetilde{H}}{\partial x^{2}}\right)  \mathcal{E}_{j-1,m}%
\]
avec la convention que $\mathcal{E}_{\alpha,\beta}=0$ si $m\notin\left[
2,j-1\right]  $
\end{proposition}

\begin{proof}
\textbf{1. }D'apr\`{e}s les d\'{e}finitions donn\'{e}es dans le
corollaire~\ref{C/ constructionOndeMult}, $s_{1}\left(  \mu,B\right)
=L\mathcal{E}_{1}\left(  \mu\right)  $. La proposition~\ref{P/ Ekj} donne
\[
\mathcal{E}_{1}\left(  \mu\right)  =\sum\limits_{1\leqslant j\leqslant
d}\mathcal{E}^{j-1}\left(  \mu_{j}\otimes1\right)  =\sum\limits_{1\leqslant
j\leqslant d}\sum\limits_{0\leqslant m\leqslant j-1}\left(  \mu_{j}^{\left(
m\right)  }\otimes1\right)  \mathcal{E}_{j-1,m}%
\]
Donc $G+s_{1}\left(  \mu,B\right)  =R$ si et seulement si%
\begin{equation}
\omega^{\delta}\sum\limits_{0\leqslant m<j\leqslant d}\left(  \mu_{j}^{\left(
m\right)  }\otimes1\right)  \mathcal{E}_{j-1,m}=\left(  X\otimes Y^{\delta
}B^{\prime}\right)  e^{-\widetilde{H}}+\left(  1\otimes Y^{\delta}A\right)
e^{-\widetilde{H}}-\left(  1\otimes Y^{\delta}B\right)  G_{1}e^{-\widetilde{H}%
}. \label{F/ dev0}%
\end{equation}
o\`{u} $A=\sum\limits_{0\leqslant j\leqslant r-1}a_{j}Y^{j}$. Le
d\'{e}veloppement en s\'{e}rie Laurent de $e^{-\widetilde{H}}$ \'{e}tant de la
forme $e^{-\widetilde{H}}=\sum\limits_{n\in\mathbb{Z}_{-}}\widetilde{E}%
_{n}\left(  x\right)  y^{n}$ o\`{u} les $\widetilde{E}_{n}$ sont des fonctions
holomorphes sur $D_{\omega}$, il vient, puisque $r+\delta=d$, que pour tout
$\left(  x,y\right)  \in Z$
\begin{align*}
xy^{\delta}B^{\prime}\left(  y\right)  e^{-\widetilde{H}\left(  x,y\right)  }
&  =x\sum\limits_{\nu\leqslant0}\sum\limits_{\delta\leqslant j<d}\left(
j-\delta+1\right)  \beta_{j-\delta+1}\widetilde{E}_{\nu}\left(  x\right)
y^{j+\nu}\\
&  =\sum\limits_{n\leqslant d-1}y^{n}\sum\limits_{n-d<\nu\leqslant\min\left(
n-\delta,0\right)  }\left(  n-\nu-\delta+1\right)  \beta_{n-\nu-\delta
+1}x\widetilde{E}_{\nu}\left(  x\right)
\end{align*}
et
\[
y^{\delta}A\left(  y\right)  e^{-\widetilde{H}\left(  x,y\right)  }%
=\sum\limits_{n\leqslant d-1}y^{n}\sum\limits_{n-d<\nu\leqslant\min\left(
n-\delta,0\right)  }a_{n-\nu-\delta}\widetilde{E}_{\nu}\left(  x\right)
\]
$G_{1}e^{-\widetilde{H}}$ se d\'{e}veloppant sous la forme $G_{1}%
e^{-\widetilde{H}}=\sum\limits_{n\in\mathbb{Z}_{-}^{\ast}}\widehat{E}%
_{n}\otimes Y^{n}=\frac{1}{y}\sum\limits_{n\in\mathbb{Z}_{-}}\widehat{E}%
_{n-1}\left(  x\right)  y^{n}$, $\left(  1\otimes Y^{\delta}B\right)
G_{1}e^{-\widetilde{H}}=\left(  1\otimes Y^{\delta-1}B\right)  \sum
\limits_{n\in\mathbb{Z}_{-}}\widehat{E}_{n-1}\otimes Y^{n}$ et on a de
m\^{e}me
\[
y^{\delta}B\left(  y\right)  G_{1}\left(  x,y\right)  e^{-\widetilde{H}\left(
x,y\right)  }=\sum\limits_{n\leqslant d-1}y^{n}\sum\limits_{n-d<\nu
\leqslant\min\left(  n-\delta+1,0\right)  }\beta_{n-\nu-\delta+1}%
\widehat{E}_{\nu-1}\left(  x\right)
\]
Ainsi, les coefficients $\omega^{-\delta}K_{n}^{0}\left(  B,A,x\right)  $ du
d\'{e}veloppement du membre de droite de (\ref{F/ dev0}) sont nuls si
$n\geqslant d$ et sinon donn\'{e}s par les formules,%
\[
K_{n}^{0}\left(  B,A,x\right)  =\omega^{-\delta}\sum\limits_{\nu=n-d+1}%
^{\min\left(  n-\delta,0\right)  }\left(  \left[  \left(  n-\nu-\delta
+1\right)  \beta_{n-\nu-\delta+1}x+a_{n-\nu-\delta}\right]  \widetilde{E}%
_{\nu}\left(  x\right)  +\beta_{n-\nu-\delta+1}\widehat{E}_{\nu-1}\left(
x\right)  \right)  .
\]
Si $R$ est assez grand, (\ref{F/ dev0}) est \'{e}quivalent au fait que pour
tout $n\in\mathbb{Z}$,
\[
\sum\limits_{0\leqslant m<j\leqslant d}\mu_{j}^{\left(  m\right)  }\frac
{1}{2\pi}\int_{R\mathbb{T}}\mathcal{E}_{j-1,m}\frac{dy}{y^{n+1}}=K_{n}%
^{0}\left(  B,A,x\right)  ,
\]
l'existence des int\'{e}grales \'{e}tant assur\'{e}e par le fait que les
$\mathcal{E}_{k,j}$ ont plus des discontinuit\'{e}s de premi\`{e}re esp\`{e}ce
le long de $\mathbb{R}\tau$. Ceci prouve le premier point de la proposition.

\textbf{2. }Puisque $\frac{\partial\widetilde{H}}{\partial x}=\frac{\partial
H}{\partial x}$, $e^{-\widetilde{H}}\frac{\partial}{\partial x}%
e^{\widetilde{H}}=e^{-H}\frac{\partial}{\partial x}e^{H}$ et posant
$L=e^{\widetilde{H}}/\left(  1\otimes\left(  Y/\omega\right)  ^{\delta
}B\right)  $
\begin{align*}
\frac{\partial s_{1}\left(  \mu,B\right)  }{\partial x}  &  =L\sum
\limits_{j=1}^{d}e^{-H}\frac{\partial}{\partial x}e^{H}\mathcal{E}%
^{j-1}\left(  \mu_{j}\otimes1\right)  =L\sum\limits_{0\leqslant m<j\leqslant
d}e^{-H}\frac{\partial}{\partial x}e^{H}\left(  \mu_{j}^{\left(  m\right)
}\otimes1\right)  \mathcal{E}_{j-1,m}\\
&  =L\sum\limits_{0\leqslant m<j\leqslant d}\left(  \left(  \mu_{j}^{\left(
m+1\right)  }\otimes1\right)  \mathcal{E}_{j-1,m}+\left(  \mu_{j}^{\left(
m\right)  }\otimes1\right)  \mathcal{DE}_{j-1,m}\right) \\
&  =L\sum\limits_{j=1}^{d}\sum\limits_{m=0}^{j}\mathcal{E}_{j,m}^{1}\left(
\mu_{j}^{\left(  m\right)  }\otimes1\right)
\end{align*}
avec $\mathcal{E}_{j,j}^{1}=\mathcal{E}_{j-1,j-1}=1\otimes\frac{\left(
Y-\omega\right)  ^{j-1}}{\left(  j-1\right)  !}$, $\mathcal{E}_{j,0}%
^{1}=\mathcal{DE}_{j-1,0}$ et pour $~1\leqslant m\leqslant j$, $\mathcal{E}%
_{j,m}^{1}=\mathcal{E}_{j-1,m-1}+\mathcal{DE}_{j-1,m}$. L'\'{e}quation
$\frac{\partial G_{1}}{\partial x}+\frac{\partial s_{1}\left(  \mu,B\right)
}{\partial x}=1\otimes\frac{B^{\prime}}{B}$ est alors \'{e}quivalente \`{a}%
\[
\sum\limits_{j=1}^{d}\sum\limits_{m=0}^{j}\mathcal{E}_{j,m}^{1}\left(  \mu
_{j}^{\left(  m\right)  }\otimes1\right)  =\left(  1\otimes\frac{B^{\prime}%
}{\left(  Y/\omega\right)  ^{\delta}}-1\otimes\frac{B}{\left(  Y/\omega
\right)  ^{\delta}}\frac{\partial G_{1}}{\partial x}\right)  e^{\widetilde{H}}%
\]
et donc au syst\`{e}me d'\'{e}quations annonc\'{e}. Puisque le
d\'{e}veloppement en s\'{e}rie de Laurent de $\frac{\partial G_{1}}{\partial
x}e^{\widetilde{H}}$ est de la forme $\sum\limits_{n\in\mathbb{Z}_{-}^{\ast}%
}\Lambda_{n}\left(  x\right)  y^{n}$ o\`{u} les $\Lambda_{n}$ sont des
fonctions holomorphe sur $D_{\omega}$, le coefficient en $y^{n}$ de celui de
$\frac{\partial G_{1}}{\partial x}e^{\widetilde{H}}\left(  1\otimes B\right)
$ est nul lorsque $n\geqslant\deg B=r$. On en d\'{e}duit que $K_{n}^{1}\left(
B,.\right)  =0$ lorsque $n\geqslant r+\delta=d$.

\textbf{3. }Pour simplifier les \'{e}critures, on \'{e}crit $f$ pour
$f\otimes1$. Puisque
\begin{align*}
\frac{\partial^{2}s_{1}\left(  \mu,B\right)  }{\partial x^{2}}  &
=L\sum\limits_{j=1}^{d}e^{-H}\frac{\partial^{2}}{\partial x^{2}}%
e^{H}\mathcal{E}^{j-1}\left(  \mu_{j}\otimes1\right) \\
&  =L\sum\limits_{j=1}^{d}\left(  \frac{\partial^{2}\mathcal{E}^{j-1}\mu_{j}%
}{\partial x^{2}}+2\frac{\partial\mathcal{E}^{j-1}\mu_{j}}{\partial x}%
\frac{\partial H}{\partial x}+\left(  \left(  \frac{\partial H}{\partial
x}\right)  ^{2}+\frac{\partial^{2}H}{\partial x^{2}}\right)  \mathcal{E}%
^{j-1}\mu_{j}\right)
\end{align*}
et donc%
\[
\frac{\partial^{2}s_{1}\left(  \mu,B\right)  }{\partial x^{2}}=L\sum
\limits_{j=1}^{d}\sum\limits_{\ell=0}^{j-1}\left(
\begin{array}
[c]{l}%
\mathcal{E}_{j-1,\ell}\mu_{j}^{\left(  \ell+2\right)  }+2\frac{\partial
\mathcal{E}_{j-1,\ell}}{\partial x}\mu_{j}^{\left(  \ell+1\right)  }%
+\frac{\partial^{2}\mathcal{E}_{j-1,\ell}}{\partial x^{2}}\mu_{j}^{\left(
\ell\right)  }\\
2\frac{\partial H}{\partial x}\left(  \mathcal{E}_{j-1,\ell}\mu_{j}^{\left(
\ell+1\right)  }+\frac{\partial\mathcal{E}_{j-1,\ell}}{\partial x}\mu
_{j}^{\left(  \ell\right)  }\right) \\
+\mathcal{E}_{j-1,\ell}\left(  \left(  \frac{\partial H}{\partial x}\right)
^{2}+\frac{\partial^{2}H}{\partial x^{2}}\right)  \mu_{j}^{\left(
\ell\right)  }%
\end{array}
\right)  =L\sum\limits_{j=1}^{d}\sum\limits_{m=0}^{j+1}\mathcal{E}_{j,m}%
^{2}\mu_{j}^{\left(  m\right)  }%
\]
o\`{u} les $\mathcal{E}_{j,m}^{2}$ sont d\'{e}finis comme annonc\'{e} pour
$\mathcal{E}_{j,j+1}^{2}$, $\mathcal{E}_{j,j}^{2}$, $\mathcal{E}_{j,0}^{2}$ et%
\begin{align*}
\mathcal{E}_{j,1}^{2}  &  =2\frac{\partial\mathcal{E}_{j-1,0}}{\partial
x}+\frac{\partial^{2}\mathcal{E}_{j-1,1}}{\partial x^{2}}+2\frac{\partial
H}{\partial x}\left(  \mathcal{E}_{j-1,0}+\frac{\partial\mathcal{E}_{j-1,1}%
}{\partial x}\right)  +\left(  \left(  \frac{\partial H}{\partial x}\right)
^{2}+\frac{\partial^{2}H}{\partial x^{2}}\right)  \mathcal{E}_{j-1,1}\\
&  =2\mathcal{DE}_{j-1,0}+\frac{\partial^{2}\mathcal{E}_{j-1,1}}{\partial
x^{2}}+2\frac{\partial H}{\partial x}\frac{\partial\mathcal{E}_{j-1,1}%
}{\partial x}+\left(  \left(  \frac{\partial H}{\partial x}\right)  ^{2}%
+\frac{\partial^{2}H}{\partial x^{2}}\right)  \mathcal{E}_{j-1,1}%
\end{align*}
ainsi que
\begin{align*}
\mathcal{E}_{j,m}^{2}  &  =\mathcal{E}_{j-1,m-2}+2\frac{\partial
\mathcal{E}_{j-1,m-1}}{\partial x}+\frac{\partial^{2}\mathcal{E}_{j-1,m}%
}{\partial x^{2}}+2\frac{\partial H}{\partial x}\left(  \mathcal{E}%
_{j-1,m-1}+\frac{\partial\mathcal{E}_{j-1,m}}{\partial x}\right)  +\left(
\left(  \frac{\partial H}{\partial x}\right)  ^{2}+\frac{\partial^{2}%
H}{\partial x^{2}}\right)  \mathcal{E}_{j-1,m.}\\
&  =\mathcal{E}_{j-1,m-2}+2\mathcal{DE}_{j-1,m-1}+\frac{\partial
^{2}\mathcal{E}_{j-1,m}}{\partial x^{2}}+2\frac{\partial H}{\partial x}%
\frac{\partial\mathcal{E}_{j-1,m}}{\partial x}+\left(  \left(  \frac{\partial
H}{\partial x}\right)  ^{2}+\frac{\partial^{2}H}{\partial x^{2}}\right)
\mathcal{E}_{j-1,m.}%
\end{align*}
Lorsque $\frac{\partial^{2}G}{\partial x^{2}}$ n'est pas la fonction nulle,
l'\'{e}quation $\frac{\partial^{2}G}{\partial x^{2}}+\frac{\partial^{2}%
s_{1}\left(  \mu,B\right)  }{\partial x^{2}}=0$ s'\'{e}crit donc%
\[
-1\otimes\left(  Y/\omega\right)  ^{\delta}B=\sum\limits_{j=1}^{d}%
\sum\limits_{m=0}^{j+1}\frac{\mathcal{E}_{j,m}^{2}}{\partial G/\partial x^{2}%
}\mu_{j}^{\left(  m\right)  }.
\]
Puisque $Y^{\delta}B=\sum\limits_{\delta\leqslant n\leqslant d}\beta_{n}y^{n}%
$, ceci est \'{e}quivalent \`{a} ce que pour tout $j\in\mathbb{Z}$ et $R>\rho
$,
\[
\beta_{n}\mathbf{1}_{\left\{  \delta,...,d\right\}  }\left(  n\right)
=\sum\limits_{j=1}^{d}\sum\limits_{m=0}^{j+1}\mu_{j}^{\left(  m\right)  }%
\frac{1}{2\pi i}\int_{R\mathbb{T}}\frac{\mathcal{E}_{j,m}^{2}}{\partial
G/\partial x^{2}}\frac{dy}{y^{n+1}}.
\]

\end{proof}

\subsection{Unicit\'{e} des d\'{e}compositions en ondes de choc}

On se donne $A,B\in\mathbb{C}\left[  Y\right]  $ avec $\deg A<r=\deg B$,
$B=\underset{1\leqslant q\leqslant r}{\Pi}\left(  1+s_{q}Y\right)
=\sum\limits_{0\leqslant j\leqslant r}\beta_{j}Y^{j}$ tel que $\left\{
B=0\right\}  \subset\rho\overline{\mathbb{D}}$ et on pose $P=1\otimes\frac
{A}{B}+X\otimes\frac{B^{\prime}}{B}$ ainsi que $d=r+\delta$. On note
$\mathcal{M}_{\omega}\left(  B\right)  $ l'ensemble des $\mu\in\mathcal{O}%
\left(  D_{\omega}\right)  ^{d}$ tel que le discriminant $\Delta\left(
\mu,B\right)  $ du polyn\^{o}me $S\left(  \mu,B\right)  =T^{d}+s_{1}\left(
\mu,B\right)  T^{d-1}+\cdots+s_{d}\left(  \mu,B\right)  \in\mathcal{O}\left(
D_{\omega}\right)  \left[  T\right]  $ n'est pas identiquement nul et tel que
$\mu$ est solution de (\ref{E0}).

Lorsque $\mu\in\mathcal{M}_{\omega}\left(  B\right)  $, on sait que pour tout
$\tau\in\mathbb{T}$ la fonction $-s_{1}\left(  \mu,B\right)  $ est une
$d$-onde de chocs sur $Z_{\omega,\tau}$. Pour tout $z_{\ast}$ appartenant au
compl\'{e}mentaire dans $Z_{\omega,\tau}$ d'un ensemble analytique de
dimension $1$, on sait gr\^{a}ce \`{a} la
proposition~\ref{C/ constructionOndeMult} que si $U_{\ast}$ un voisinage
suffisamment petit de $z_{\ast}$, il existe des ondes de chocs $g_{1}%
,...,g_{d}$ sur $U_{\ast}$ dont les images sont deux \`{a} deux distinctes
telles que pour tout $z\in U_{\ast}$,
\begin{gather*}
-s_{1}\left(  \mu,B\right)  \left(  z\right)  =N_{g,1}\left(  z\right) \\
N_{g,1}\left(  z\right)  +P\left(  z\right)  =G_{1}\left(  z\right)
=N_{h,1}\left(  z\right)  +P_{1}\left(  z\right)  =N_{Q,1}\left(  z\right)
+P_{1}%
\end{gather*}
o\`{u} les fonctions $h_{j}$ sont les ondes de chocs $h_{j}^{z_{\ast}}$
d\'{e}finies dans le corollaire~\ref{L/ serieLdeN}, c'est-\`{a}-dire les ondes
de chocs engendr\'{e}es par la collision de $Q$ avec les droites $L_{z} $,
$z\in U_{\ast}$.\medskip

A priori, rien ne garantit que $\left\{  g_{1},...,g_{d}\right\}  =\left\{
h_{1},...,h_{p}\right\}  $ car il se peut par exemple qu'il existe une partie
finie non vide $J$ de $\left\{  1,...,d\right\}  $ telle que $\sum
\limits_{j\in J}g_{j}$ se prolonge comme un \'{e}l\'{e}ment de l'espace
$\mathbb{C}\left(  Y\right)  _{1}\left[  X\right]  $ des fractions
rationnelles affines en $X$. Dans ce cas, $G_{1}=N_{\widetilde{g}%
,1}-\widetilde{P}$ avec $\widetilde{P}=P-\sum\limits_{j\in J}g_{j}%
\in\mathbb{C}\left(  Y\right)  _{1}\left[  X\right]  $ et $\left\{
\widetilde{g}_{1},...,\widetilde{g}_{\widetilde{d}}\right\}  $ o\`{u}
$\widetilde{d}=d-\operatorname{Card}\widetilde{J}\in\left\{  0,..,d-1\right\}
$. It\'{e}rant cette r\'{e}duction, on se ram\`{e}ne au cas o\`{u}
\begin{equation}
\forall J\in\mathcal{P}\left(  \left\{  1,..,d\right\}  \right)
\backslash\left\{  \varnothing\right\}  ,~\sum\limits_{j\in J}g_{j}%
\notin\mathbb{C}\left(  Y\right)  _{1}\left[  X\right]  . \label{F/ hyp}%
\end{equation}
Le cas $d=0$ ne se produit \`{a} l'issue de ces it\'{e}rations que si au
d\'{e}part, $\sum\limits_{1\leqslant j\leqslant d}g_{j}$ et donc $G_{1}$, se
prolonge comme \'{e}l\'{e}ment de $\mathbb{C}\left(  Y\right)  _{1}\left[
X\right]  $. Le lemme ci-dessous \'{e}tudie ce cas.

\begin{lemma}
\label{L/ G1affine}On utilise les notations du corollaire~\ref{L/ serieLdeN}.
$G_{1}$ se prolonge comme \'{e}l\'{e}ment de $\mathbb{C}\left(  Y\right)
_{1}\left[  X\right]  $ si et seulement si $Q$ est un domaine dans une courbe
compacte connexe $K$ telle que pour tout $z_{\ast}$ dans
$Z_{\operatorname{reg}}$ et $z$ dans un voisinage suffisamment petit $U_{\ast
}$ de $z_{\ast}$ dans $Z_{\operatorname{reg}}$,
\[
K\cap L_{z}=\left\{  \left(  1:h_{j}^{z_{\ast}}\left(  z\right)
:-x-yh_{j}^{z_{\ast}}\left(  z\right)  \right)  ;~1\leqslant j\leqslant
p\right\}  =Q\cap L_{z}.
\]

\end{lemma}

\begin{proof}
Supposons tout d'abord que $K$ est une courbe compacte ayant les
propri\'{e}t\'{e}s ci-dessus. Fixons $z_{\ast}$ et $U_{\ast}$ comme dans
l'\'{e}nonc\'{e}. Puisque $K$ est une courbe alg\'{e}brique, on sait depuis
les travaux d'Abel que $\sum\limits_{1\leqslant j\leqslant p}h_{j}^{z_{\ast}%
}\in\mathbb{C}\left(  Y\right)  _{1}\left[  X\right]  $ (voir par
exemple~\cite{HeG1995}). Il s'ensuit que $G_{1}=N_{h^{z_{\ast}},1}+P_{1}$ est,
sur $U_{\ast}$ et donc sur $Z$, rationnelle en $y$ et affine en $x$.

R\'{e}ciproquement, supposons que $G_{1}\in\mathbb{C}\left(  Y\right)
_{1}\left[  X\right]  $. Alors $N_{h^{z_{\ast}},1}=G_{1}-P_{1}$ est sur
$U_{\ast} $ rationnelle en $y$ et affine en $x$. Comme $\left\{  \left(
1:h_{j}^{z_{\ast}}\left(  z\right)  :-x-yh_{j}^{z_{\ast}}\left(  z\right)
\right)  ;~1\leqslant j\leqslant p\right\}  =Q\cap L_{z}$ pour tout $z\in
U_{\ast}$, un th\'{e}or\`{e}me de Wood~\cite{WoJ1984} stipule l'existence
d'une courbe alg\'{e}brique compacte $K$ de degr\'{e} $p$ contenant $Q$.
Puisque $K$ est de degr\'{e} $p$, $K\cap L_{z}=\left\{  \left(  1:h_{j}\left(
z\right)  :-x-yh_{j}\left(  z\right)  \right)  ;~1\leqslant j\leqslant
\lambda\right\}  =Q\cap L_{z}$ pour tout $z\in U$.\medskip
\end{proof}

Dans le cas o\`{u} $G_{1}$ est rationnelle en $y$ et affine en $x$, la courbe
alg\'{e}brique $K$ du lemme~\ref{L/ G1affine} est connue au voisinage de $bQ$.
On peut alors s\'{e}lectionner de fa\c{c}on g\'{e}n\'{e}rique des
coordonn\'{e}es homog\`{e}nes $w$ pour qu'au moins une droite $L_{z}$, $z\in
U$, rencontre $K\backslash Q$. On est ainsi ramen\'{e} au cas g\'{e}n\'{e}ral
puisque le lemme~~\ref{L/ G1affine} assure alors que m\^{e}me apr\`{e}s
r\'{e}duction, $d$ n'est pas nul.

Gr\^{a}ce \`{a} des r\'{e}sultats de Henkin~\cite{HeG1995} et de
Collion~\cite{CoS1996}, cette r\'{e}duction de la famille $\left(
g_{j}\right)  $ ram\`{e}ne \`{a} sa premi\`{e}re \'{e}quation le syst\`{e}me
sur-d\'{e}termin\'{e}~(\ref{F/ G=N+P}).

\begin{proposition}
\label{P/ UniciteP1}Les notations sont celles \'{e}nonc\'{e}es au d\'{e}but de
cette section. On suppose que (\ref{F/ hyp}) est v\'{e}rifi\'{e}e. Pour le cas
o\`{u} $Q$ est contenue dans une courbe alg\'{e}brique, $\widehat{Q}$
d\'{e}signant alors la plus petite courbe ayant cette propri\'{e}t\'{e}, on
suppose que $\left(  0:1:0\right)  \notin\widehat{Q}$ et qu'au moins une des
droites $L_{z}$, $z\in U$, rencontre $Q$ et $\widehat{Q}\backslash Q$. Ceci
\'{e}tant, $\left\{  g_{1},...,g_{d}\right\}  =\left\{  h_{1},...,h_{p}%
\right\}  $ et $P=P_{1}$.
\end{proposition}

\begin{proof}
Avec une \'{e}ventuelle renum\'{e}rotation, on se ram\`{e}ne au cas o\`{u}
$g_{\nu}=h_{\nu}$, $1\leqslant\nu\leqslant t\in\mathbb{N}$ et que $\left\{
g_{t+1},...,g_{d}\right\}  \cap\left\{  h_{t+1},...,h_{p}\right\}
=\varnothing$.

1) On suppose que $Q$ n'est pas contenue dans une courbe alg\'{e}brique. Dans
ce cas $d\in\mathbb{N}^{\ast}$ car sinon, $N_{h,1}\in\mathbb{C}\left(
Y\right)  _{1}\left[  X\right]  $ et $G_{1}$ se prolonge comme \'{e}l\'{e}ment
de $\mathbb{C}\left(  Y\right)  _{1}\left[  X\right]  $, ce qui est impossible
d'apr\`{e}s lemme~\ref{L/ G1affine}.

Supposons $t<\min\left(  p,d\right)  $. Quitte \`{a} changer le point de
r\'{e}f\'{e}rence $z_{\ast}$ et \`{a} diminuer $U_{\ast}$, on suppose que les
courbes $H_{\nu}=\left\{  \left(  1:h_{\nu}\left(  z\right)  :-x-yh_{\nu
}\left(  z\right)  \right)  ;~z\in U_{\ast}\right\}  $, $t+1\leqslant
\nu\leqslant p$ et $C_{\nu}=\left\{  \left(  1:g_{\nu}\left(  z\right)
:-x-yg_{\nu}\left(  z\right)  \right)  ;~z\in U_{\ast}\right\}  $,
$t+1\leqslant\nu\leqslant d$ sont lisses et deux \`{a} deux disjointes. On
note alors $\varphi$ la forme diff\'{e}rentielle d\'{e}finie sur la
r\'{e}union $C$ de ces courbes par $\varphi\left\vert _{H_{\nu}}\right.
=d\frac{w_{1}}{w_{0}}$ lorsque $t+1\leqslant\nu\leqslant p$ et $\varphi
\left\vert _{C_{\nu}}\right.  =-d\frac{w_{1}}{w_{0}}$ quand $t+1\leqslant
\nu\leqslant d$. On note $AR$ la transform\'{e}e d'Abel-Radon du courant
$\varphi\wedge\left[  C\right]  $. Par d\'{e}finition (voir \cite{HeG1995},
\cite{CoS1996} ou \cite{HeG2004}),
\[
AR=d(\sum\limits_{t+1\leqslant\nu\leqslant p}h_{\nu}-\sum\limits_{t+1\leqslant
\nu\leqslant q}g_{\nu}).
\]
Mais du fait des hypoth\`{e}ses,%
\[
\sum\limits_{t+1\leqslant\nu\leqslant p}h_{\nu}-\sum\limits_{t+1\leqslant
\nu\leqslant q}g_{\nu}=N_{h,1}-N_{g,1}=R-P_{1}.
\]
$AR$ est donc alg\'{e}brique au sens de~\cite{CoS1996} de sorte que le
th\'{e}or\`{e}me~1.2 de~\cite{CoS1996} s'applique et donne en particulier
l'existence d'une courbe alg\'{e}brique $\Lambda$ qui contient $C$. Puisque
$Q$ n'est pas contenue dans $\Lambda$, la connexit\'{e} de $Q$ impose
qu'aucune des courbes $H_{\nu}$ n'est contenue dans $\Lambda$ et donc que
$\left\{  h_{1},...,h_{p}\right\}  \subset\left\{  g_{1},...,g_{d}\right\}  $.
Ainsi, $\sum\limits_{p<\nu\leqslant d}g_{\nu}$ est une fonction alg\'{e}brique
affine en $x$, ce qui est impossible de par la r\'{e}duction faites sur
$\left(  g_{j}\right)  _{1\leqslant j\leqslant d}$. D'o\`{u} $t=\min\left(
p,d\right)  $.

Si $t=d<p$, la relation $N_{g,1}+P=N_{h,1}+P_{1}$ se lit aussi $h_{t+1}%
+\cdots+h_{p}=P_{1}-P\in\mathbb{C}\left(  Y\right)  _{1}\left[  X\right]  $ et
le th\'{e}or\`{e}me de Wood implique, puisque $Q$ est connexe, que $Q$ est
contenue dans une courbe alg\'{e}brique ce qui est exclu par hypoth\`{e}se. Si
$t=p<d$, $g_{t+1}+\cdots+g_{d}=N_{g,1}-N_{h,1}+P-P_{1}\in\mathbb{C}\left(
Y\right)  _{1}\left[  X\right]  $ ce qui est exclu du fait de la r\'{e}duction
faite sur la famille $\left(  g_{j}\right)  $.

Finalement $t=p=d$, $\left\{  h_{1},...,h_{p}\right\}  =\left\{
g_{1},...,g_{d}\right\}  $ et $P_{1}=R$.

2) Supposons maintenant que $Q$ est contenue dans une courbe alg\'{e}brique
$\widehat{Q}$, minimale au sens de l'inclusion. Puisque par hypoth\`{e}se
$\left(  0:1:0\right)  \notin\widehat{Q}$, on peut appliquer le
corollaire~\ref{C/ confinementZeroBinfini} \`{a} $\widehat{Q}\backslash Q$ qui
est bord\'{e} par $-\partial Q$ et obtenir que pour $z=\left(  x,y\right)  \in
Z$, $(\widehat{Q}\backslash Q)\cap L_{z}\subset\rho\overline{\mathbb{D}}$. Par
cons\'{e}quent, pour tout $z\in Z$,%
\[
\widehat{Q}\cap L_{z}\subset\left\{  w_{0}\neq0\right\}  .
\]
Quitte \`{a} changer de point de r\'{e}f\'{e}rence $z_{\ast}$ et \`{a}
diminuer $U_{\ast}$, on peut supposer que pour tout $z\in U_{\ast}$, $L_{z} $
coupe transversalement $\widehat{Q}$. On note alors $h_{p+1}%
,...,h_{\widehat{p}}$ les ondes de choc sur $U_{\ast}$ telles que pour tout
$z\in U$,
\[
(\widehat{Q}\backslash Q)\cap L_{z}=\left\{  \left(  1:h_{\nu}\left(
z\right)  :-x-yh_{\nu}\left(  z\right)  \right)  ;~p+1\leqslant\nu
\leqslant\widehat{p}\right\}  .
\]
Puisque $\widehat{Q}$ est une courbe alg\'{e}brique, $N_{\widehat{Q}%
,1}\overset{d\acute{e}f}{=}N_{h,1}+N_{h_{p+1},...,h_{\widehat{p}}%
}\overset{d\acute{e}f}{=}N_{h,1}+\widehat{N}_{1}$ est rationnelle en $y$ et
affine en $x $. D'o\`{u}%
\[
N_{g,1}+\widehat{N}_{1}=N_{g,1}-N_{h,1}+N_{\widehat{Q},1}=P_{1}%
-R+N_{\widehat{Q},1}\in\mathbb{C}\left(  Y\right)  _{1}\left[  X\right]
\]
La somme $N_{g,1}+\widehat{N}_{1}$ se r\'{e}crit $\sum\limits_{1\leqslant
\lambda\leqslant s}c_{\lambda}f_{\lambda}$ o\`{u} $f_{1},...f_{s}$ sont les
fonctions deux \`{a} deux distinctes constituant la r\'{e}union de $\left\{
g_{\nu};~1\leqslant\nu\leqslant q\right\}  $ et $\left\{  h_{\nu
};~p+1\leqslant\nu\leqslant\widehat{p}\right\}  $ et o\`{u} $c_{\lambda}=2 $
si $f_{\lambda}$ est dans l'intersection de ces deux ensembles et $1$ sinon.
Comme pr\'{e}c\'{e}demment, on peut choisir $z_{\ast}$ et $U_{\ast}$ pour que
les fonctions $f_{\lambda}$ aient des images deux \`{a} deux disjointes. On
peut alors introduire la forme $\psi$ qui sur $F_{\lambda}=\left\{  \left(
1:f_{\lambda}\left(  z\right)  :-x-yf_{\lambda}\left(  z\right)  \right)
;~z\in U\right\}  $ vaut $d\frac{w_{1}}{w_{0}}$ si $c_{\lambda}=1$ et
$2d\frac{w_{1}}{w_{0}}$ si $c_{\lambda}=2$. La forme $\sum\limits_{1\leqslant
\lambda\leqslant s}c_{\lambda}df_{\lambda}$ est la transform\'{e}e
d'Abel-Radon de $\psi\wedge\left[  F\right]  $ o\`{u} $F=\cup F_{\lambda}$.
Celle-ci \'{e}tant rationnelle, le th\'{e}or\`{e}me principal de Henkin
dans~\cite{HeG1995} s'applique et donne en particulier l'existence d'une
courbe alg\'{e}brique $\widetilde{F}$ et d'une forme rationnelle $\Psi$ telles
que pour tout $\lambda$, $\Psi\left\vert _{F_{\lambda}}\right.  =\psi$ et pour
tout $z\in U_{\ast}$, $\widetilde{F}\cap L_{z}=\cup L_{z}\cap F_{\lambda}$.
Etant donn\'{e} que $\widehat{Q}\cap\widetilde{F}$ contient $(\widehat{Q}%
\backslash Q)\underset{z\in U_{\ast}}{\cup}L_{z}$, $\widehat{Q}\subset
\widetilde{F}$. Si $\widetilde{F}\neq\widehat{Q}$, $\overline{\widehat{Q}%
\backslash\widetilde{F}}$ est une courbe alg\'{e}brique dont les intersections
avec les $L_{z}$, $z\in U_{\ast}$, se param\`{e}tre avec une sous-famille des
$g_{j}$. Ceci est impossible car du fait des hypoth\`{e}ses, $d\neq0$ et
aucune sous-famille de $\left(  g_{j}\right)  $ n'a de somme qui soit
rationnelle en $y$ et affine en $x$. Ainsi, $\widehat{Q}=\widetilde{F}$ et
lorsque $z\in U_{\ast}$, $\widehat{Q}\cap L_{z}$ est la r\'{e}union de
$(\widehat{Q}\backslash Q)\cap L_{z}$ et de $\left\{  \left(  1:g_{j}\left(
z\right)  :-x-yg_{\lambda}\left(  z\right)  \right)  ;~1\leqslant j\leqslant
d\right\}  $. Ceci force $\left\{  h_{1},...,h_{p}\right\}  =\left\{
g_{1},...,g_{d}\right\}  $ et $P_{1}=R$.
\end{proof}

\subsection{Proc\'{e}d\'{e} de reconstruction\label{S/ algorithme}}

Pour d\'{e}crire le proc\'{e}d\'{e} algorithmique de reconstruction de $Q$, in
reprend int\'{e}gralement les notations de la section pr\'{e}c\'{e}dente.

\begin{enumerate}
\item Si $G_{1}$ est rationnelle en $y$ et affine en $x$, $Q$ est contenue
dans une courbe alg\'{e}brique connexe $K$ et on choisit des coordonn\'{e}es
pour qu'au moins une des droites $L_{z}$, $z\in Z$, rencontre $Q$ et
$K\backslash Q$.

\item On se donne $A,B\in\mathbb{C}\left[  Y\right]  $ avec $\deg A<r=\deg B$,
$B=\underset{1\leqslant q\leqslant r}{\Pi}\left(  1+s_{q}Y\right)
=\sum\limits_{0\leqslant j\leqslant r}\beta_{j}Y^{j}$ tel que $\left\{
B=0\right\}  \subset\rho\overline{\mathbb{D}}$, $R=1\otimes\frac{A}%
{B}+X\otimes\frac{B^{\prime}}{B}$ et on pose $d=r+\delta$. Etant donn\'{e} que
(\ref{E0}) est un syst\`{e}me diff\'{e}rentiel lin\'{e}aire
surd\'{e}termin\'{e}, l'existence d'une solution \`{a} (\ref{E0}) est
conditionn\'{e}e par un jeu, forc\'{e}ment fini, d'\'{e}quations lin\'{e}aires
portant sur $\left(  A,B\right)  $. Lorsque $\left(  A,B\right)  $ v\'{e}rifie
ce syst\`{e}me lin\'{e}aire, (\ref{E0}) a au moins une solution $\mu$ dans
$\mathcal{M}_{\omega,\tau}\left(  B\right)  $.

\item Localement, dans $Z_{\omega,\tau}\cap U_{\operatorname{reg}}$, il existe
$d$-onde de chocs $g_{1},..,g_{d}$ telles que $-s_{1}\left(  \mu,B\right)
=N_{h,1}$. Appliquant \`{a} cette famille $\left(  g_{j}\right)  $ la
r\'{e}duction d\'{e}crite au d\'{e}but de cette section puis la
proposition~\ref{P/ UniciteP1}, on en d\'{e}duit que $d\geqslant p$,
$r=d-\delta\geqslant p-\delta=q_{\infty}$ et que si $\left(  \widetilde{g}%
_{j}\right)  _{1\leqslant j\leqslant p}$ est le jeu de fonctions obtenu \`{a}
partir de $\left(  g_{j}\right)  $ par r\'{e}duction, $\left\{  \widetilde{g}%
_{1},...,\widetilde{g}_{p}\right\}  =\left\{  h_{1},...,h_{p}\right\}  $ et
$P_{1}=R$. En cons\'{e}quence, on conna\^{\i}t aussi $\left(  P_{k}\right)
_{k\in\mathbb{N}^{\ast}}$ comme le prolongement en fraction rationnelle de
$\left(  G_{k}\left\vert _{U_{\ast}}\right.  -N_{h,k}\right)  _{k\in
\mathbb{N}^{\ast}}$.

\item On sait qu'il existe une fonction localement constante $\pi
:U_{\operatorname{reg}}\rightarrow\mathbb{N}$ telle que lorsque $z_{\ast}\in
U_{\operatorname{reg}}$, il existe un voisinage $U_{z_{\ast}}$ de $z_{\ast}$
dans $Z_{\operatorname{reg}}$ et des ondes de chocs $h_{1}^{z_{\ast}%
},...,h_{\pi\left(  z_{\ast}\right)  }^{z_{\ast}}$ deux \`{a} deux distinctes
telles que $Q$ contient $Q_{z_{\ast}}=\underset{1\leqslant k\leqslant
\pi\left(  z_{\ast}\right)  }{\cup}\left\{  \left(  1:h_{j}^{z_{\ast}}\left(
z\right)  :-x-yh_{j}^{z_{\ast}}\left(  z\right)  \right)  ;~z\in U_{z_{\ast}%
}\right\}  $ et $\left(  G_{k}\left\vert _{U_{z_{\ast}}}\right.  \right)
_{k\in\mathbb{N}^{\ast}}=\left(  N_{h^{z_{\ast}},k}+P_{k}\left\vert
_{U_{z_{\ast}}}\right.  \right)  _{k\in\mathbb{N}^{\ast}}$. Gr\^{a}ce aux
formules de Newton~(\ref{F/ NS}), on conna\^{\i}t donc $\left(  S_{h^{z_{\ast
}},k}\right)  _{k\in\mathbb{N}^{\ast}}$. En outre, $\pi\left(  z_{\ast
}\right)  =G_{0}\left\vert _{U_{z_{\ast}}}\right.  -q_{\ast}$ est connu. On
peut donc calculer individuellement les fonctions $h_{j}^{z_{\ast}}$,
$1\leqslant j\leqslant\pi\left(  z_{\ast}\right)  $ \`{a} partir $\left(
S_{h^{z_{\ast}},k}\right)  _{1\leqslant k\leqslant\pi\left(  z_{\ast}\right)
}$.

\item Gr\^{a}ce au lemme~\ref{L/ Qexhaustion}, $Q\cap\left\{  w_{0}%
\neq0\right\}  $ et donc $Q$ sont connues.\smallskip
\end{enumerate}

Pour reconstruire une surface de Riemann \`{a} bord $M$ \'{e}quip\'{e}e d'une
conductivit\'{e} $\sigma$ \`{a} partir de $\left(  \partial M,N_{d}^{\sigma
}\right)  $, on peut proc\'{e}der comme suit~:

\begin{enumerate}
\item Le th\'{e}or\`{e}me de Henkin-Santacesaria produit une surface de
Riemann nodale \`{a} bord $\mathcal{M}$ dont $M$ est une normalisation.
$\mathcal{M}$ est une courbe complexe de $\mathbb{C}^{2}$ qu'on peut
d\'{e}terminer explicitement \`{a} partir de formules de type de Cauchy.

\item Avec la section~\ref{S/ FGN}, on produit \`{a} partir de $\mathcal{M}$
une fonction de Green pour la surface de Riemann $M$.

\item Avec le th\'{e}or\`{e}me~\ref{T/ plgt1} et l'\'{e}tape
pr\'{e}c\'{e}dente, on obtient une surface de Riemann \`{a} bord $S$
plong\'{e}e dans $\mathbb{CP}_{3}$, isomorphe \`{a} $M$ et dont les
projections sur $\left\{  w_{2}\neq0\right\}  $ et $\left\{  w_{3}%
\neq0\right\}  $ sont des surfaces de Riemann nodales $Q_{2}$ et $Q_{3}$
v\'{e}rifiant les hypoth\`{e}ses g\'{e}n\'{e}riques \'{e}nonc\'{e}es au
d\'{e}but de la section~\ref{S/ RcoeffCst}.

\item On applique \`{a} $Q_{2}$ et $Q_{3}$ le proc\'{e}d\'{e}
pr\'{e}c\'{e}dent pour construire un atlas de $S$.

\item $S$ \'{e}tant connue, le th\'{e}or\`{e}me de Henkin-Novikov~\ref{T/ HN}
permet de reconstruire le coefficient de conductivit\'{e} $s$ de la
conductivit\'{e} qui est l'image directe de $\sigma$ par l'isomorphisme du
point 3. D\'{e}signant alors par $c_{S}$ la conductivit\'{e} de coefficient
$1$ associ\'{e}e \`{a} la structure complexe de $S$ induite par celle de
$\mathbb{CP}_{3}$, $\left(  S,sc_{S}\right)  $ est une solution explicite du
probl\`{e}me pos\'{e}.
\end{enumerate}

\subsection{Genre d'une surface de Riemann \`{a} bord \label{S/ Genre}}

Conna\^{\i}tre a priori le nombre $p$ des fonctions inconnues $h_{j}$
intervenant dans la proposition~\ref{P/ DH1997} permettrait d'am\'{e}liorer
grandement l'efficacit\'{e} de l'algorithme propos\'{e} dans la section
pr\'{e}c\'{e}dente. Dans~\cite{BeM2003}, Belishev donne une formule pour
calculer le genre \`{a} partir de l'op\'{e}rateur de Dirichlet-Neumann%
\[
2g\left(  M\right)  =\operatorname*{rg}\left(  T+\left(  N^{\nu}J\right)
^{2}T\right)
\]
o\`{u} $T$ est la d\'{e}rivation tangentielle, $N^{\nu}$ est l'op\'{e}rateur
Dirichlet-Neumann dans sa pr\'{e}sentation m\'{e}trique, c'est \`{a} dire
celui qui \`{a} $u\in C^{\infty}\left(  bM\right)  $ associe la
d\'{e}riv\'{e}e normale le long de $bM$ de l'extension harmonique de $u$ \`{a}
$M$ et $J$ est l'op\'{e}rateur de primivitisation d\'{e}fini de la fa\c{c}on
suivante~: on note $\Gamma$ l'ensemble des composantes connexes de $bM$, pour
chaque $\gamma\in\Gamma$ on fixe un point $p_{\gamma}$ de $\gamma$ et lorsque
$q\in\gamma$, on note $\gamma_{q}$ l'arc de $\gamma$ dont le bord orient\'{e}
est $q-p_{\gamma}$. $J$ est\ alors l'op\'{e}rateur qui \`{a} $u\in
C^{0}\left(  bM\right)  $ telle que $\int_{\gamma}u\tau^{\ast}=0 $ pour tout
$\gamma\in\Gamma$ associe la fonction d\'{e}finie par $\left(  Ju\right)
\left\vert _{\gamma}\left(  p\right)  \right.  =\int_{\gamma_{p}}u\tau^{\ast}%
$, $\gamma\in\Gamma$. Cependant le calcul a priori du rang de $T+\left(
N^{\nu}J\right)  ^{2}T$ n'est pas forc\'{e}ment facile. Pour contourner cette
difficult\'{e}, \cite{BeM-ShV2008} et \cite{ShV-ShC2013} proposent d'utiliser
des op\'{e}rateurs de Dirichlet-Neumann agissant sur les formes. Ceci donne
des formules simples pour $g\left(  M\right)  $ dans le cas o\`{u} la
conductivit\'{e} se r\'{e}duit \`{a} une structure complexe mais il n'est pas
clair que ces op\'{e}rateurs aient un sens physique.

La formule (\ref{F/ qinfini-g}), prouv\'{e}e par le
corollaire~\ref{C/ genre via dp(u)}, relie le nombre $p$ ci-dessus au nombre
$q_{\infty}$ de points d'intersection de la surface de Riemann nodale \`{a}
bord $Q$ \`{a} reconstruire avec $\left\{  w\in\mathbb{CP}_{2};~w_{0}%
=0\right\}  $.

\begin{lemma}
\label{L/ GenreDuDouble}Soit $M$ une surface de Riemann \`{a} bord. On note
$c$ le nombre de composantes connexes du bord de $M$ et $\widehat{M}$ le
double de $M$. Le genre $g\left(  \widehat{M}\right)  $ de $\widehat{M}$ et
celui $g\left(  M\right)  $ de $M$, qui par d\'{e}finition est le genre de la
vari\'{e}t\'{e} compacte obtenue en recollant $c$ disques conformes (deux
\`{a} deux disjoints) le long des $c$ composantes connexes de $M$, sont
li\'{e}s par la relation%
\begin{equation}
g\left(  \widehat{M}\right)  =2g\left(  M\right)  +c-1. \label{F/ gdouble}%
\end{equation}

\end{lemma}

\begin{proof}
On se donne une triangulation $T$ de $M$. Lorsque $\alpha$ est une composante
connexe de $\gamma=bM$, on note $\Sigma_{\gamma}$ l'ensemble des sommets
d'\'{e}l\'{e}ments $T$ qui sont sur $\gamma$ et $A_{\gamma}$ celui des
arr\^{e}tes d'\'{e}l\'{e}ments $T$ qui sont incluses dans $\gamma$. On pose
$\Sigma^{b}=\underset{\gamma\in\mathcal{C}}{\cup}M_{\gamma}$ et $A^{b}%
=\underset{\gamma\in\mathcal{C}}{\cup}T_{\gamma}$ o\`{u} $\mathcal{C}$ est
l'ensemble des composantes connexes de $bM$. Pour chaque $\gamma\in
\mathcal{C}$, $\left\vert \Sigma_{\gamma}\right\vert =\left\vert A_{\gamma
}\right\vert $ et supposant, quitte \`{a} changer de triangulation, que les
ensembles $\underset{t\in T,\ T\cap M_{\gamma}\neq\varnothing\ }{\cup}$ sont
deux \`{a} deux disjoints quand $\gamma$ d\'{e}crit $\mathcal{C}$, on obtient
$\left\vert \Sigma^{b}\right\vert =\left\vert A^{b}\right\vert $. Enfin, on
note $\sigma\left(  T\right)  $ le nombre de sommets de $T$, $a\left(
T\right)  $ le nombre d'arr\^{e}tes de $T$, $f\left(  T\right)  $ le nombre de
faces de $T$ et on pose $\widetilde{M}=\widehat{M}\backslash\overline{M}$. On
note $\widetilde{T}$ la triangulation de $\widetilde{M}$ obtenue comme
sym\'{e}trique de $T,$ c'est-\`{a}-dire celle obtenue en faisant agir sur $T$
l'involution naturelle de $\widehat{M}$. $\widehat{T}=T\cup\widetilde{T}$ est
alors une triangulation de $\widehat{M}$. Par d\'{e}finition de la
caract\'{e}ristique d'Euler d'une surface, il vient alors
\begin{align*}
\chi\left(  \widehat{M}\right)   &  =\sigma\left(  \widehat{T}\right)
-a\left(  \widehat{T}\right)  +f\left(  \widehat{T}\right) \\
&  =\left[  2\left(  \sigma\left(  T\right)  -\Sigma^{b}\right)  +\Sigma
^{b}\right]  -\left[  2\left(  a\left(  T\right)  -A^{b}\right)
+A^{b}\right]  +2f\left(  T\right) \\
&  =\left[  2\sigma\left(  T\right)  -\Sigma^{b}\right]  -\left[  2a\left(
T\right)  -A^{b}\right]  +2f\left(  T\right) \\
&  =2\sigma\left(  T\right)  -2a\left(  T\right)  +2f\left(  T\right)
=2\chi\left(  M\right)  .
\end{align*}
On sait gr\^{a}ce \`{a} la th\'{e}orie usuelles sur les surfaces de Riemann
compactes que $\chi\left(  \widehat{M}\right)  =2-2g\left(  \widehat{M}%
\right)  $. D'o\`{u}
\[
g\left(  \widehat{M}\right)  =1-\chi\left(  M\right)
\]
Notons $M^{\prime}$ la surface obtenu en rattachant $c$ disques le long des
composantes connexes de $\gamma$. Alors $\chi\left(  M^{\prime}\right)
=\chi\left(  M\right)  +c$ et par d\'{e}finition, $g\left(  M\right)
=g\left(  M^{\prime}\right)  $. On obtient donc%
\[
g\left(  \widehat{M}\right)  =1-\chi\left(  M\right)  =1+c-\chi\left(
M^{\prime}\right)  =1+c-\left[  2-2g\left(  M^{\prime}\right)  \right]
\]
D'o\`{u} $g\left(  \widehat{M}\right)  =2g\left(  M\right)  +c-1$.
\end{proof}

\begin{definition}
Soit $M$ une surface de Riemann \`{a} bord. On se donne une fonction
d\'{e}finissante $\rho$ de $bM$, ce qui signifie $\rho\in C^{\infty}\left(
M,\mathbb{R}\right)  $ telle que $\rho\left\vert _{B\cap M}\right.  <0$,
$\rho\left\vert _{B\cap bM}\right.  =0$ et $\left(  d\rho\left\vert
_{B}\right.  \right)  _{s}\neq0$ pour tout $s\in bM$. Dans ces conditions,
toute section $\omega$ de $\Lambda^{p,q}T^{\ast}\overline{M}$ de classe
$C^{k}$ sur un ouvert $U$ de $\overline{M}$ peut s'\'{e}crire sous la forme
$\omega_{0}+\rho\omega_{1}$ o\`{u} $\omega_{j}$, $j=0,1$, est une section de
$\Lambda^{p,q}T^{\ast}\overline{M}$ sur $U$ de classe $C^{k-j}$, le couple
$\left(  \omega_{\rho}^{\left(  0\right)  },\omega_{\rho}^{\left(  1\right)
}\right)  =\left(  \omega_{0}\left\vert _{U\cap M}\right.  ,\omega
_{1}\left\vert _{U\cap M}\right.  \right)  $ \'{e}tant le m\^{e}me pour tout
$\left(  \omega_{0},\omega_{1}\right)  $ tel que $\omega=\omega_{0}+\rho
\omega_{1}$. Le fait que $\omega_{\rho}^{\left(  1\right)  }$ soit nulle ne
d\'{e}pend pas du choix de la fonction d\'{e}finissante $\rho$ choisie. On dit
que $\omega$ est tangente \`{a} $bM$ lorsque $\omega_{\rho}^{\left(  1\right)
}=0$.
\end{definition}

Notons que l'existence d'une d\'{e}composition $\omega=\omega_{0}+\rho
\omega_{1}$ d\'{e}coule de ce que $\rho$ peut \^{e}tre prise comme l'une des
coordonn\'{e}es d'un syst\`{e}me de coordonn\'{e}es r\'{e}elles pour (chaque
branche) de $\overline{M}$ au voisinage de $bM$. L'unicit\'{e} de $\left(
\omega_{\rho}^{\left(  0\right)  },\omega_{\rho}^{\left(  1\right)  }\right)
$ provient de la m\^{e}me raison et si $\rho^{\prime}$ est une autre fonction
d\'{e}finissante de $bM$, on peut \'{e}crire $\rho^{\prime}=\lambda\rho$
o\`{u} $\lambda$ est une fonction de classe $C^{k-1}$ qui ne s'annule pas, de
sorte que l'annulation de $\omega_{\rho^{\prime}}^{\left(  1\right)  }%
=\lambda\left\vert _{M}\right.  \omega_{\rho}^{\left(  1\right)  }$ et
$\omega_{\rho}^{\left(  1\right)  }$ est simultan\'{e}e.

Notons que lorsque $M$ est munie d'une m\'{e}trique hermitienne et que $\rho$
est la distance \`{a} $bM$, $\frac{\partial\omega}{\partial\rho}$ n'est rien
d'autre que la d\'{e}riv\'{e}e de $\omega$ par rapport au vecteur unitaire qui
oriente la normale ext\'{e}rieure \`{a} $\overline{M}$ aux points de $bM$. Le
lemme ci-dessous assure qu'il existe des formes volumes satisfaisant
l'hypoth\`{e}se du th\'{e}or\`{e}me principal de cette section.

\begin{lemma}
Une surface de Riemann \`{a} bord admet des formes volumes de classe $C^{2}$
tangentes \`{a} son bord.
\end{lemma}

\begin{proof}
Soit $V$ une forme volume arbitraire de classe $C^{2}$ sur le double
$\widehat{M}$ de $M$. On se donne $\rho\in C^{\infty}\left(  \widehat{M}%
,\mathbb{R}\right)  $ telle que $M=\left\{  \rho<0\right\}  ,bM=\left\{
\rho=0\right\}  $ et $\left(  d\rho\right)  _{s}\neq0$ pour tout $s\in bM$. La
restriction $V_{0}$ de $V$ \`{a} $bM$ est donc une section de $\Lambda
^{1,1}T\widehat{M}$ de classe $C^{2}$ sur $bM$. En utilisant le
th\'{e}or\`{e}me d'extension de Whitney (voir \cite[prop. 2.2]{AnA-HiD1972}),
on peut construire une section $\widetilde{V}$ de $\Lambda^{1,1}T\widehat{M}$
de classe $C^{2}$ telle que $\widetilde{V}\left\vert _{M}\right.  =V_{0}$ et
$V_{\rho}^{\left(  1\right)  }=\frac{\partial\widetilde{V}}{\partial\rho}=0$.
Par continuit\'{e}, il existe un voisinage $\Sigma$ de $bM$ dans $\widehat{M}$
tel que $\widetilde{V}\left\vert _{\Sigma}\right.  $ est une forme volume.
Soit $\chi\in C^{\infty}\left(  M,\left[  0,1\right]  \right)  $ valant $1$ au
voisinage de $bM$ dans $\Sigma$ et dont le support est contenu dans $\Sigma$.
$W=\chi\widetilde{V}+\left(  1-\chi\right)  V$ est une forme volume $W$ de
classe $C^{2}$ sur $\widehat{M}$ telle que $\frac{\partial W}{\partial\rho}=0$.
\end{proof}

\begin{theorem}
\label{T/ ZPgenre}Soit $M$ une surface de Riemann \`{a} bord dont le bord est
constitu\'{e} de $c$ composantes connexes. On se donne une forme volume $\mu$
de classe $C^{2}$ sur $\overline{M}$ tangente \`{a} $bM$ puis on munit le
fibr\'{e} des $\left(  1,0\right)  $-formes sur $\overline{M}$ de la
m\'{e}trique $h^{\ast}$ d\'{e}finie pour tout $s\in\overline{M}$ et $\alpha
\in\Lambda^{1,0}T^{\ast}M$ par%
\begin{equation}
h_{s}^{\ast}\left(  \alpha\right)  =\left(  \frac{\alpha\wedge\ast
\overline{\alpha}}{\mu_{s}}\right)  ^{1/2} \label{F/ h*}%
\end{equation}
o\`{u} $\ast$ est l'op\'{e}rateur de conjugaison associ\'{e} \`{a} la
structure complexe de $M$. On note $D$ la connexion de Chern associ\'{e}e
\`{a} $h^{\ast}$. Alors, quand $\omega$ est une $\left(  1,0\right)  $-forme
m\'{e}romorphe sur $\overline{M}$, sans p\^{o}le ni z\'{e}ro sur $bM$,
\begin{equation}
\frac{1}{2\pi i}\int_{\partial M}\frac{D\omega}{\omega}=N_{z}\left(
\omega\right)  -N_{p}\left(  \omega\right)  +2-2g\left(  M\right)  -c
\label{F/ genreS}%
\end{equation}
o\`{u} $N_{z}\left(  \omega\right)  $ et $N_{p}\left(  \omega\right)  $ sont
respectivement le nombre de z\'{e}ros et de p\^{o}les de $\omega$ compt\'{e}s
avec leur multiplicit\'{e}.
\end{theorem}

\noindent\textbf{Remarque. }Supposons que $\mu^{\prime}$ est une forme volume
pour $\overline{M}$ ayant les m\^{e}mes propri\'{e}t\'{e}s que $\mu$. La
fonction $\kappa:\overline{M}\rightarrow\mathbb{R}$ telle que $\mu=e^{2\kappa
}\mu^{\prime}$ v\'{e}rifie $D_{\mu}=D_{\mu^{\prime}}-\partial\kappa$, ce qui
donne $\int_{\partial M}\frac{D_{\mu}\omega}{\omega}=\int_{\partial M}%
\frac{D_{\mu^{\prime}}\omega}{\omega}-\int_{\partial M}j_{bM}^{\ast}%
\partial\kappa$. (\ref{F/ genreS}) indique alors que $\int_{\partial M}%
j_{bM}^{\ast}\partial\kappa=0$. Pour v\'{e}rifier cela a priori,
consid\'{e}rons une fonction d\'{e}finissante $\rho$ de $bM$. De la relation
$\frac{\partial\mu}{\partial\rho}=e^{\kappa\left\vert _{M}\right.  }%
\frac{\partial\mu^{\prime}}{\partial\rho}+\mu^{\prime}\left\vert _{M}\right.
\frac{\partial\kappa}{\partial\rho}$, on tire $\frac{\partial\kappa}%
{\partial\rho}=0$. Munissons $M$ d'une m\'{e}trique hermitienne et
consid\'{e}rons une section lisse $\left(  \nu,\tau\right)  $ de $\left(
T_{bM}\overline{M}\right)  ^{2}$ telle que pour chaque $s\in bM$, $\left(
\nu_{s},\tau_{s}\right)  $ une base orthonorm\'{e}e directe de $T_{s}%
\overline{M}$. Alors pour tout $s\in bM$, $\left(  \partial\kappa\right)
_{s}=\frac{1}{2}\left(  \left(  \nu\kappa\right)  _{s}-i\left(  \tau
\kappa\right)  _{s}\right)  \left(  \tau_{s}^{\ast}+i\nu_{s}^{\ast}\right)  $
o\`{u} $\left(  \tau_{s}^{\ast},\nu_{s}^{\ast}\right)  $ est la base duale de
$\left(  \nu_{s},\tau_{s}\right)  $. Lorsque $s\in bM$, le fait que
$\frac{\partial\kappa}{\partial\rho}\left(  s\right)  =0$ indique que $\left(
d\kappa\right)  _{s}\in\mathbb{R}\tau_{s}^{\ast}$ et donc que $\left(
\nu\kappa\right)  _{s}=0$, ce qui donne $\left(  \partial\kappa\right)
_{s}=\frac{1}{2i}\left(  \tau\kappa\right)  _{s}\left(  \tau_{s}^{\ast}%
-i\nu_{s}^{\ast}\right)  $. D'o\`{u} $j_{bM}^{\ast}\partial\kappa=\frac{1}%
{2i}\left(  \tau\kappa\right)  \tau^{\ast}\left\vert _{M}\right.  =\frac
{1}{2i}j_{bM}^{\ast}d\kappa$. Ainsi, $j_{bM}^{\ast}\partial\kappa$ est exacte
et son int\'{e}grale sur $\partial M$ est nulle.

\begin{proof}
[Preuve du th\'{e}or\`{e}me~\ref{T/ ZPgenre}]Commen\c{c}ons par d\'{e}tailler
une construction du double $\widehat{M}$ de $M$ qu'on trouve par exemple
dans~\cite{AhL-SaL1960Livre}. Soit $\mathcal{U}$ un atlas de $M$. On utilise
les notations suivantes~: pour $\nu\in\left\{  -1,+1\right\}  $ et
$X\subset\overline{M}$, $X_{\nu}=X\times\left\{  \nu\right\}  $ et si $\left(
s,\nu\right)  \in M_{1}\cup M_{-1}$, $\pi\left(  s,\nu\right)  =s$~; quand
$s\in bM$; les points de $\widehat{M}=\overline{M_{1}}\cup\overline{M_{-1}}$
de la forme $\left(  s,-1\right)  $ et $\left(  s,1\right)  $ sont
identifi\'{e}s et forment la courbe r\'{e}elle $\gamma$. On munit $M_{1}$ de
la structure complexe associ\'{e}e \`{a} l'atlas $\mathcal{U}_{1}$
constitu\'{e} par les applications $\varphi_{1}:U_{1}\ni p\mapsto
\varphi\left(  \pi\left(  p\right)  \right)  $ o\`{u} $\varphi:U\rightarrow
\mathbb{C}$ d\'{e}crit $\mathcal{U}$. Pour $M_{-1}$, on utilise l'atlas
$\mathcal{U}_{-1}$ des applications $\varphi_{-1}:U_{-1}\ni p\mapsto
-\overline{\varphi\left(  \pi\left(  p\right)  \right)  }$, $\varphi
:U\rightarrow\mathbb{C}$ d\'{e}crivant $\mathcal{U}$. On obtient un atlas
$\widehat{\mathcal{U}}=\mathcal{U}_{1}\cup\mathcal{U}_{b}\cup\mathcal{U}_{-1}$
conf\'{e}rant \`{a} $\widehat{M}$ une structure complexe en d\'{e}finissant
$\mathcal{U}_{b}$ comme l'ensemble des applications $\varphi_{b}$ d\'{e}finies
de la mani\`{e}re suivante~: on se donne une carte de bord pour $\overline{M}$
c'est \`{a} dire $\varphi\in C^{\infty}\left(  U,\mathbb{C}\right)  $ o\`{u}
$U$ est un ouvert de $\overline{M}$ tel que $b_{U}M=\overline{U}\cap bM$ est
un ouvert de $bM$, $\varphi\left(  U\backslash M\right)  =\mathbb{D}%
^{+}=\mathbb{D}\cap\left\{  \operatorname{Im}>0\right\}  $ et $\varphi\left(
b_{U}M\right)  =\left]  -1,1\right[  $~; $\varphi_{b}$ est l'application de
$U_{b}=U_{1}\cup U_{-1}$ dans $\mathbb{C}$ obtenue en posant $\varphi
_{b}\left(  s,1\right)  =\varphi\left(  s\right)  $ et $\varphi_{b}\left(
s,-1\right)  =\overline{\varphi\left(  s\right)  }$ pour tout $s\in U$.

On d\'{e}finit des formes volumes $\mu_{1}$ et $\mu_{-1}$ sur $\overline
{M}_{1}$ et $\overline{M}_{-1}$ en posant lorsque $\varphi:U\rightarrow
\mathbb{C}$ est une carte de $\overline{M}$,
\begin{gather*}
\left(  \varphi_{1\ast}\mu_{1}\right)  _{z}=\left(  \varphi_{\ast}\mu\right)
_{z}=\lambda_{\varphi}\left(  z\right)  idz\wedge d\overline{z},~z\in U\\
\left(  \varphi_{-1\ast}\mu_{-1}\right)  _{w}=\left(  \varphi_{\ast}%
\mu\right)  _{-\overline{w}}=\lambda_{\varphi}\left(  -\overline{w}\right)
idw\wedge d\overline{w},~-\overline{w}\in U
\end{gather*}
Cette d\'{e}finition est \'{e}videmment coh\'{e}rente pour $\mu_{1}$.
Supposons que $\psi:V\rightarrow\mathbb{C}$ est une autre carte de $M$ et
$\psi_{\ast}\mu=\lambda_{\psi}idz\wedge d\overline{z}$. On note $\Phi
:\psi\left(  U\cap V\right)  \ni z\mapsto\varphi\left(  \psi^{-1}\left(
z\right)  \right)  $ le changement de carte de $\psi$ \`{a} $\varphi$. On a
donc $\lambda_{\psi}=\left\vert \Phi^{\prime}\right\vert ^{2}\lambda_{\varphi
}\circ\Phi$. Le changement de carte de $\psi_{-1}:V_{-1}\rightarrow\mathbb{C}$
\`{a} $\varphi_{-1}:U_{-1}\rightarrow\mathbb{C}$ est alors l'application
$\Phi_{-1}$ d\'{e}finie sur $\psi_{-1}\left(  V_{-1}\cap U_{-1}\right)
=-\overline{\psi}\left(  U\cap V\right)  $ par
\[
\Phi_{-1}\left(  w\right)  =\varphi_{-1}\left(  \left(  \psi_{-1}\right)
^{-1}w\right)  =\varphi_{-1}\left(  \psi^{-1}\left(  -\overline{w}\right)
,-1\right)  =-\overline{\varphi}\left(  \psi^{-1}\left(  -\overline{w}\right)
\right)  =-\overline{\Phi\left(  -\overline{w}\right)  }.
\]
D'o\`{u}%
\begin{align*}
\Phi_{-1}^{\ast}\left(  \lambda_{\varphi}\left(  -\overline{z}\right)
idz\wedge d\overline{z}\right)   &  =\lambda_{\varphi}\left(  \Phi\left(
-\overline{w}\right)  \right)  i\left(  -\frac{\partial\overline{\Phi\left(
-\overline{w}\right)  }}{\partial w}dw\right)  \wedge\left(  \left(
-\frac{\partial\Phi\left(  -\overline{w}\right)  }{\partial\overline{w}%
}d\overline{w}\right)  \right) \\
&  =\lambda_{\varphi}\left(  \Phi\left(  -\overline{w}\right)  \right)
\left\vert \Phi^{\prime}\left(  -\overline{w}\right)  \right\vert
^{2}idw\wedge d\overline{w}=\lambda_{\psi}\left(  -\overline{w}\right)
idw\wedge d\overline{w},
\end{align*}
ce qui prouve la coh\'{e}rence de la d\'{e}finition de $\mu_{-1}$.

Les formes $\mu_{1}$ et $\mu_{-1}$ se recollent contin\^{u}ment le long de
$\gamma$ en une forme volume $\widehat{\mu}$ pour $\widehat{M}$. En effet,
consid\'{e}rons une carte de bord $\varphi:U\rightarrow\mathbb{C}$ pour
$\overline{M}$ et la carte $\varphi_{b}:U_{b}\rightarrow\mathbb{C}$
d\'{e}finie comme plus haut. Posons $\varphi_{\ast}\mu=\lambda_{\varphi
}idz\wedge d\overline{z}$ . Lorsque $s\in U$, $\varphi_{b}\left(  s,-1\right)
=\overline{\varphi\left(  s\right)  }$ et $\varphi_{-1}\left(  s,-1\right)
=-\overline{\varphi\left(  s\right)  }$. Donc le changement de carte de
$\varphi_{b}$ \`{a} $\varphi_{-1}$ est l'application de $\overline{U}$ dans
$-\overline{U}$, $z\mapsto-z$. D'o\`{u}
\[
\left(  \left(  \varphi_{b}\right)  _{\ast}\mu_{-1}\right)  _{z}%
=\lambda_{\varphi}\left(  \overline{z}\right)  idz\wedge d\overline{z}=\left(
\varphi_{1\ast}\mu_{1}\right)  _{\overline{z}}%
\]
pour tout $z\in\mathbb{D}^{-}\cup\left]  -1,1\right[  $ o\`{u} $\mathbb{D}%
^{-}=\mathbb{D}\cap\left\{  \operatorname{Im}>0\right\}  $. Etant donn\'{e}
que $\varphi\left(  b_{U}M\right)  =\left]  -1,1\right[  $, ceci montre que
$\mu_{-1}=\mu_{1}$ en chaque point de $\gamma\cap U$. Ecrivons au voisinage
dans $\mathbb{D}^{+}\cup\left]  -1,1\right[  $ la fonction $\lambda_{\varphi}$
sous la forme $\lambda_{\varphi,0}\left(  x\right)  +\lambda_{\varphi
,1}\left(  x\right)  y+\lambda_{\varphi,2}\left(  x\right)  y^{2}+o\left(
y^{2}\right)  $. Puisque $\mu$ est tangente \`{a} $bM$ par hypoth\`{e}se,
$0=\lambda_{\varphi,1}$ sur $bM$ et il appara\^{\i}t que $\widehat{\mu}$ est
de classe $C^{2}$.

On peut maintenant munir $\Lambda^{1,0}T_{p}^{\ast}\widehat{M}$,
$p\in\widehat{M}$, de la m\'{e}trique $\widehat{h_{p}^{\ast}}$ d\'{e}finie par%
\[
\widehat{h_{p}^{\ast}}\left(  \alpha\right)  =\left(  \frac{\alpha\wedge
\ast\overline{\alpha}}{\widehat{\mu}_{p}}\right)  ^{1/2}%
\]
pour tout $\alpha\in\Lambda^{1,0}T_{p}^{\ast}\widehat{M}$. La connexion de
Chern $D$ qui lui est associ\'{e}e est donc de classe $C^{2}$ sur
$\widehat{M}$ car $\widehat{h^{\ast}}$ et la formule de Stokes livre que
lorsque $\omega$ est une $\left(  1,0\right)  $-forme m\'{e}romorphe sur $M$
sans p\^{o}le ni z\'{e}ro sur $bM$,%
\begin{equation}
\frac{1}{2\pi i}\int_{\partial M}\frac{D\omega}{\omega}=\frac{1}{2\pi i}%
\int_{\partial M_{1}}\frac{D\omega}{\omega}=N_{z}\left(  \omega\right)
-N_{p}\left(  \omega\right)  -\frac{1}{2\pi}\int_{M_{1}}i\widehat{\Theta}
\label{F/ genreS1}%
\end{equation}
o\`{u} $\widehat{\Theta}$ est la courbure de $D$. Si on admet que $\frac
{1}{2\pi}\int_{M_{1}}i\widehat{\Theta}=\frac{1}{2\pi}\int_{M_{-1}%
}i\widehat{\Theta}$, (\ref{F/ genreS}) r\'{e}sulte des formules
(\ref{F/ genreS1}) et (\ref{F/ gdouble}) car alors $\frac{1}{2\pi}\int_{M_{1}%
}i\widehat{\Theta}=\frac{1}{2}\frac{1}{2\pi}\int_{\widehat{M}}i\widehat{\Theta
}=g\left(  \widehat{M}\right)  -1$ puisque $\widehat{M}$ est compacte et $D$
de classe $C^{2}$.

Notons $j$ la sym\'{e}trie naturelle de $\widehat{M}$ par rapport \`{a}
$\gamma$ et $c$ la conjugaison de $\mathbb{C}$. Lorsque $\varphi
:U\rightarrow\mathbb{C}$ est une carte de $M$, l'expression de $j$ dans les
cartes $\varphi_{1}$ et $\varphi_{-1}$ est $\varphi_{-1}\circ j\circ\left(
\varphi_{1}\right)  ^{-1}$ c'est \`{a} dire $-c\left\vert _{U}^{\overline{U}%
}\right.  $. $j$ \'{e}change donc les orientations de $M_{1}$ et $M_{-1}$ ce
qui donne%
\[
\int_{M_{1}}\widehat{\Theta}=-\int_{M_{-1}}j^{\ast}\widehat{\Theta}.
\]

Lorsque $\psi:V\rightarrow\mathbb{C}$ est une carte de $\widehat{M}$,
l'application $\widetilde{\psi}:j\left(  V\right)  \rightarrow\mathbb{C}$
d\'{e}finie par $\widetilde{\psi}=\overline{\psi}\circ j$ est aussi une carte
de $\widehat{M}$. Ceci permet (voir \cite{AhL-SaL1960Livre} par exemple) \`{a}
partir d'une section $\omega$ de $\Lambda T^{\ast}\widehat{M}$ sur une partie
$X$ de $\widehat{M}$, de d\'{e}finir une section $\widetilde{\omega}$ de
$\Lambda T^{\ast}\widehat{M}$ sur $j\left(  X\right)  $ en posant pour chaque
carte $\psi:V\rightarrow\mathbb{C}$ de $\widehat{M} $ telle que $V\cap
X\neq\varnothing$, $\left(  \widetilde{\psi}_{\ast}\widetilde{\omega}\right)
_{w}=\beta\left(  \overline{w}\right)  dw+\alpha\left(  \overline{w}\right)
d\overline{w}$ quand $\psi_{\ast}\omega=\alpha dz+\beta d\overline{z}$ et
$\overline{w}\in\psi\left(  V\cap X\right)  $. En particulier, $\omega$
\'{e}tant une section fix\'{e}e de $\Lambda^{1,0}T^{\ast}M$ sans z\'{e}ro sur
$\overline{M}$, holomorphe sur $bM$ et de classe $C^{\infty}$ sur
$\overline{M}$, $\omega_{1}=\pi^{\ast}\omega$ (resp. $\omega_{-1}%
=\overline{\widetilde{\omega_{1}}}$) est une section de $\Lambda^{1,0}T^{\ast
}\widehat{M}$ holomorphe sans z\'{e}ro sur $\overline{X}_{1}$ (resp.
$\overline{X_{-1}}$), holomorphe sur $X_{1}$ et de classe $C^{\infty}$ sur
(resp. $\overline{X_{-1}}$). Posant $f_{\nu}=\ln\widehat{h}\left(  \omega
_{\nu}\right)  ^{2}$, on sait alors que%
\[
\widehat{\Theta}\left\vert _{M_{\nu}}\right.  =d\partial f_{\nu},~\nu=\pm1.
\]
Fixons une carte $\varphi:U\rightarrow\mathbb{C}$ et posons $\varphi_{\ast
}\omega=\alpha dz$. Alors $\left(  \varphi_{1}\right)  _{\ast}\omega
_{1}=\alpha dz$ et $\left(  \widetilde{\varphi_{1}}\right)  _{\ast}\omega
_{-1}=\overline{\alpha\left(  \overline{w}\right)  }dw$. Puisque $\ast$ agit
sur les $\left(  0,1\right)  $-formes comme la multiplication par $\frac{i}%
{2}$, on obtient que
\begin{align*}
\left(  \widetilde{\varphi_{1}}\right)  _{\ast}\left(  \omega_{-1}\wedge
\ast\overline{\omega_{-1}}\right)   &  =\overline{\alpha\left(  \overline
{w}\right)  }dw\wedge\frac{i}{2}\alpha\left(  \overline{w}\right)
d\overline{w}\\
&  =\left\vert \alpha\left(  \overline{w}\right)  \right\vert ^{2}\frac{i}%
{2}dw\wedge d\overline{w}%
\end{align*}
Posons $\mu=\lambda_{\varphi}\frac{i}{2}dz\wedge d\overline{z}$. L'expression
de $\mu_{-1}$ dans la carte $\varphi_{-1}$ est par d\'{e}finition
$\varphi_{-1\ast}\mu_{-1}=\lambda_{\varphi}\left(  -\overline{z}\right)
\frac{i}{2}dz\wedge d\overline{z}$. $\widetilde{\varphi_{1}}$ est aussi une
carte d\'{e}finie sur $j\left(  U_{1}\right)  =U_{-1}$ et le changement de
carte de $\widetilde{\varphi_{1}}$ \`{a} $\varphi_{-1}$ est l'application
$\Phi$ qui \`{a} $w\in\widetilde{\varphi_{1}}\left(  U_{-1}\right)
=\overline{U}$ associe le nombre $\Phi\left(  w\right)  $ d\'{e}fini par%
\begin{align*}
\Phi\left(  w\right)   &  =\widetilde{\varphi_{1}}\left(  \left(  \varphi
_{-1}\right)  ^{-1}\left(  w\right)  \right)  =\left(  \overline{\varphi_{1}%
}\circ j\right)  \left(  \varphi^{-1}\left(  -\overline{w}\right)  ,-1\right)
\\
&  =\overline{\varphi_{1}\left(  \varphi^{-1}\left(  -\overline{w}\right)
,1\right)  }=\overline{\varphi\left(  \varphi^{-1}\left(  -\overline
{w}\right)  \right)  }=-w.
\end{align*}
Par cons\'{e}quent, pour $w\in\mathbb{D}^{-}\cup\left[  -1,1\right]  $,%
\begin{align*}
\left(  \left(  \widetilde{\varphi_{1}}\right)  _{\ast}\mu_{-1}\right)  _{w}
&  =\left(  \left(  \widetilde{\varphi_{1}}\right)  {}^{-1}\right)  ^{\ast
}\varphi_{-1}^{\ast}\varphi_{-1\ast}\mu_{-1}=\left(  \varphi_{-1}\circ\left(
\widetilde{\varphi_{1}}\right)  {}^{-1}\right)  ^{\ast}\varphi_{-1\ast}%
\mu_{-1}\\
&  =\left(  \Phi^{-1}\right)  ^{\ast}\varphi_{-1\ast}\mu_{-1}=\left(
\Phi^{-1}\right)  ^{\ast}\left(  \lambda_{\varphi}\left(  -\overline
{z}\right)  \frac{i}{2}dz\wedge d\overline{z}\right) \\
&  =\lambda_{\varphi}\left(  \overline{w}\right)  \frac{i}{2}dw\wedge
d\overline{w}=\left(  \varphi_{1\ast}\mu_{1}\right)  _{\overline{w}}%
\end{align*}
et donc
\begin{align*}
\left(  \left(  \widetilde{\varphi_{1}}\right)  _{\ast}\widehat{h}\left(
\omega_{-1}\right)  \right)  \left(  w\right)   &  =\frac{\left(
\widetilde{\varphi_{1}}\right)  _{\ast}\left(  \omega_{-1}\wedge\ast
\overline{\omega_{-1}}\right)  }{\varphi_{-1\ast}\mu_{-1}}\left(  w\right)
=\frac{\left\vert \alpha\left(  \overline{w}\right)  \right\vert ^{2}}%
{\lambda\left(  \overline{w}\right)  }\\
&  =\left(  \varphi_{1}\right)  _{\ast}\left(  \widehat{h}\left(  \omega
_{1}\right)  \right)  \left(  \overline{w}\right)
\end{align*}
On en d\'{e}duit que $\widehat{h}\left(  \omega_{-1}\right)  \circ
\widetilde{\varphi_{1}}^{-1}=\widehat{h}\left(  \omega_{1}\right)
\circ\left(  \varphi_{1}\right)  ^{-1}\circ c$ et donc $\left(
\widetilde{\varphi_{1}}\right)  _{\ast}f_{-1}=\left(  \varphi_{1}\right)
_{\ast}f_{1}\circ c$ (ce qui, au passage, livre $f_{-1}=f_{1}\circ
j$).D\'{e}rivant deux fois cette relation et en utilisant que $d\overline
{\partial}=-d\partial$, on obtient finalement que $j^{\ast}\widehat{\Theta
}=-\widehat{\Theta}$ et donc que $\int_{M_{1}}\widehat{\Theta}=\int_{M_{-1}%
}\widehat{\Theta}$, ce qui ach\`{e}ve la preuve du th\'{e}or\`{e}me.
\end{proof}

Le th\'{e}or\`{e}me~\ref{T/ ZPgenre} contient le corollaire.

\begin{corollary}
\label{C/ genre via dp(u)}On se donne une surface de Riemann \`{a} bord $M$ et
une fonction harmonique $u$ non localement constante. On d\'{e}signe par $q$
le nombre de z\'{e}ros de $\partial u$ compt\'{e}s avec multiplicit\'{e} et
par $c$ le nombre de composantes connexes de $bM$. On \'{e}quipe $M$ d'une
forme volume $\mu$ telle que $j_{bS}^{\ast}\partial\mu$ est exacte puis on
construit pour $\Lambda^{1,0}T^{\ast}M$ une m\'{e}trique hermitienne $h^{\ast
}$ et une connexion $D$ comme dans le th\'{e}or\`{e}me~\ref{T/ ZPgenre}.
Alors
\[
\frac{1}{2\pi i}\int_{\partial S}\frac{D\partial u}{\partial u}=q+2-2g\left(
M\right)  -c.
\]

\end{corollary}

\begin{proof}
C'est un cas particulier du th\'{e}or\`{e}me~\ref{T/ ZPgenre}.
\end{proof}

Les formules de cette section n\'{e}cessitent de pouvoir calculer des
connexions de Chern le long du bord d'une surface de Riemann \`{a} bord. La
proposition ci-dessous explique comment ceci est possible quand on
conna\^{\i}t l'op\'{e}rateur de Dirichlet-Neumann.

\begin{lemma}
\label{L/ D|bS}On se donne une surface de Riemann \`{a} bord $M$.
L'op\'{e}rateur de Dirichlet-Neumann $N$ de $M$ et la structure complexe de
$M$ le long de $bM$ d\'{e}terminent l'action sur $\Lambda^{1,0}T_{bM}^{\ast}M$
d'une connexion associ\'{e}e \`{a} une m\'{e}trique tangente \`{a} $bM$.
\end{lemma}

\begin{proof}
L'op\'{e}rateur $\ast$ de conjugaison associ\'{e} \`{a} la structure complexe
de $M$ est par hypoth\`{e}se connu quand il agit sur le fibr\'{e} r\'{e}el
$T_{bM}^{\ast}\overline{M}=\underset{s\in bM}{\cup}T_{s}\overline{M}$. $bM$
\'{e}tant connu, on peut se donner une section lisse et g\'{e}n\'{e}ratrice
$\tau^{\ast}$ de $T^{\ast}bM$. Pour tout $s\in bM$, on pose alors $\nu
_{s}^{\ast}=-\ast_{s}\tau_{s}^{\ast}$. Par d\'{e}finition d'une
conductivit\'{e}, $\mu_{0}:bM\ni s\mapsto\nu_{s}^{\ast}\wedge\tau_{s}^{\ast}$
est alors une section du fibr\'{e} des formes volumes de $\overline{M}$. On se
donne alors une forme volume $\mu$ sur $\overline{M}$ qui est tangentielle
\`{a} $bM$ et telle que $\mu\left\vert _{M}\right.  =\mu_{0}$ puis On munit
$T^{\ast}\overline{M}$ de la m\'{e}trique $h^{\ast}$ d\'{e}finie pour $s\in M$
et $\alpha\in T_{s}\overline{M}$ par%
\[
h_{s}^{\ast}\left(  \alpha\right)  =\left(  \frac{\alpha\wedge\ast_{s}\alpha
}{\mu_{s}}\right)  ^{1/2}%
\]
Lorsque $u$ est une fonction r\'{e}elle de classe $C^{\infty}$ sur
$\overline{M}$ et harmonique sur $M$, telle que $\partial u$ ne s'annule pas
sur $bM$, il existe un voisinage $\Sigma$ de $bM$ dans $\overline{M}$ tel que
$\partial u$ ne s'annule pas dans $\Sigma$, est de classe $C^{\infty}$ sur
$\Sigma$ et holomorphe sur $\Sigma\backslash M$. On sait alors que pour tout
$s\in bM$%
\[
\frac{\left(  D\partial u\right)  _{s}}{\left(  \partial u\right)  _{s}%
}=\left(  \partial\ln h^{\ast}\left(  \partial u\right)  ^{2}\right)
_{s}=\partial\ln\left(  \frac{\partial u\wedge\ast\overline{\partial u}}%
{\nu_{s}^{\ast}\wedge\tau_{s}^{\ast}}\right)  ,
\]
ce qui prouve le lemme.
\end{proof}

{\normalsize
\bibliographystyle{amsperso}
\bibliography{ref}
}%

\end{spacing}%

\end{document}